\documentclass{amsart}

\RequirePackage{amsthm,amsmath,amsfonts,amssymb,dsfont,color}
\RequirePackage[colorlinks,citecolor=blue,urlcolor=blue]{hyperref}%
\RequirePackage{graphicx}%
\RequirePackage{stmaryrd}

\usepackage{tikz}
\usepackage[labelsep=quad,indention=10pt]{caption}
\usepackage[labelfont=bf,list=true]{subcaption}
\usepackage{bbm}

\usepackage{a4wide}
\usepackage[foot]{amsaddr}

\newcommand{\indep}{\perp \!\!\! \perp}

\newcommand{\R}{\mathbb{R}}
\newcommand{\eS}{\mathbb{S}}

\newcommand{\eP}{\mathbb{P}}
\newcommand{\X}{\mathbb{X}}
\newcommand{\E}{\mathbb{E}}

\newcommand{\iP}{\mathcal{P}}
\newcommand{\fp}{f_{\varphi}}
\newcommand{\iX}{\mathcal{X}}

\newcommand{\lamb}{\lambda^{\mathbf{m}}}
\newcommand{\bX}{X}

\newcommand{\bC}{\mathbf{C}}
\newcommand{\bT}{\mathbf{T}}
\newcommand{\q}{\mathbf{q}}

\newcommand{\fe}{\mathsf{le}}
\newcommand{\ch}{\mathsf{ch}}

\newcommand{\ind}{\mathbbm{1}}

\newcommand{\tw}{\widetilde{w}}

\newcommand{\Du}{\Delta^{(1)}}
\newcommand{\Dd}{\Delta^{(2)}}
\newcommand{\Dt}{\Delta^{(3)}}
\newcommand{\Dq}{\Delta^{(4)}}
\newcommand{\Di}{\Delta^{(i)}}
\newcommand{\Dj}{\Delta^{(j)}}

\newcommand{\Ez}{\mathbb{E}_{\indep}}
\newcommand{\Eo}{\mathbb{E}}

\newcommand{\V}{\mathbb{V}}
\newcommand{\Vz}{\mathbb{V}_{\indep}}
\newcommand{\Vo}{\mathbb{V}}

\newcommand{\Pz}{\mathbb{P}_{\indep}}
\newcommand{\Po}{\mathbb{P}}

\newcommand{\dd}{\mathop{}\!\mathrm{d}}

\theoremstyle{plain}
\newtheorem{theorem}{Theorem}[section]
\newtheorem{prop}[theorem]{Proposition}
\newtheorem{lemma}[theorem]{Lemma}

\newtheorem{model}{Model}

\theoremstyle{remark}

\newtheorem{remark}[theorem]{Remark}

\title[Synchronization detection]{Separation rates for the detection of synchronization of interacting point processes in a mean field frame. Application to neuroscience. }

\author{Josu\'e Tchouanti$^1$}
\address{$^1$Neuromod Institute, Universit\'e C\^ote-d'Azur}
\email{josue.tchouanti\_fotso@math.univ-toulouse.fr}

\author{\'Eva L\"ocherbach$^2$}
\address{$^2$
CMAP, Ecole Polytechnique,  Institut Polytechnique de Paris,  Palaiseau}
\email{eva.loecherbach@polytechnique.edu}

\author{Patricia Reynaud-Bouret$^3$}
\address{$^3$Universit\'e C\^ote d'Azur, CNRS, LJAD, France}
\email{Patricia.Reynaud-Bouret@univ-cotedazur.fr}

\author{Etienne Tanr\'e$^4$}
\address{$^4$Université Côte d'Azur, Inria, CNRS, LJAD, France}
\email{Etienne.Tanre@inria.fr}

\begin{document}

\begin{abstract}
Permutation tests have been proposed by Albert et al. (2015) to detect dependence between point processes, modeling in particular spike trains, that is the time occurrences of action potentials emitted by neurons.  Our present work focuses on exhibiting a criterion on the separation rate to ensure that the Type II errors of these tests are controlled non asymptotically. This criterion is then discussed in two major models in neuroscience: the jittering Poisson model  and  Hawkes processes having \(M\) components interacting in a mean field frame and evolving in stationary regime. For both models,  we obtain a lower bound of the size \(n\) of the sample
necessary to detect the dependency between two neurons.
\end{abstract}

\maketitle


\medskip

\noindent\textbf{MSC2020 subject classification: }62G10, 62M07,  62F40

\section{Introduction}
Neurons in the brain form a vast and dense  directed network, in which each vertex (the neurons) emits action potentials that excite or inhibit downstream vertices. When recordings are performed, neurobiologists are usually able to track the spikes of a few neurons, that is, the times at which a neuron emits an action potential. The set of spikes of a given neuron is commonly called a \textit{spike train}. The neurons that are recorded in a given region often show synchronization during specific periods depending on the cognitive process at play, even if they might be not physically directly connected. Here, \textit{synchronization} means that the spikes of these neurons seem to appear "at the same time" during a given period. This phenomenon varies over time and usually marks cognitive events of importance, so that  synchronization is usually thought a key element to understand the neural code, that is the way the brain encodes information \cite{singer, riehle}.  

Let us formalize the problem in the case of two neurons emitting two spike trains, modeled by two point processes $X^1$ and $X^2$. We reformulate the neuroscience problem of synchronization detection as a testing problem where one wants to reject the null hypothesis "$X^1$ is independent from $X^2$" by counting the number of coincidences. Here by number of coincidences we mean the number of couples of spikes (one per neuron) that appear almost at the same time. Then the test rejects the null hypothesis if the number of coincidences is too large.
Despite the fact that the synchronization is the result of the physical connectivity hidden behind, neurobiologists usually separate completely the problem of detecting synchronization from the problem of understanding how the network produces such phenomena. The above test is therefore usually applied by computing a reference distribution under the null hypothesis, which is done either by assuming that the spike trains %
have a known distribution (Poisson \cite{tuleau,pipa2013impact,grun1999detecting} or Integrate-and-Fire \cite{tamborrino2012})
or by a bootstrap distribution \cite{shuffling,mel2015,mel2016}. 
In particular, it has been proved that permutation tests (based on permutation of i.i.d. trials) lead to very  good results in practice because of their non asymptotic properties \cite{mel2015}.

In the present work, we investigate the non asymptotic separation rate of these permutation tests, focusing on two main examples in neuroscience. Our main goal is to understand what the fact that the test rejects the null hypothesis and detects synchronization reveals on the underlying link between the processes.  

The first model we consider  is the injection Poisson model which is a classic model of synchronization  \cite{date,grun1999detecting}. 
For this model, we are in particular interested in the precise relationship between the minimal number of trials and the injection strength which are necessary to detect synchronization.

The second model is a homogeneous network of Hawkes processes. This case is of particular relevance since it is the first time that one tries to mathematically link the shape of the hidden network with the synchronization detection. Here we envision the case of a homogeneous network of exponential Hawkes processes of size $M$. The integer $M$ represents the size of the relevant network at play, the two recorded neurons being just two neurons embedded in this network. We study more specifically the particular case of an exponential memory function. 

A classical result by Delattre et al.~\cite{del} states that such a network has a mean-field limit when $M$ tends to infinity, meaning that in the mean field limit, the two recorded neurons are  independent, and therefore the corresponding spike trains satisfy the null hypothesis. Our goal is to precisely establish the link between $M$ and the minimal number of trials $n$, to be able to detect synchronization. 
We show that $n$ has to be larger than $M^2$ to detect the dependence between the two neurons. The implication for neuroscience is that if $M$ is too large then one will never detect synchronization, despite the fact that the neurons are physically connected, because we will never have enough trials to do so. In this sense,  we can do as if the neurons were independent, despite the fact that they are embedded in a common network.

Separation rates for general independence tests based on permutation have been recently investigated by Kim et al.~\cite{kim} (see also \cite{mel2019} for the case of point processes). We derive from \cite{kim} a general bound with precise variances that improves the one of \cite{mel2019}. The main purpose of the remainder of the paper is to interpret these bounds in terms of our model parameters, within the two main examples we have in mind, coming from neuroscience.

If the injection Poisson model has been rarely investigated in the mathematical literature from a probabilistic point of view, probably because of its simplicity, stochastic systems of interacting particles in a mean field frame and their associated non-linear limit processes have been extensively studied 
during the last three or four decades. To build a statistical theory within this frame is the goal of a relatively new research direction (with the exception of 
 \cite{kas}) that has been mainly explored with regards to 
estimation theory, and most often in the framework of interacting diffusion models. One can cite in particular the work of Della Maestra and Hoffmann \cite{DH1,DH2}, which is devoted to an asymptotic estimation theory in the large population setting  when $N$ interacting diffusive particles evolve in a mean field frame and are observed continuously in time over a fixed time interval $[0, T ], $ both in a parametric and in a non-parametric setting. See also \cite{amorino2023parameter} who deals with the parameter estimation of both drift and diffusion coefficient for discretely 
observed interacting diffusions, in the asymptotic where the step size tends to 0 and the number of particles to infinity.  In \cite{semi_param_McKV}, the authors have conducted a related study, yet in a semi parametric setting. Finally, in a series of two papers, Genon-Catalot and Laredo deal with the estimation problem in the limit model of a non-linear in the sense of McKean-Vlasov diffusion model, in the small noise regime \cite{gcatalot_laredo1, gcatalot_laredo2}. 

The statistical theory of systems of interacting point processes in a mean-field frame is less developed, but we can mention  \cite{mf_inference_Bacry} devoted to the estimation of the underlying parameters in systems of interacting Hawkes processes, as well as \cite{mf_inference_Delattre} which proposes a study in a very particular framework where $N$ point processes are observed during a fixed time interval,  interacting on an Erd\"os-Renyi random graph, and one wants to recover the underlying parameters of the linear Hawkes process as well as the connection probability.
To the best of our knowledge no test theory has been developed in either of the two frameworks (interacting diffusions or interacting point processes). Injection models in testing theory have been investigated slightly in \cite{mel2015,mel2016,tuleau} but mainly from a level (Type I error) point of view.

The goal of our paper is to understand in both situations (injection Poisson model and Hawkes processes in a mean-field frame) the \textit{separation rate}. This separation rate is a measure of the distance at which a certain distribution, $\mathcal{P}$ -- the joint distribution of the two spike trains $ X^1 $ and $X^2 $ -- has to be from the null hypothesis to be rejected by the test with large probability. The ``measure'' we use for this purpose is the difference between the  expectation of the coincidence counts under $\mathcal{P}$ and the expectation under $\mathcal{P}_1\otimes\mathcal{P}_2 $, where $\mathcal{P}_1$ and $\mathcal{P}_2 $ are the marginals of $\mathcal{P}$. Bounds on these separation rates can in turn be interpreted as a minimal number of trials to detect a given dependence.

Our approach to obtain the separation rates is non asymptotic, in the line of \cite{BarLauH1,BarLauH2,fromLau,from,from2}. 
The rates that we are getting are parametric in $n$, but the interesting feature is their dependence on the other important parameters of the problem (like $M$, the size of the hidden homogeneous network, or the strength of injection), and this is why the non asymptotic approach on the separation rates is important.

We treat the Poisson injection model in the most general way. However, concerning homogeneous networks of spike trains, we consider only a very particular case (namely exponential Hawkes processes). This is due to the fact that the general separation bounds we derive involve the computation of variances of U-statistics based on coincidences counts under $\mathcal{P}$, 
which involves the computation of cumulants of order 4 for the corresponding Hawkes process. At this point, we leveraged the result of Jovanovi\'c et al.~\cite{rotter} which gives an algorithmic way to compute these quantities. Related results are due to Privault~\cite{priv} but in a more restricted framework, that can not be applied in our case. 
Most importantly, we show that all cumulants are of size $1/M,$ and we study 
how these cumulants depend on the parameters of the exponential kernel and the window size of the coincidence count. We also give some more general results on cumulants, that might have a general interest per se: they are ready to use formulae for the cumulants of Hawkes processes, that are an application of the result in \cite{rotter}.

The rest of the article is structured as follows. In Section~\ref{sec:mainresults}, we formulate our main results, that is the separation rates of the test in both settings (injection Poisson and homogeneous exponential Hawkes processes). In Section~\ref{sec:simulations}, we illustrate the sharpness of our bounds on simulations. In Section~\ref{sec:generalframework}, 
we sketch some proofs and provide the ready-to-use formulae on Hawkes cumulants. Full proofs containing all computations are postponed to the appendix.

    \section{Main results}\label{sec:mainresults}

        \subsection{Presentation of the problem}
        Let us start by formulating the problem of interest. We denote by $\iX$ the set of finite increasing sequences of spiking times \(0\leq \tau_1 < \cdots < \tau_\ell \leq T\) for $\ell\geq 0$ on a fixed time interval \([0,T]\). A point process $X$ has random values in 
        $\iX$, and its associated point measure is denoted $dX_t$. We consider a probability distribution $\iP$ on $\iX^2$ and denote by $\iP^1$ and $\iP^2$ its first and second marginals respectively. Given \textit{i.i.d.}
        observations $\X_n:=\left\{ X_1=(X^1_1,X^2_1),\cdots ,X_n=(X_n^1,X_n^2) \right\}$ from the same distribution $\iP$, we are interested in the test of independence
        \[ 
        (H_0):\iP = \iP^1\otimes\iP^2 ~\textrm{ versus }~(H_1):\iP\neq\iP^1\otimes\iP^2. 
        \]
        In terms of neuroscience, a typical observation $X_i=(X^1_i,X^2_i)$ corresponds 
        to the spike trains of two neurons, and a natural way to construct such a test consists of computing and comparing the number of coincidences 
        under each hypothesis. We count those coincidences thanks to a symmetric function $\varphi(X^1_i,X^2_i)$. In this work and following the ideas 
        of \cite{mel2016, tuleau}, what we actually count is the number of realisations of points of both variables with a delay $\delta>0$, $ \delta < T.$ More precisely, we {\color{black}set}%
        \begin{equation}\label{eq:defvarphi}
            \varphi(X^1_i,X^2_i) = \sum_{\substack{u\in X^1_i,v\in X^2_i}} \ind_{|u-v|\leq\delta} .
        \end{equation}
        In this framework, the test aims to compare the number of observed coincidences with the number of coincidences that corresponds to an independent model, that is, that would be observed under $(H_0)$. It is based on the statistics defined by
        \begin{equation}\label{eq:defCnXn}
            \bC_n(\X_n) = \frac{1}{n}\sum_{i=1}^n\varphi(X^1_i,X^2_i)
        \end{equation}
        which corresponds to the empirical mean number of coincidences in the sample. 
        The test works well whenever the distributions of this statistics under $(H_1)$ 
        and $(H_0)$ are well separated. 
        We denote by $\Po$ (and $\Eo$) the distribution (and  expectation) with respect to  $\X_n$ and 
        $\Pz$ (and $\Ez$)  the distribution (and  expectation) with respect to an \textit{i.i.d.}
        sample of size $n$ with similar marginals as $\X_n$ but with independent coordinates. Since $\Pz$, that is the distribution under $(H_0)$, is unknown in practice, we use a permutation 
        approach.
        
        \subsection{The permutation test}\label{sec:perm_test}
        Let us now introduce the permutation test and some of its most important properties.

        Let $\eS_n$ be the set of permutations of the integers 
        $\{1,\cdots,n\}$. Let \(\X_n\) be a sample and  $\pi\in\eS_n$. The permuted sample $\X^{\pi}_n$ is  
        obtained by permutation of the second spike train  in \(\X_n\) according to $\pi$. 
        Precisely, we set
        \begin{equation}
            \X^{\pi}_n = \left\{ X^{\pi}_1 = (X^1_1,X^2_{\pi(1)}), \cdots, X^{\pi}_n = (X^1_n,X^2_{\pi(n)}) \right\}
        \end{equation}
        and we consider the permuted statistics $\bC_n(\X^{\pi}_n)$. 
        
        Given the sample $\X_n$, we denote by $q_{1-\alpha}(\X_n)$ the $(1-\alpha)$-quantile of the uniform distribution $\frac{1}{n!}\sum_{\pi\in\eS_n}\delta_{\bC_n(\X^{\pi}_n)}.$ Then the permutation test 
        \[
        \ind_{\bC_n(\X_n)>q_{1-\alpha}(\X_n)}
        \]
        which rejects $(H_0)$ when the coincidence count is too large is exactly of level $\alpha$.  This follows from the fact that 
        under $(H_0),$ for any permutation \(\pi\in\eS_n\), the samples $\X^{\pi}_n$ and $\X_n$ have the same distribution. We refer the reader to \cite{lehrom, pesal}.  
        
        Now let us turn towards the control of the Type II error. The informal idea behind the permutation test is that, even if we are under $(H_1)$, if we permute the second coordinate, we can mimic $\Pz$ by associating $X_i^1$ with $X_j^2$ for two different trials $i$ and $j$, say for instance $i=1$ and $j=2$. One important quantity is therefore 
        \[
        \fp(X_1,X_2) = \varphi(X^1_1,X^2_1)-\varphi(X^1_1,X^2_2).
        \]
        We then introduce
        \begin{equation}\label{eq:defdeltaphi}
        \Delta_\varphi:= \Eo\left[\fp(X_1,X_2)\right] = \Eo\left[\varphi(X^1_1,X^2_1)\right]-\Ez\left[\varphi(X^1_1,X^2_1)\right]
        \end{equation}
        as a measure of how far the current distribution $\Po$ is from $(H_0)$, that is far from $\Pz$. In the sequel the separation rates of the permutation test will therefore be expressed in terms of  $\Delta_\varphi.$ 
        
        Kim et al.~\cite{kim} showed a general bound on the separation rates of independence tests by permutation. 
        By applying their result to our setting, one can show the following general lemma.
        \begin{lemma}\label{lem:critgen}
            Let $\alpha$ and $\beta$ be fixed constants in $(0,1)$. Then there exists an absolute constant $C>0$ such that if the number of trials $n$ satisfies $n\geq 3/\sqrt{\alpha\beta}$ and if 
            \begin{equation}\label{eq:20230127_1}
                \Delta_\varphi  \geq \frac{C}{\sqrt{n\alpha\beta}}\left\{ \sqrt{\Vz\big[ \fp(X_1,X_2) \big]} + \sqrt{\Vo\big[ \fp(X_1,X_2) \big]} \right\},
            \end{equation}
         then the Type II error of the permutation test satisfies
         \[
         \Po(\bC_n(\X_n)\leq q_{1-\alpha}(\X_n))\leq \beta.
         \]
        \end{lemma}
        \begin{remark} In practice, accessing the $n!$ permutations is computationally too intensive. However one can approximate $q_{1-\alpha}(\X_n)$ by a Monte Carlo method. It can be proved that this does not affect the level of the test (see for instance \cite{mel2015} and the references therein). Also by slightly increasing $\beta$, one can show that this does not really affect the previous result either. {\color{black}Indeed by applying Dvoretzky-Kiefer-Wolfowitz inequality \cite{massartDKW}, it is easy to show for instance that, for a fixed $x>0$, with probability larger than $1-2e^{-2x^2}$, the $1-\alpha$ quantile obtained by Monte Carlo over $N_{sim}$ replicates, is upper bounded by the true quantile over all permutations at level $1-\alpha'$ with $\alpha'=\alpha-x/\sqrt{N_{sim}}$. Hence the test with the Monte Carlo quantile accepts $H_0$ even less than the test with the true quantile at level $\alpha'$. By slightly increasing $\beta$ to $\beta'$ such that $\alpha\beta=\alpha'\beta'$, \eqref{eq:20230127_1} will still be true and the conclusion will follow with error of second kind bounded by $\beta'$. Because this holds only with probability 
        larger than $1-2e^{-2x^2}$, one needs to increase slightly more the upper bound on the error of second kind.}
        
        However to keep the writing of the next results as simple as possible, we {\color{black}decided to} not incorporate this case and focus on the ideal test.
        \end{remark}
        
        Let us now interpret the above criterion within our two main examples, the injection Poisson case  and the network of Hawkes processes in a mean-field frame.
        
        \subsection{Jittering injection Poisson model}
        In this section, we are interested in the following model introduced by Date et al.~\cite{date} (see also \cite{grun1999detecting,tuleau}) under which the spike trains are modeled by Poisson processes. 
        \begin{model}\label{mod:jittering}
            Consider three independent Poisson processes \(Z^1\), \(Y^1\) and \(Y^2\) on $\R$ with respective 
            intensities \(\eta^1\), \(\lambda^1\) and \(\lambda^2\).
            \begin{itemize}
                \item The \textit{first} spike train \(X^1\) is a superposition of the points of \(Y^1\) and \(Z^1\).
                \item The \textit{second} spike train \(X^2\) is the superposition of the  points of \(Y^2\) and of a random perturbation \(Z^2\) of \(Z^1\). Precisely, we consider an \textit{i.i.d.} sequence of real valued random variables \((\xi_i)_{i\in\mathbb{Z}}\) independent of anything else. Denote by \( \cdots < u_{k-1} < u_k < u_{k+1} < \cdots\) the points of \(Z^1\), then the points of \(Z^2\) are \(\cdots, u_{k-1}+\xi_{k-1}, u_k + \xi_k, u_{k+1} + \xi_{k+1}, \cdots\).
            \end{itemize}
        \end{model}
        We prove in Proposition~\ref{prop:z2estunpoisson} that $Z^2$ is still a Poisson process of rate $ \eta^1 $ such that the spike trains $X^1$ and $X^2$ correspond to the points of Poisson processes with intensities $\lambda^1+\eta^1$ and  $\lambda^2+\eta^1$ respectively. For the convenience of the reader we give a proof of this fact in the appendix, 
        section~\ref{app:probabilisticresults}. 
        
        The following result provides a lower bound for the separation rate between $(H_0)$ and $(H_1)$, measured by $\Delta_\varphi$. From this we can extract a minimal number of observations $\mathbf{n_{min}}$ to ensure a Type II error less than $\beta$.
        \begin{theorem}\label{theo:jit}
        Let $\alpha,\beta\in (0,1)$, $0\leq \delta\leq T/2$ and 
            assume that the observations obeys 
            Model~\ref{mod:jittering} with parameters $\lambda^1,\lambda^2,\eta^1$ on the time interval $[0,T].$ Then
            \begin{equation}\label{eq:mean_jit}
                \Delta_\varphi = \eta^1 \E\left[\left(T-|\xi|\right)\ind_{|\xi|\leq\delta} \right].
            \end{equation} 
            Moreover there exists a constant $C>0$ not depending on the model parameters, such that if  $n\geq 3/\sqrt{\alpha\beta}$ and if
            \begin{equation}\label{eq:crit_jit}
                \Delta_\varphi \geq C\left[\sqrt{\frac{(1+\lamb\delta)(\lamb)^2\delta T}{n\alpha\beta}} + \frac{1+\lamb\delta}{n\alpha\beta}\right],
            \end{equation}
            with  $\lamb=\max\{\lambda^1,\lambda^2\}$ and  $\eta^1\leq \lamb$, then the Type II error of the permutation test satisfies
            \[
            \Po(\bC_n(\X_n)\leq q_{1-\alpha}(\X_n))\leq \beta.
            \]
            Moreover, there exists a constant $C'>0,$ not depending on the model parameters, such that the previous control of the Type II error is guaranteed as soon as 
            \begin{equation}\label{eq:min_size_jit}
                n\geq \mathbf{n_{min}} := C'\max\left\{ \frac{1}{\sqrt{\alpha\beta}} , \frac{1+\lamb\delta}{\alpha\beta \eta^1 T\eP(|\xi|\leq\delta)}\left( 1 + \frac{\lamb}{\eta^1}\frac{\lamb\delta}{\eP(|\xi|\leq\delta)} \right) \right\}.
            \end{equation}
        \end{theorem}
        
        Note first that {\color{black}the lower bound on } the measure of the distance between $(H_0)$ and $(H_1)$, $\Delta_\varphi$, only depends on the distribution of the perturbation $\xi$ and the frequence of the injected synchronizations $\eta^1$ (see \eqref{eq:mean_jit}). Next the bound \eqref{eq:crit_jit} states that the separation rate diminishes in $n^{-1/2}$ as one would expect in this parametric situation. Moreover we also recover the dependency in $\sqrt{T}$ that one would expect in this parametric Poissonian situation. 
        
       At this point the following remark is at order. Speaking of ``trials'' is a bit artificial here because of the memoryless property of the Poisson processes  which implies that observing $n$ independent copies of one Poisson process during a time interval of length $T$ is equivalent to observing a single trajectory of the same Poisson process during a time interval of length $ n T .$ In other words,  for a global duration of observation $\mathcal{T}$, any choice of $n$ and $T$ such that $\mathcal{T}=nT$ are essentially possible. Since we have renormalized our test statistics $ \bC_n(\X_n)$ by $ 1/n, $ the number of trials, in \eqref{eq:defCnXn}, to take into account of this fact, one has to multiply \eqref{eq:crit_jit} by $n$ in order to obtain on its right hand side indeed a term involving $ \sqrt{n T } = \sqrt{ \mathcal{T}}, $ as expected. We decided to work with the renormalized test statistics as it is commonly done. 
       
       Note also that to fix asymptotics, we work here with $T\geq 2\delta$ to ensure that in the middle of the interval $[0,T]$, we do not have boundary effects due to $\delta$.
        Because of this convention, we have that
        \begin{equation}\label{eq:deltaphilb}
        \eta^1 T \eP(|\xi|\leq\delta) \geq \Delta_\varphi \geq \eta^1 \frac{T}{2} \eP(|\xi|\leq\delta), 
        \end{equation}
        so that $\Delta_\varphi$ essentially behaves like $\eP(|\xi|\leq\delta)$ and increases with $\delta$. The right hand side of \eqref{eq:crit_jit} increases with $\delta$ as well, at least as $\sqrt{\delta}.$ Therefore, depending on the rate of growth of $\eP(|\xi|\leq\delta)$ as a function of  $\delta$, one can have situations where the largest possible choice of $\delta$ is a good choice, or situations where a compromise need to be made. 
        
       We stress that  \eqref{eq:min_size_jit} is essentially a lower bound on $\eta^1nT= \eta^1 \mathcal{T}$, the average total number of injected common points in the data. So if the lower bounds are achieved, we are essentially saying that, for the independence test to detect synchronization,  this number should be larger than
       \[
       \frac{1+\lamb\delta}{\eP(|\xi|\leq\delta)}\left( 1 + \frac{(\lamb)^2\delta}{\eta^1\eP(|\xi|\leq\delta)} \right).
       \]
       In the asymptotic $\delta\to 0$, if there exists $\gamma$ such that $\eP(|\xi|\leq\delta)=_{\delta\to 0} \mathcal{O}(\delta^\gamma),$ we obtain that the minimal number of observations 
       should be at least of order
       \[
        \mathcal{O}(\delta^{-\gamma})+ \frac{(\lamb)^2}{\eta^1} \mathcal{O}(\delta^{1-2\gamma}).
        \]
       Hence
       \begin{itemize}
           \item if $\gamma>1/2$, it decreases with $\delta,$ and choosing $\delta$ large  is a good choice. A typical example is when $\xi$ is uniform on $[-D,D]$, for which the best choice would be $\delta=D$.
           \item if $\gamma<1/2$, an equilibrium in $\delta$ can be found in $\mathcal{O}\left(\dfrac{\eta^1}{(\lamb)^2}\right)^{\frac{1}{1-\gamma}}$.
       \end{itemize}
        Therefore, one sees that, even in this very simple model, the choice of $\delta$ is in fact quite subtle and depends heavily on the distribution of the noise, that is not known in practice.

            \subsection{Hawkes model}
            Introduced for the first time by Hawkes \cite{hawkes}, this model describes a stochastic dynamics on a network 
            that is widely used nowadays to model the  activity of a neuronal network (see also \cite{bremaud1996stability} for multivariate and non linear cases). We envision here a particular case where interactions in the (unobserved) network of size $M$ are homogeneous, 
            while we observe only two neurons embedded in the network.

           \begin{model}\label{mod:Hawkes}
                                Let $\nu, a,b>0$ and $\ell = a/b \in (0,1)$.
                For $i\in\llbracket 1,M\rrbracket$, neuron \(i\) spikes at time \(t\) with intensity
                \begin{equation}
                    \lambda^i(t) = \nu + \frac{a}{M}\sum_{j=1}^M\int_{(-\infty,t)}e^{-b(t-s)}\dd N^j_s,
                \end{equation}
                where \(\dd N^j_t\) is the point measure associated to the spikes of neuron \(j\).  
                In this model, ${\color{black}(X^1,X^2)}$ is the observation of the first two coordinates of this $M$-multivariate Hawkes process in a stationary regime.
            \end{model}
            The hypothesis $\ell < 1 $ and the exchangeability of the network allow to construct the stationary distribution of the previous model as follows. 
            \begin{itemize}
            \item Simulate a univariate stationary Hawkes process $N$ with intensity 
            \begin{equation}\label{eq:intensity}
                \lambda(t) = \nu M + a\int_{(-\infty,t)}e^{-b(t-s)}\dd N_s.
            \end{equation}
            This process is well defined since $\ell= a/b<1$ is the integral of the interaction kernel of the univariate linear Hawkes process (see \cite{HawkesOakes1974, bremaud1996stability}). 
            \item Attribute at random, independently of anything else,  each point $T$ of $N$ to the process $N^j$ of neuron $j$ with  uniform probability on \(\llbracket 1, M\rrbracket\).
            \end{itemize}
            Note then that $\dd N_t = \sum_{j=1}^M \dd N_t^j$. Note also that the univariate process $N$ can be seen as a cluster process (see \cite{HawkesOakes1974}) that is described in more details in Section~\ref{sec:generalframework}. So one can see $N$ as a superposition of clusters of points, each of them arriving according to a Poissonian rate $\nu M$.

            For this model, one can show that
            \begin{theorem}\label{theo:Hawkes}
             Let $\alpha,\beta\in (0,1)$, $0\leq \delta\leq T/2$ and 
            assume that the observations obey 
            Model~\ref{mod:Hawkes} with parameters $\nu, a, b >0, $ and $\ell = a/b \in (0,1)$ on the time interval $[0,T]$. Then
            \begin{equation}\label{eq:mean_Haw}
                \Delta_\varphi = \frac{\nu\delta}{M} \frac{(2-\ell)\ell}{(1-\ell)^3} \left(1+ \left[b(1-\ell)(T-\delta)-1\right]\frac{1-e^{-b(1-\ell)\delta}}{b(1-\ell)\delta}\right).
            \end{equation}    
            
            Moreover there exists a constant $C>0$ not depending on the model parameters, such that if  $n\geq 3/\sqrt{\alpha\beta}$ and if
            \begin{equation}\label{eq:crit_Hawkes}
                \Delta_\varphi \geq C \sqrt{\frac{\delta T}{n\alpha\beta}} \left(\frac{\nu+a/M}{1-\ell}\right) \sqrt{1+\frac{\delta(\nu+a/M)}{1-\ell}\left[ 1 + \frac{\ell}{M(1-\ell)^2} \right] },
            \end{equation}
             then the Type II error of the permutation test satisfies
             \[
             \Po(\bC_n(\X_n)\leq q_{1-\alpha}(\X_n))\leq \beta.
             \]
                If in addition we assume that $b(1-\ell)T>4,$ then there exists a constant $C'>0$ not depending on the model parameters, such that the previous control of the Type II error is guaranteed if  $n\geq 3/\sqrt{\alpha\beta}$ and if 
                \begin{equation}\label{eq:min_size_Hawkes}
                 n \geq \frac{C'}{\alpha\beta} \frac{(M+a/\nu)^2 {(1-\ell)^4}}{\ell^2 T} \left( \frac{{\delta+\delta^2}\frac{(\nu+a/M)}{1-\ell}\left[ 1 + \frac{\ell}{M(1-\ell)^2} \right]}{(1-e^{-b(1-\ell)\delta})^2}\right).
                \end{equation}
            \end{theorem}
            The expression of the distance between $(H_0)$ and $(H_1)$ in \eqref{eq:mean_Haw} is complicated and depends on all the parameters of the model. In particular, it involves the quantity  $(1-e^{-b(1-\ell)\delta})/(b(1 - \ell ) \delta) $ which is the mean total offspring of one single jump -- that is, within one single cluster -- occuring during a time interval of length $ \delta . $ See Section~\ref{sec:generalframework} for more details.

            By looking at \eqref{eq:crit_Hawkes}, we see again, as for the jittering model, a parametric asymptotic in $n^{-1/2}$. We also see that the dependence is, as in the previous Poissonian case,  in $\sqrt{T}$. If the dependency structure of the Hawkes process makes it less straightforward that one can cut the total duration of observation $nT$ in as many trials as one wants, we still have that in a context of large $T$, much larger than the typical length of one cluster of $N$, $T$ is proportional to the number of clusters and the clusters are independent, which provides an asymptotic for the variances of the coincidences count in $\sqrt{T}$ as well.

            Finally let us comment \eqref{eq:min_size_Hawkes}. Let us start by discussing the condition on $b(1-\ell)T>4$: this compares $T$ to the characteristic time 
            $ 1/(b ( 1 - \ell ) ) $ of the cluster. Hence it ensures that one sees a certain number of clusters in one trial, making it possible to see interactions.
            
            Next we see the clear asymptotics in $M^2$ for the number of trials. As announced in the introduction, this means that we need a number of trials larger than the square of the size of the network to detect the dependency. In particular, this leads us to the following interpretation in neuroscience: if we accept the null hypothesis on real data, and if the network is homogeneous as given by Model~\ref{mod:Hawkes}, we cannot distinguish the fact that the neurons are really independent from the fact that they are embedded in a network of order $\sqrt{n}$ or larger, where $n$ is the number of trials.

            The factor $\nu$ also plays an important role. If $\nu$ tends to 0, the right hand side explodes: indeed if the rate of apparition of the clusters vanishes, we do not observe any cluster, hence no interaction.
            
            In the same way, when $\ell=a/b$ tends to $0$ it means that the interaction vanishes and again, the detection problem becomes more difficult to solve (more trials are needed). 
            
            Finally, the dependency in terms of $\delta$ is more complicated. When $\delta$ tends to 0 or infinity, the right hand side explodes, so there is a minimal value of the right hand side in $\delta$ that does not depend on $T$ nor $n$ or $\alpha$ and $\beta$. It seems however that a good choice would be of the order of $(b(1-\ell))^{-1}$, that is the characteristic time of the cluster.
    
    \section{Simulations}\label{sec:simulations}
        In this section, we perform numerical simulations in order to illustrate the separation rates proved for the permutation tests under Models~\ref{mod:jittering} and~\ref{mod:Hawkes}. 
        As explained in Section~\ref{sec:perm_test}, the permutation test is based on estimating the \((1-\alpha)\)-quantile of \(\bC_n(\X^{\pi}_n)\), conditionally to $\X_n$ where \(\X_n\) is the observed sequence of point processes and 
        the permutation \(\pi\) is chosen uniformly. To reduce computational time and following the idea of \cite{mel2015,mel2016}, we use a Monte Carlo method to approximate the bootstrap quantile, with only  \(\pi^1, \cdots, \pi^{B}\) i.i.d. permutations picked uniformly at random.

        In all the sequel, we use $B= 5000$ and $\alpha=0.05$.
    
        Since the level of the test was studied in \cite{mel2015,mel2016}, we do not provide results under the null hypothesis here, but we have checked on simulation that as expected, the permutation tests have a Type I error less than 0.05 in all cases. 
    
        \subsection{The jittering Poisson model}\label{sec:sim_jit}
        We have simulated the jittering Poisson model with $\lambda_1=\lambda_2= 10\,\text{Hz}$ and the injection rate \({\color{black}\eta_1}=1\,\text{Hz}\) on $[0,2]$. The noise $\xi$ is of {\color{black} four} different shapes (i) Uniform on $[0,0.1]$ (ii) of density $200*(0.1-x)$ on $[0,0.1]$ (iii) of density $200 x$ on $[0,0.1]$ {\color{black} (iv) Uniform on $[-0.1,0.1]$}.
        
        Case (i),  (ii) {\color{black}and (iv)} have a density that is lower bounded around 0, whereas case (iii) implies a jitter that is seldom in a vicinity of $0$. The resulting power of the permutation test is plotted in Figure~\ref{fig:Fig_jit}.

           \begin{figure}[ht]
\centerline{
\includegraphics[width=0.9\textwidth]{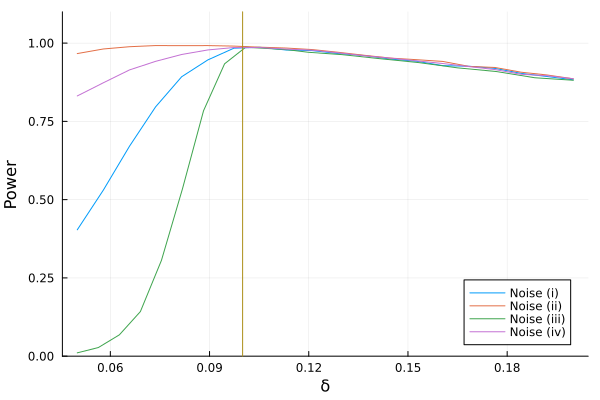}}
\caption{Influence of the noise and $\delta$ on the power of the test. We simulated $n=200$ trials and we replicated this simulation $N_{{\color{black}sim}}=10000$ times to estimate the power of the test. Noises  (i),(ii) (iii) {\color{black}and (iv)} are defined in Section~\ref{sec:sim_jit}. The parameter $\delta$ varies in a regular grid between 0.05 and 0.2. The vertical line indicates the boundary of support for all noises.}
\label{fig:Fig_jit}
\end{figure}
    
    We see that there is a high sensitivity of the test as a function of $\delta$, especially if the jittering noise does not charge a vicinity of $0$ (Noise (iii)). The best choice of $\delta$ seems to depend on the noise, but a fair choice seems to be the one that will encompass most of the support of the jittering noise. After this (beyond 0.10), the power is decreasing slowly, indicating that there is no good reason to choose a large $\delta$ in all cases but that a slight overestimation of the support is not very problematic. The fact that there should be an optimal $\delta$ is also in accordance with the theoretical results of Theorem~\ref{theo:jit}.

{\color{black}    
In Figure~\ref{fig:Fig_jit_etavar}, we simulated the jittering Poisson model with parameters
\(\lambda_1 = \lambda_2 = 10\,\text{Hz}\). 
However, we varied the value of the injection rate, 
choosing \(\eta_1 = 1.2, 1.1, 1.0, 0.9,\) and \(0.8\,\text{Hz}\). We use the noise (iv).
We observe that when the injection rate  is too small, the statistical power of the test decreases noticeably. 
In such cases, a larger sample size \(n\) would be required to achieve a comparable level of power.

This effect is further illustrated in Figure~\ref{fig:Fig_jit_n_var}, where all parameters are kept fixed 
except for \(\delta\) and the number of trials \(n\). 
Under noise (iv) and with parameters 
\(\lambda_1 = \lambda_2 = 10\,\text{Hz}\) and \(\eta_1 = 0.9\,\text{Hz}\), 
we observe that the statistical power increases with the number of trials \(n\).

          \begin{figure}[ht]
\centerline{
\includegraphics[width=0.9\textwidth]{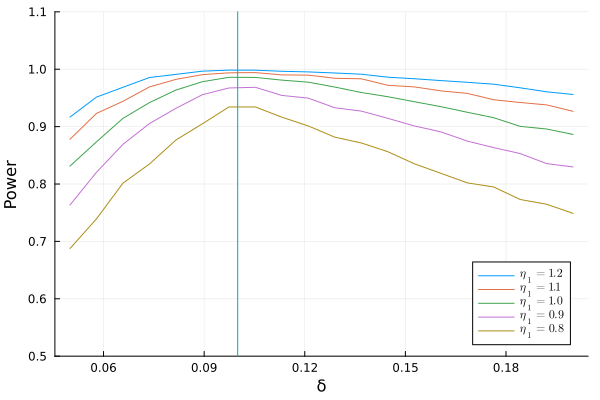}}
\caption{{\color{black} Influence of the injection rate and $\delta$ on the power of the test. We simulated $n=200$ trials and we replicated this simulation $N_{sim}=10000$ times to estimate the power of the test. We use noise (iv), defined in Section~\ref{sec:sim_jit}. The parameter $\delta$ varies in a regular grid between 0.05 and 0.2. The vertical line indicates the boundary of support for all noises.}}
\label{fig:Fig_jit_etavar}
\end{figure}

          \begin{figure}[ht]
\centerline{
\includegraphics[width=0.9\textwidth]{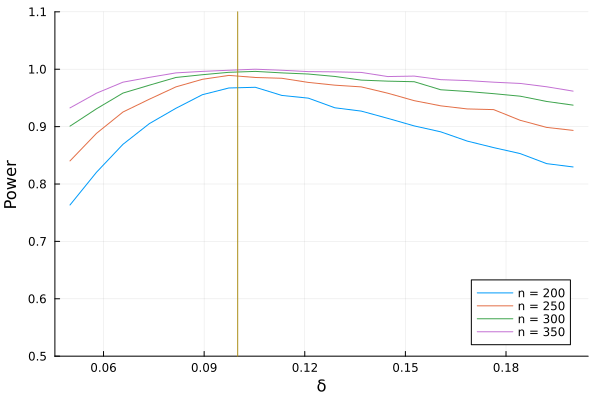}}
\caption{{\color{black} Influence of the number of trial \(n\) and $\delta$ on the power of the test. We simulated $n=200, 250, 300$ and \(350\) trials and we replicated this simulation $N_{{\color{black}sim}}=10000$ times to estimate the power of the test. We use noise (iv), defined in Section~\ref{sec:sim_jit}. The parameter $\delta$ varies in a regular grid between 0.05 and 0.2. The vertical line indicates the boundary of support for all noises.}}
\label{fig:Fig_jit_n_var}
\end{figure}
}
        
        \subsection{The Hawkes model}\label{sec:Haw_sim}
        
        We simulated a homogeneous network of Hawkes processes of size $M$ for various $M$ between 10 and 30 and different $a$ and $b$. We fixed $\nu =1$ and we simulated all trials on $[0,2]$ after a warming up time period of simulation of size \(10\) to reach stationarity.
        
        In Figure~\ref{fig:Fig_Haw1}, we see the evolution {\color{black}of} the power of the test as a function of $n$ the number of trials. As expected the power rapidly deteriorates when $M$ increases. With $M = 30$ and $n = 560$ we do not even reach a power of 80\%.  
        
        To understand if the estimate  of Theorem~\ref{theo:Hawkes} in terms of number of trials necessary to achieve a certain Type II error is valid, at least in term of $M$, we decided to fix the Type II error $\beta=0.05$ and find $n^*$ the minimal number of trials to achieve such a small error. This is represented in Figure~\ref{fig:Fig_Haw2}. We see a very good fit with the superposed curve, showing that indeed the number of necessary trials is growing as $M^2$.

                   \begin{figure}[ht]
\centerline{
\includegraphics[width=0.9\textwidth]{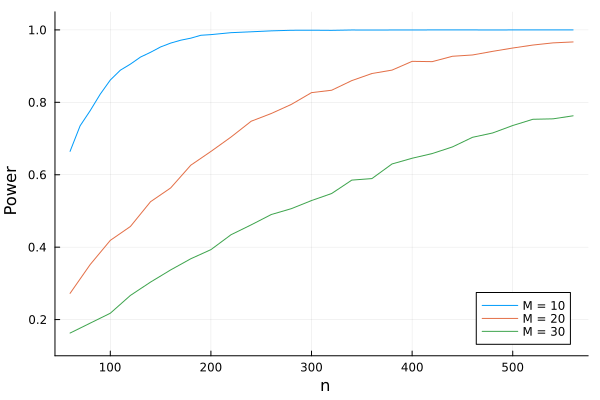}}
\caption{Influence of the number of trials $n$ and of the size of the network $M$.  The number of trials $n$ varied in a regular grid between 60 and 550 and the point processes were the first two coordinates of  an homogeneous network  of size $M = 10, 20, 30$, with $a=3$, $b=4$ and $\delta=0.1$. We simulated  $N_{{\color{black}sim}}=10000$ times the whole procedure to approximate the power of the test. }
\label{fig:Fig_Haw1}
\end{figure}
        \begin{figure}[ht]
\centerline{
\includegraphics[width=0.9\textwidth]{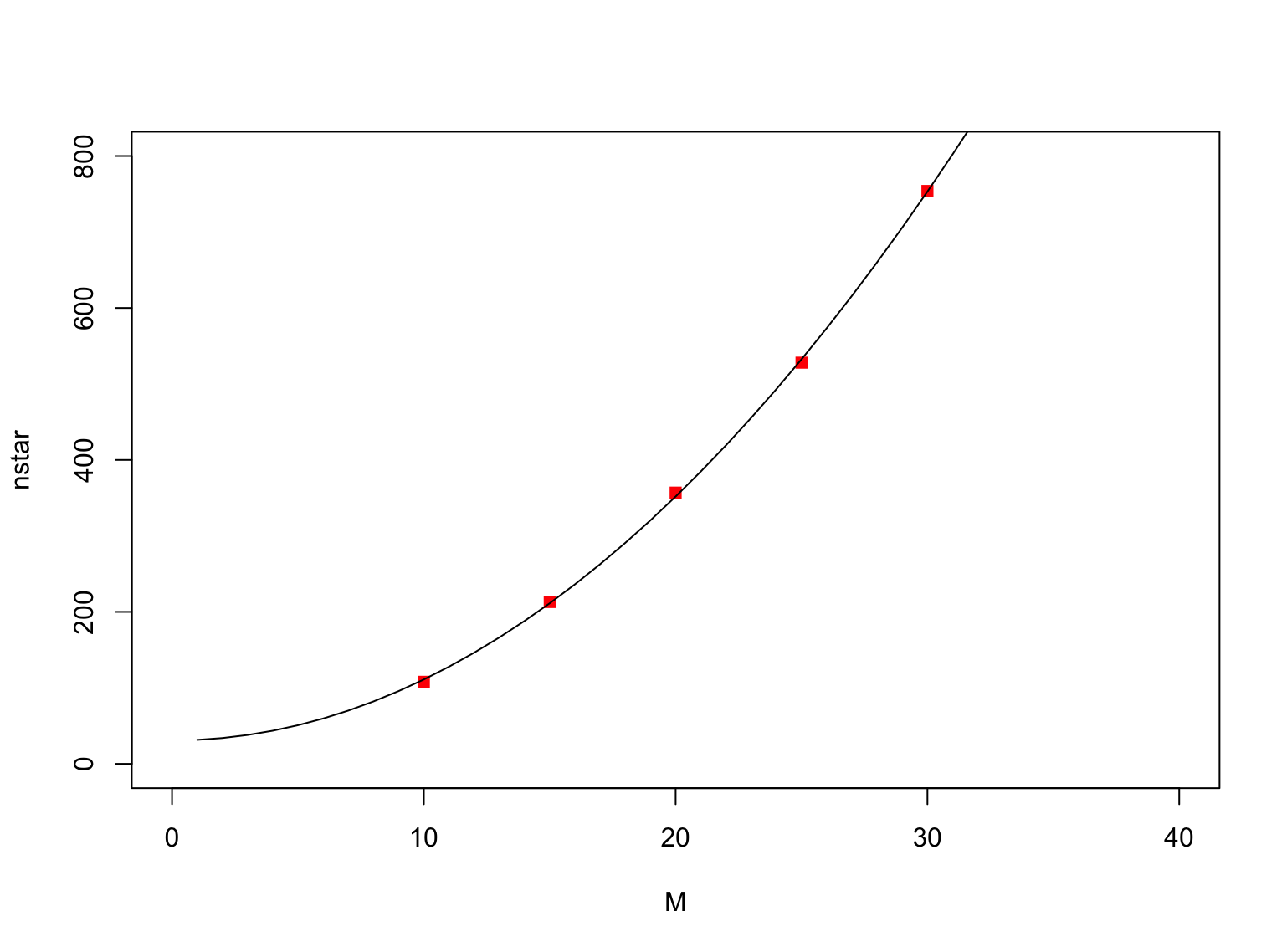}}
\caption{The number of necessary trials grows in $M^2$. We plot with red squares $n^*$ (see Section~\ref{sec:Haw_sim}) as a function of the size of the network $M$. The simulation setting is similar to the one of Figure~\ref{fig:Fig_Haw1} except that $N=10000$,  $a=10$, $b=20$, $\delta=0.1$. We fitted the curve by linear regression with respect to $M^2$ and superposed in black the curve $30 +0.8M^2$.}
\label{fig:Fig_Haw2}
\end{figure}
    \section{Sketch of  proofs and moment formula}\label{sec:generalframework}
    
    The proofs of the above results are rather technical and postponed to the appendix. We give here some ideas of the proofs. Moreover we give exact formulas or bounds on the moments of the coincidence counts in both settings: the injection Poisson model and the homogeneous network of exponential Hawkes process.
        
        \subsection{On another formulation of the test}
        As already discussed in Section~\ref{sec:perm_test}, the permutation test is based on the empirical mean number of coincidences $\bC_n(\X_n)$
        defined in~\eqref{eq:defCnXn}. 
        However the variance of this {\color{black}statistic} can be very large while the correlation between the observed pairs of spike trains is not. This can be due to a large amount of independent spikes for example. 
        Following \cite{mel2015,mel2016}, we consider therefore the U-statistics
        \begin{equation}\label{ustat}
            \bT_n(\X_n) := \frac{1}{n(n-1)}\sum_{i,j=1}^n\fp(X_i,X_j)
        \end{equation}
        where $\fp(X_i,X_j) = \varphi(X_i^1,X_i^2) - \varphi(X_i^1,X_j^2)$. 
        We have
        \begin{equation}\label{eq:mean_ustat}
            \E\left[ \bT_n(\X_n) \right] = \Eo\left[ \varphi(X^1_1,X^2_1) \right] - \Ez\left[ \varphi(X^1_1,X^2_1) \right],
        \end{equation}
        and so \(\bT_n(\X_n)\) is centered under the null hypothesis.
        As noted in \cite{mel2016}, we have 
        \begin{equation}
            \label{ustat_stat}
            \bT_n(\X_n)  = \frac{n}{n-1}\mathbf{C}_n(\X_n) - \frac{1}{n(n-1)}\sum_{i,j=1}^n\varphi(X^1_i,X^2_j).
        \end{equation}
        Let us now denote by $\Pi_n$ a random variable uniformly distributed on the set $\eS_n$ of permutations,  independent of the sample $\X_n$. We deduce from \eqref{ustat_stat} that for the permuted sample, we also have
        \begin{equation}\label{pustat_pstat}
            \bT_n(\X^{\Pi_n}_n) = \frac{n}{n-1}\mathbf{C}_n(\X^{\Pi_n}_n) - \frac{1}{n(n-1)}\sum_{i,j=1}^n\varphi(X^1_i,X^2_j).
        \end{equation}
        It follows that the permutation tests based on $\bC_n(\X_n)$ and $\bT_n(\X_n)$ are exactly the same. In the sequel, we focus on this new formulation, use \cite{kim} and prove Lemma~\ref{lem:critgen}. The computations are quite technical and postponed to the appendix.%
        
        \subsection{Proof of Theorem~\ref{theo:jit}}
            This result follows from Lemma~\ref{lem:critgen} and its proof consists of computing the means and variances that appear in this lemma under Model~\ref{mod:jittering}. Notice that having an upper bound of the variances is enough. The computations use the following intermediate result that can be useful (and even generalized) in other situations if needed.
            \begin{lemma}\label{lem:Poisson}
                Let us consider:
                \begin{itemize}
                    \item A Poisson process $N$ on $\R$ with intensity $\eta^1>0$. 
                    
                    \item An i.i.d family $(\xi_k)_{k\in\mathbb{Z}}$ of real valued random variables that is independent of $N$.
                    
                    \item  For any $ l \geq 2$ the set of vectors with different coordinates
                    \begin{equation}\label{eq:Deltaell}
                        \Delta_{l} = \left\{ (s_1,\cdots,s_{l})\in\R^{l}: s_i\neq s_j, \forall i\neq j \right\}.
                    \end{equation}
                \end{itemize}
                Then for any deterministic and measurable function $F:\R^n\times\R^{l}\to\R$ such that
                \[ \E\left[ \int_{\R^{l}} \big|F(s_1,\cdots,s_{l}, \xi_{1},\cdots,\xi_{l})\big| \dd s_1\cdots \dd s_{l} \right] < \infty, \]
                we have
                 \begin{multline}\label{eq:stoch_int}
                    \E\left[ \int_{\Delta_{l}}F\left(s_1,\cdots, s_{l}, \xi_{N_{s_1}},\cdots,\xi_{N_{s_{l}}}\right)\dd N_{s_1}\cdots \dd N_{s_{\ell}} \right]  \\
                    = \left(\eta^1\right)^{l}\E\left[ \int_{\R^{\ell}} F(s_1,\cdots,s_{l}, \xi_{1},\cdots ,\xi_{l}) \dd s_1\cdots \dd s_{l} \right] . 
                \end{multline}
            \end{lemma}
     We give a short proof of this lemma in the appendix section~\ref{app:probabilisticresults}.

            \begin{remark}
                The counting process \((N_t)\) has only jumps of size \(1\), i.e. \(N_{s} = N_{s-} + 1\) for any jump time \(s\).
                In \eqref{eq:stoch_int}, we should replace \(\xi_{N_{s_1}}\) by \(\xi_{N_{s_1-}+1}\) to avoid any doubt of measurability in the stochastic integral.
                However, we decide to keep this form to alleviate the notation.
            \end{remark}

            Using this lemma, we split the proof of Theorem~\ref{theo:jit} in the following steps:
            \begin{itemize}
                \item[\textbf{Step 1:}] We start with the proof of \eqref{eq:mean_jit}. It directly follows from the equality
                \begin{equation}\label{eq:mean_sjit}
                    \Eo\left[ \varphi(X^1_1,X^2_1) \right] - \Ez\left[ \varphi(X^1_1,X^2_1) \right] = \Eo\left[ \int_0^T\ind_{|\xi_{Z^1_1(u)}|\leq\delta}\ind_{u+\xi_{Z^1_1(u)}\in[0,T]}\dd Z^1_1 (u) \right] ,
                \end{equation}
                where $Z^1_1$ is the Poisson process which links the spike trains $X_1^1$ and $X_1^2$ 
                (see Model~\ref{mod:jittering}). Eq. \eqref{eq:mean_sjit} is obtained by distinguishing the jump times of both spike-trains that come 
                from the same underlying jump time. It is then easy to verify that the right hand term in \eqref{eq:mean_sjit} is equal to the one of 
                \eqref{eq:mean_jit} (see Appendix~\ref{subsectionB1}). In particular, we deduce the lower bound
                \begin{equation}\label{eq:lower_bound_jit}
                    \Eo\left[ \varphi(X^1_1,X^2_1) \right] - \Ez\left[ \varphi(X^1_1,X^2_1) \right] \geq 
                    \eta^1 (T - \delta)\eP(|\xi|\leq\delta) \geq
                    \frac{\eta^1 T}{2}\eP(|\xi|\leq\delta),
                \end{equation}
                since $ \delta \le T/2.$
                
                \item[\textbf{Step 2:}] Setting $\lamb = \max\{\lambda^1,\lambda^2\}$ and assuming that $\eta^1\leq\lamb$, we show that there exists a constant $C>0$ independent of the parameters such that the variances of $\fp(X_1,X_2)$ under the observed and the independent model satisfy the bounds
                \begin{equation}\label{eq:var_jit_ind}
                    \Vz\big[ \fp(X_1,X_2) \big] \leq C(\lamb)^2\delta T\big( 1 + \lamb\delta \big)
                \end{equation}
                and
                \begin{equation}\label{eq:var_jit}
                    \Vo\big[ \fp(X_1,X_2) \big] \leq C\big( 1+\lamb\delta \big)\big( (\lamb)^2\delta T + \eta^1 T\eP(|\xi|\leq\delta) \big).
                \end{equation}
                
            \end{itemize}
            We conclude the proof by replacing the mean, variances and the second order moment of $\fp(X_1,X_2)$ in the inequality \eqref{eq:20230127_1} of Lemma~\ref{lem:critgen} by their bounds in \eqref{eq:lower_bound_jit}, \eqref{eq:var_jit_ind} and \eqref{eq:var_jit}. The resulting inequality is solved according to the unknown $n$ in order to get the lower bound \eqref{eq:min_size_jit} on the number of observations which ensures the control of the {\color{black}error of second kind}.

        \subsection{Proof of Theorem~\ref{theo:Hawkes}}\label{sec:theo_Hawkes}
            As in the case of the jittering injection Poisson model, this result follows from Lemma~\ref{lem:critgen}, and its proof 
            consists of computing the expectations and variances of $\fp(X_1,X_2)$ for the observed and independent Hawkes models. 
            These computations use the cumulant measures of the univariate Hawkes process $N$ with intensity $\lambda(t)$ defined 
            by \eqref{eq:intensity} since this process allows to construct the entire network. 
            We start by giving some important {\color{black}information}
            about cumulant measures and their computation which are mainly based on \cite{rotter,priv}. 
          
            \subsubsection*{About the cumulant measures of a univariate Hawkes process}
            Let $N$ be a univariate Hawkes process with intensity
            \begin{equation}\label{eq:rate_Hawkes}
             \lambda(t) = \mu + \int_{(-\infty,t)}h(t-s)\dd N_s
            \end{equation}
            where the interaction function $h:\R_+\to\R_+$ is such that $\|h\|_{L^1}<1$. Let $l \geq 1$ be an integer, then (see \cite{rotter,priv},
             {\color{black}or \cite[Chap. 5]{DVJ2003} for a general presentation}) 
            the $l$-order cumulant measure of $N$ is the non negative measure defined by 
            \begin{equation}\label{eq:cumrec}
                k_{l}(\dd t_1,\cdots ,\dd t_{l}) = \sum_{\mathbf{p}}(|\mathbf{p}|-1)!(-1)^{|\mathbf{p}|-1}\bigotimes_{B\in\mathbf{p}}\E\big[\bigotimes_{i\in B}\dd N_{t_i}\big]
            \end{equation}
            where the sum goes over all partitions $\mathbf{p}$ of the set $\{1,\cdots ,l\}$, $|\mathbf{p}|$ corresponds to the number of elements of this partition and the notation $\E\big[\dd N_{t_1}\cdots \dd N_{t_n}\big]$ refers to the positive measure on $\R^n$ defined 
            for all measurable test  functions $ \phi $ having compact support by
            \[ 
            \int\phi(t_1,\cdots ,t_{n})\E\big[\dd N_{t_1}\cdots \dd N_{t_{n}}\big] = \E\bigg\{ \int\phi(t_1,\cdots ,t_{n})\dd N_{t_1}\cdots \dd N_{t_{n}} \bigg\}. 
            \]

            In the sequel we will mainly use the first and second order cumulants. The first order cumulant is the mean measure $k_1(\dd t_1) = \E[\dd N_{t_1}]$ and the second order cumulant is the covariance
            \begin{align*}
                k_2(\dd t_1,\dd t_2) = \E[\dd N_{t_1}\dd N_{t_2}] - \E[\dd N_{t_1}]\E[\dd N_{t_2}].
            \end{align*}
            In particular, the second order cumulant is equivalent to the first order cumulant on the diagonal $\{t_1=t_2\}$, that is
            \[ \ind_{t_1=t_2}k_2(\dd t_1,\dd t_2) = \delta_{t_1}(\dd t_2)k_1(\dd t_1). \]

            One well known property of the cumulant measure is the dual formula
            \begin{equation}\label{eq:dualformula}
                \E\big[\dd N_{t_1}\cdots \dd N_{t_{l}}\big] = \sum_{\mathbf{p}}\bigotimes_{B\in\mathbf{p}}k_{\# B}(\dd t_{i^B_1}, \cdots, \dd t_{i^B_{\# B}})
            \end{equation}
            where the notation $\# B$ refers to the cardinal of the subset $B$ and where we suppose that 
            $B=\{i^B_1,\cdots ,i^B_{\# B}\}$. Thus the positive measure $\E\big[\dd N_{t_1}\cdots \dd N_{t_{l}}\big]$ is entirely determined by the knowledge of the cumulant measures up to order $l$. An efficient way to compute these cumulant measures has been proposed by Jovanovi\'c et al. \cite{rotter} and consists in using the cluster representation of the Hawkes process (see \cite{Graham2021, HawkesOakes1974, Rasmussen2013}). Indeed, the Hawkes process $N$ can be constructed as {\color{black}a branching point process} as follows:
            \begin{enumerate}
                \item Simulate a realization of an homogeneous Poisson process with rate $\mu$ on $(-\infty,T]$. We call each point of this process an \emph{immigrant}. {\color{black}The immigrant process is denoted $\mathbf{I}$.}
                
                \item Each immigrant $t_0$ generates a cluster $C_{t_0}$. These clusters are mutually independent and each of them is constructed recursively:
                \begin{enumerate}
                    \item {\color{black} The generation 0 consists of the ancestor $t_0$, which implies that $t_0\in C_{t_0}$.}
                    \item The first generation is an inhomogeneous Poisson process on $[t_0,T]$ with rate $r_1(t) = h(t-t_0), t \geq t_0$.
                    \item Knowing the offspring generations $1,\cdots ,n$, every child $s$ of generation $n$ generates, {\color{black}independently of everything else,} its own offspring of generation $n+1$ as a realization of an inhomogeneous Poisson process on the time interval $[s,T]$ with rate $r_{n+1}(t) = h(t-s), t \geq s $.
                \end{enumerate}
            \end{enumerate}
            The Hawkes process $N$ corresponds to the superposition of all such clusters {\color{black}and can thus be seen as a branching point process}. {\color{black}Since $\|h\|_{L^1}<1$, all clusters are a.s. finite.}
            {\color{black} We prove in the next Theorem that the}
            cumulant measure $k_{l}(\dd t_1,...,\dd t_{l})$ expresses situations where  $t_1, \cdots, t_{l}$ 
            are indices of points belonging to 
            the same cluster. {\color{black}This result follows the idea of Jovanovic et al. \cite{rotter}.}

            {\color{black}In the sequel, we denote by $\mathbb{P}^x$ the distribution of the cluster $C_x$ issued from an immigrant at time $x$, for a given $x$. Note that by %
            the branching property, the law of all the descendants of a given point $u$ in a cluster $C_x$ %
            is also given by %
            $\mathbb{P}^u$. 
            More generally we write $\mathbb{P}^{x_1,...,x_k}$ for the distribution of the clusters issued from $x_1, \ldots , x_k,$ and $\mathbb{E}^{x_1,...,x_k}$ for the corresponding expectation. If the $x_i$'s are all different, then $\mathbb{P}^{x_1,...,x_k}=\bigotimes_{i=1,...,k}\mathbb{P}^{x_i}.$}
            
            {\color{black}
            	\begin{theorem}\label{theo:cluster_cumulant}
            		Let \( N \) denote a Hawkes process obtained by {\color{black}the above construction.}%
            		We have the identity:
            		\begin{equation}\label{eq:theo_clusterunique}
            			\mathbb{E}\left[\sum_{x\in \mathbf{I}} 
            			\sum_{t_1,  \cdots, t_l  \in C_{x}} \phi(t_1,\cdots,t_l)\right] = \int_{\mathbb{R}^l} \phi(t_1,\cdots,t_l)   k_l(\dd t_1, \cdots, \dd t_l)
            		\end{equation}
            		for any test function \(\phi\) {\color{black}having compact support}.
            	\end{theorem}
             }
             {\color{black}
            	\begin{proof}
            		Without loss of generality, we {\color{black}only consider} %
              $\phi(t_1,\cdots,t_l)$ {\color{black}of the form} $\phi_1(t_1)\times\cdots\times\phi_l(t_l)$.
            		
            		We prove the result by induction.
            		
            		\noindent	\textbf{Case \(l = 1\)}: by definition of \(k_1\).
            		
            		\noindent	\textbf{Case \(l = 2\)}: We have
            		\begin{eqnarray}\label{eq:20251016_3}
            			\mathbb{E}\left[
            			\sum_{t_1 %
               {\color{black}, } t_2 \in N}  \phi_1(t_1) \phi_2(t_2)\right] & =& 
            			\mathbb{E}\left[ \sum_{x\in \mathbf{I}} \sum_{\substack{t_1 %
               {\color{black}, } t_2 \\ \in C_{x}}} \phi_1(t_1) \phi_2(t_2)\right]
            			+  \mathbb{E}\left[ \sum_{\substack{x_1 \neq x_2 \\ \in \mathbf{I}}} \sum_{\substack{t_1  \in C_{x_1} \\ t_2 \in C_{x_2}}} \phi_1(t_1) \phi_2(t_2)\right] \nonumber \\
               &=:& \mathbb{E}\left[ \sum_{x\in \mathbf{I}} \sum_{\substack{t_1 %
               {\color{black}, } t_2 \\ \in C_{x}}} \phi_1(t_1) \phi_2(t_2)\right] + S . 
               \end{eqnarray}

            		We now evaluate further the expression $S.$ To do so we use {\color{black}that the intensity measure of a pair of different immigrants \(x_1\neq x_2\) is the product $\mu^2 \dd y_1 \dd y_2$.} The conditional independence of different clusters then implies that 
            		\begin{align*}
            			S & =
            			\mathbb{E}\left[ \sum_{\substack{x_1 \neq x_2 \\ \in \mathbf{I}}} \sum_{\substack{t_1  \in C_{x_1} \\ t_2 \in C_{x_2}}} \phi_1(t_1) \phi_2(t_2)\right] 
            			= 
            			\mathbb{E}\left[ \sum_{\substack{x_1 \neq x_2 \\ \in \mathbf{I}}}
            			\mathbb{E}^{x_1, x_2}\left[
            			\sum_{\substack{t_1  \in C_{x_1} \\ t_2 \in C_{x_2}}} \phi_1(t_1) \phi_2(t_2)\right]\right]\\
            			& =  \int_{\mathbb{R}^2} \mathbbm{1}_{y_1\neq y_2}
            			\mathbb{E}^{y_1,y_2}\left[\sum_{\substack{t_1  \in C_{y_1} \\ t_2 \in C_{y_2}}} \phi_1(t_1) \phi_2(t_2)\right]\mu \dd y_1 \,\mu \dd y_2 \\
            			&=  \int_{\mathbb{R}^2} \mathbbm{1}_{y_1\neq y_2}
            			\mathbb{E}^{y_1}\left[ \sum_{t_1  \in C_{y_1}} \phi_1(t_1)\right]
            			\mathbb{E}^{y_2}\left[ \sum_{t_2  \in C_{y_2}} \phi_2(t_2)\right]\mu \dd y_1 \,\mu \dd y_2.
            		\end{align*}
            		Finally, we have
            		\begin{align}
            			\nonumber	S & =  \int_{\mathbb{R}}
            			\mathbb{E}^{y_1}\left[ \sum_{t_1  \in C_{y_1}} \phi_1(t_1)\right]\mu \dd y_1
            			\int_{\mathbb{R}}\mathbb{E}^{y_2}\left[ \sum_{t_2  \in C_{y_2}} \phi_2(t_2)\right] \mu \dd y_2\\
            			\nonumber& = 
            			\mathbb{E}\left[ \sum_{x_1\in\mathbf{I}}\sum_{t_1  \in C_{x_1}} \phi_1(t_1)\right]
            			\mathbb{E}\left[ \sum_{x_2\in\mathbf{I}}\sum_{t_2  \in C_{x_2}} \phi_2(t_2)\right]\\
            			\label{eq:20251016_1}	& = 
            			\mathbb{E}\left[\sum_{t_1  \in N} \phi_1(t_1)\right]
            			\mathbb{E}\left[ \sum_{t_2  \in N} \phi_2(t_2)\right].
            		\end{align}
            		
            		We now use \eqref{eq:cumrec}. For \(l = 2\), there are only two partitions, 
            		\(\mathbf{p}_1 = \{\{t_1,t_2\}\}\) and \(\mathbf{p}_2 = \{\{t_1\}, \{t_2\}\}\).
            		Thus, taking \eqref{eq:20251016_1} into account, 
            		\begin{eqnarray}\label{eq:20251016_2}
            			\int_{\mathbb{R}^2} \phi_1(t_1)\phi_2(t_2) k_2(\dd t_1,\dd t_2) &=&  
            			\E\left[\sum_{t_1 %
               {\color{black}, } t_2 \in N} \phi_1(t_1)\phi_2(t_2)\right] 
            			-  \E\left[\sum_{t_1 \in N} \phi_1(t_1)\right]\E\left[\sum_{t_2\in N } \phi_2(t_2)\right] \nonumber \\
               &=& \E\left[\sum_{t_1 %
               {\color{black}, } t_2 \in N} \phi_1(t_1)\phi_2(t_2)\right] - S.
            		\end{eqnarray}
            		Substituting \eqref{eq:20251016_2} and \eqref{eq:20251016_1} into \eqref{eq:20251016_3}, we obtain
            		\[
            		\mathbb{E}\left[ \sum_{x\in \mathbf{I}} \sum_{\substack{t_1 %
              {\color{black}, } t_2 \\ \in C_{x}}} \phi_1(t_1) \phi_2(t_2)\right]
            		=  \int_\mathbb{R} \phi_1(t_1)\phi_2(t_2) k_2(\dd t_1,\dd t_2),
            		\]
            		that is \eqref{eq:theo_clusterunique} for \(l = 2\).
            		
            		\noindent	\textbf{Case \(l + 1\)}:
            		
            		We assume that the result holds for any cumulant up to order \(l\) and we now prove that it also holds for \(l + 1\).
            		
            		We first decompose the sum over partitions {\color{black}$\mathbf{p}$} of \( \{1, \ldots, l + 1\} \) according to whether the indices within each block {\color{black}$B$} of the partition originate from the same immigrant.
            		We isolate the trivial partition  	\(\mathbf{p}_1 = \{\{t_1, \cdots, t_{l+1}\}\}\). So, 
            		\begin{multline*}
            			\mathbb{E}\left[
            			\sum_{t_1, \cdots,  t_{l+1} \in N}  \phi_1(t_1) \cdots \phi_{l+ 1}(t_{l +1})\right]
            			=
            			\mathbb{E}\left[
            			\sum_{{\color{black}x} %
               \in \mathbf{I}}\sum_{\substack{t_1, \cdots,  t_{l+1} \\  \in C_{{\color{black}x}%
               }}}
            			\phi_1(t_1) \cdots \phi_{l+ 1}(t_{l +1})\right]\\
            			+
            			\mathbb{E}\left[
                        {\color{black}\sum_{\mathbf{p} :  |\mathbf{p}|\geq 2}  \sum_{\substack{(x_B)_{B\in \mathbf{p}}\\ \mbox{\tiny{all diff. in }}  \mathbf{I}}}  \prod_{B\in \mathbf{p}} 
                        \sum_{\substack{ t_{i^B_1}, \cdots, t_{i^B_{\# B}}\\ \mbox{\tiny{in }}  C_{x_B}}} \phi_{i^B_1}(t_{i^B_1})\cdots \phi_{i^B_{\# B}}(t_{i^B_{\# B}})
                        }
                                                \right] ,
                       \end{multline*}
            		where {\color{black}the notation is similar to \eqref{eq:dualformula}. In words, for every possible $(t_1,...,t_{l+1}),$ %
              we can group these points into points coming from the same cluster, leading to a partition $\mathbf{p}$. It remains to sum over all partitions and all possible cluster distributions to obtain the sum over all possible $l+1$-tuples.}

            		Similarly to the proof in the case \(l = 2\), we condition on the immigrants and obtain that
{\color{black}
\begin{align*}
S &: = \mathbb{E}\left[
                        \sum_{\mathbf{p}:  |\mathbf{p}|\geq 2} \sum_{\substack{(x_B)_{B\in \mathbf{p}}\\ \mbox{\tiny{all diff. in }}  \mathbf{I}}} \prod_{B\in \mathbf{p}} 
                        \sum_{\substack{ t_{i^B_1}, \cdots, t_{i^B_{\# B}}\\ \mbox{\tiny{in }}  C_{x_B}}} \phi_{i^B_1}(t_{i^B_1})\cdots \phi_{i^B_{\# B}}(t_{i^B_{\# B}})\right]\\
    &= \mathbb{E}\left[
                        \sum_{\mathbf{p}:  |\mathbf{p}|\geq 2} \sum_{\substack{(x_B)_{B\in \mathbf{p}}\\ \mbox{\tiny{all diff. in }}  \mathbf{I}}} \mathbb{E}^{(x_B)_{B\in \mathbf{p}}}\left\{\prod_{B\in \mathbf{p}} \sum_{\substack{ t_{i^B_1}, \cdots, t_{i^B_{\# B}}\\ \mbox{\tiny{in }}  C_{x_B}}} \phi_{i^B_1}(t_{i^B_1})\cdots \phi_{i^B_{\# B}}(t_{i^B_{\# B}})
                    \right\}
                        \right] .
 \end{align*}     
 Since the $C_{x_B}$'s are independent because their immigrants are different, and since the intensity measure associated with tuples of distinct points of $\mathbf{I}$ is a tensor of $\mu \dd y$, we end up with
 \begin{align*}
S & =\mathbb{E}\left[
                        \sum_{\mathbf{p} : |\mathbf{p}|\geq 2} \sum_{\substack{(x_B)_{B\in \mathbf{p}}\\ \mbox{\tiny{all diff. in }}  \mathbf{I}}} \prod_{B\in \mathbf{p}} \mathbb{E}^{x_B}\left\{\sum_{\substack{ t_{i^B_1}, \cdots, t_{i^B_{\# B}}\\ \mbox{\tiny{in }}  C_{x_B}}} \phi_{i^B_1}(t_{i^B_1})\cdots \phi_{i^B_{\# B}}(t_{i^B_{\# B}})
                    \right\}\right]\\
                    & = \sum_{\mathbf{p} : |\mathbf{p}|\geq 2} \prod_{B\in \mathbf{p}} \int \mathbb{E}^{y_B}\left\{\sum_{\substack{ t_{i^B_1}, \cdots, t_{i^B_{\# B}}\\ \mbox{\tiny{in }}  C_{y_B}}} \phi_{i^B_1}(t_{i^B_1})\cdots \phi_{i^B_{\# B}}(t_{i^B_{\# B}})
                    \right\} \mu \dd y_B \\
                    & =\sum_{\mathbf{p}: |\mathbf{p}|\geq 2} \prod_{B\in \mathbf{p}} \mathbb{E}\left[\sum_{x\in \mathbf{I}}\sum_{\substack{ t_{i^B_1}, \cdots, t_{i^B_{\# B}}\\ \mbox{\tiny{in }}  C_{x}}} \phi_{i^B_1}(t_{i^B_1})\cdots \phi_{i^B_{\# B}}(t_{i^B_{\# B}})\right] .
 \end{align*}    
 Since $|\mathbf{p}|\geq 2$, all $B\in \mathbf{p}$ satisfy $\#B \leq l$, so we can apply the recurrence and recognize the cumulants of order less than $l$. We conclude by using \eqref{eq:dualformula}.
}              
            	\end{proof}

            	\subsubsection*{Links with dendrogram}
            	Thanks to Theorem~\ref{theo:cluster_cumulant}, we see that in order to compute cumulant measures we need to evaluate integrals of the form 
            	\begin{equation}\label{eq:evalcumul}
            		\mathbb{E}\!\left[\sum_{x%
              \in\mathbf{I}}\sum_{t_1,\ldots,t_l \in C_{x%
              }}\phi(t_1,\ldots,t_l)\right],
            	\end{equation}
            	where \(t_1,\ldots,t_l\) denote %
             nodes belonging to the same cluster issued from immigrant \(x\).

            {\color{black}It is easy to see first that cumulants of higher orders with ties,  that is with repetitions, are just cumulants of smaller order. So in what follows we are only interested in computing 
            \begin{equation}\label{eq:evalcumul2}\mathbb{E}\!\left[\sum_{x%
              \in\mathbf{I}}\sum_{\substack{ t_1,\ldots,t_l \\ \mbox{\tiny{all diff. in }} C_x } 
              }\phi(t_1,\ldots,t_l)\right] .
              \end{equation}
            The algorithm proposed by Jovanovic et al. \cite{rotter} is designed to calculate such expressions. In what follows we explain their procedure further and formalize it using dendograms. We start with some easy computations.}
            We first introduce the function \(\Psi\)
\begin{equation}\label{eq:Psi_Hawkes}
	\Psi(t) = \sum_{m\geq 1}h^{\ast m}(t), t \geq 0, 
\end{equation}
where we define recursively for any $ t \geq 0, $
\begin{equation}\left\{
	\begin{array}{l}
		\displaystyle h^{\ast 1}(t) = h(t), 
		\\
		\displaystyle h^{\ast (m+1)}(t) = \int_0^t h^{\ast m}(t-s)h(s)\dd s,  m\geq 1.
	\end{array}\right.
\end{equation}
The function $h^{\ast m}(t)$ is the rate of occurrence of offspring in the $m$-th generation issued from an immigrant at time $0.$ 
As a consequence,  $\Psi(t)$ is the rate of occurrence of offspring of an immigrant at time $0$ whatever their generation. {\color{black}Because of this, and because $v$ belongs to $C_v$, given an ancestor $v$, we have that the mean measure of the cluster process issued from $v$ is 
\begin{equation}\label{meanmeas}
\delta_v(\dd u)+\Psi(u-v) \dd u,
\end{equation}
in the sense that for all functions $\phi$ having compact support,
\[
\mathbb{E}^v\left(\sum_{u\in C_v} \phi(u)\right)= \int \phi(u) (\delta_v(\dd u)+\Psi(u-v) \dd u).
\]
}

{\color{black}Moreover, we say that inside a cluster $C_v$, points belong to different family lines if their only common ancestor is $v.$ This means that they belong in fact to clusters issued from different offspring of $v$ in the first generation. Then is it easy to compute}
\begin{equation}\label{eq:twofamilies}
\mathcal{E}:=\mathbb{E}^v\left(\sum_{\substack{u_1, u_2 \mbox{\tiny{in diff.}} \\ \mbox{\tiny{fam. lines of }} C_v}}\phi_1(u_1)\phi_2(u_2)\right),
\end{equation}
{\color{black}for two functions $\phi_1$ and $\phi_2$ having compact support. 
Indeed, denoting $\mathcal{O}_1$ the first generation of points issued by $v,$ we obtain}
\begin{multline}
    \mathcal{E}=
    \mathbb{E}^v\left( \sum_{w_1 \neq w_2 \in \mathcal{O}_1} \sum_{u_1\in C_{w_1}}\phi_1(u_1)\sum_{u_2\in C_{w_2}}\phi_2(u_2)\right) \\
    = \mathbb{E}^v\left[ \sum_{w_1 \neq w_2 \in \mathcal{O}_1}\int \phi_1(u_1) (\delta_{w_1}(\dd u_1)+\Psi(u_1-w_1) \dd u_1)\int \phi_2(u_2) (\delta_{w_2}(\dd u_2)+\Psi(u_2-w_2) \dd u_2) \right] \\
    = \int \phi_1(u_1)\phi_2(u_2) \Psi(u_1-v)\Psi(u_2-v) \dd u_1 \dd u_2, \label{meanmeas2}
\end{multline}
{\color{black}
because $\mathcal{O}_1$ is Poisson with intensity $h(\cdot -v)$ and 
$\Psi$ is defined by recursive convolutions of $h$ with itself.
}

{\color{black}These two computations show us that in \eqref{eq:evalcumul2}, the shape of the cluster is important only when two distinct family lines carrying different sets of $t_i$'s coalesce at their most common ancestor. This is why Jovanovic et al. \cite{rotter} construct a simplified tree. This simplified tree corresponds in fact to a classical mathematical object in clustering, the dendrogram (see \cite{Murtagh1984}). Let us detail this construction and give an example.
}

Let \( \mathbf{t} = \{t_1, \cdots, t_l\} \) denote a finite set of distinct times {\color{black}of the same cluster}.  
            	To  compute~\eqref{eq:evalcumul2}, we simplify the corresponding genealogical tree by {\color{black}only keeping the successive closest  common ancestors of the $t_i$'s and discarding all the points in the complete genealogical tree that are intermediate within the independent family lines. More precisely, }%
            	the simplification {\color{black}of the complete genealogical tree of the cluster} is performed according to the following steps:
            	\begin{enumerate}
            		\item We retain all nodes associated with the times \(t_i\);
            		\item We remove the descendants of each \(t_i\) that do not contain any other \(t_j\); as a result, most of the \(t_i\) become leaves in the reduced tree;
            		\item We preserve only those internal nodes that are necessary to distinguish the genealogical relationships among the leaves corresponding to the \(t_i\);
            		\item \label{itemstep4} If one of the \(t_i\), say \(t_{i_0}\), has another \(t_j\) among its descendants, we introduce an additional internal node artificially, say \(u_0\). We replace \(t_{i_0}\) by \(u_0\) in the tree and
                    \(t_{i_0}\) is now viewed as direct descendant of \(u_0\). Thus, \(t_{i_0}\) is now a leaf;
            		\item \label{itemstep5} We also retain the immigrant node. If the immigrant coincides with the most recent common ancestor of the set \(\mathbf{t}\), we introduce an additional artificial internal node in this case as well.
            	\end{enumerate}

{\color{black}We will discuss the role of the additional nodes that may be added according to steps \eqref{itemstep4} or \eqref{itemstep5} in Remark \ref{rem:additional} below.} 
              
            Before proceeding further, let us illustrate this reduction procedure with a concrete example.
            	Figure~\ref{fig:bigtree} displays a typical cluster. This cluster is almost surely finite under the assumption \(\lVert h\rVert < 1\).
            	The nodes \(t_i\) involved in~\eqref{eq:evalcumul2} are highlighted in red.
            	\begin{figure}[ht]
            		\centering
            		\begin{tikzpicture}[
            			x=18mm, y=4mm,
            			every node/.style={rectangle,draw,rounded corners,inner sep=2pt,font=\small},
            			internal/.style={fill=blue!15},
            			leaf/.style={fill=gray!10},
            			maked/.style={fill=red!15},
            			edge/.style={-}
            			]

            			\def\Tmin{0}
            			\def\Tmax{24}    %
            			\def\AxisX{-1.2}
            			
            			\draw[thick] (\AxisX,-\Tmin) -- (\AxisX,-\Tmax);
            			\foreach \tick in {\Tmin,...,\Tmax}{%
            				\draw (\AxisX-0.15,-\tick) -- (\AxisX+0.15,-\tick);
            				\node[anchor=east,font=\scriptsize] at (\AxisX-0.2,-\tick) {\tick};
            			}

            			\node[internal] (x) at (0,-1) {$u$};

            			\node[internal] (u) at (1,-4) {$v$};
            			\node[internal] (y) at (1,-7) {$y$};
            			\node[leaf]     (z) at (1,-12) {$z$};
            			\draw (x) -- (u);
            			\draw (x) -- (y);
            			\draw (x) -- (z);

            			\node[maked] (v)  at (2,-6)  {$t_3$};
            			\node[leaf]     (t4) at (2,-10) {$e_4$};
            			\node[internal] (w)  at (2,-14) {$w$};
            			\draw (u) -- (v);
            			\draw (u) -- (t4);
            			\draw (u) -- (w);

            			\node[maked]     (t3) at (3,-8)  {$t_6$};
            			\node[internal] (q)  at (3,-11) {$q$};
            			\draw (v) -- (t3);
            			\draw (v) -- (q);

            			\node[leaf] (r1) at (4,-13) {$r_1$};
            			\node[internal] (r2) at (4,-15) {$r_2$};
            			\node[leaf] (r3) at (4,-18) {$r_3$};
            			\draw (q) -- (r1);
            			\draw (q) -- (r2);
            			\draw (q) -- (r3);

            			\node[leaf] (t1) at (3,-16) {$d_1$};
            			\node[internal] (t2) at (3,-19) {$d_2$};
            			\draw (w) -- (t1);
            			\draw (w) -- (t2);

            			\node[maked] (s1) at (2,-8)  {$t_1$};
            			\node[leaf] (s2) at (2,-9)  {$s_2$};
            			\node[maked] (s3) at (2,-11) {$t_2$};
            			\node[leaf] (s4) at (2,-13) {$s_4$};
            			\draw (y) -- (s1);
            			\draw (y) -- (s2);
            			\draw (y) -- (s3);
            			\draw (y) -- (s4);

            			\node[maked] (a1) at (5,-17) {$t_4$};
            			\node[leaf] (a2) at (5,-23.2) {$a_2$};
            			\draw (r2) -- (a1);
            			\draw (r2) -- (a2);
            			
            			\node[internal] (b1) at (4,-20) {$b_1$};
            			\node[leaf] (b2) at (4,-21.5) {$b_2$};
            			\node[leaf] (b3) at (4,-23) {$b_3$};
            			\node[leaf] (c4) at (5,-20.5) {$c_4$};
            			\node[maked] (c5) at (5,-22) {$t_5$};
            			
            			\draw (t2) -- (b1);
            			\draw (t2) -- (b2);
            			\draw (t2) -- (b3);
            			\draw (b1) -- (c4);
            			\draw (b1) -- (c5);
            			\node[leaf] (c1) at (3,-12.5) {$c_1$};
            			\draw (s3) -- (c1);
            		\end{tikzpicture}
            		\caption{A typical cluster: {\color{black}on the vertical axis, the time of arrivals of the successive points inside the cluster. In this example $u$ is an immigrant.}}
            		\label{fig:bigtree}
            	\end{figure}
            	In Figure~\ref{fig:reduceddendro}, we apply the announced procedure. Note that here we added the artificial node $\bar{w}$ 
             (step \eqref{itemstep4}) and %
             added $x$ (step \eqref{itemstep5}). We stress that the levels in this {\color{black}resulting} dendrogram do not correspond to the generations in the original cluster.
            	\begin{figure}[ht]
            		\centering
            		\begin{tikzpicture}[level distance=10mm,every node/.style={{fill=blue!25},rectangle,draw},level 1/.style={sibling distance=20mm,nodes={fill=blue!1}},level 2/.style={sibling distance=25mm,nodes=},level 3/.style={sibling distance=10mm,nodes=},level 4/.style={sibling distance=10mm,nodes=}, 		maked/.style={fill=red!15},]
            			\node {\(x\)}
            			child {node {\(u\)}
            			child {node {\(v\)}
            				child {node {\(\bar{w}\)}
            					child {node [maked]{\(t_3\)}}
            					child {node [maked]{\(t_4\)}}
            					child {node [maked]{\(t_6\)}}}
            				child{node [maked]{\(t_5\)}}}
            			child {node {\(y\)}
            				child {node [maked]{\(t_1\)}}
            				child {node [maked]{\(t_2\)}}}
            			};
            		\end{tikzpicture} 
            		\caption{Reduced dendrogram associated to Figure~\ref{fig:bigtree}. We have added \(\bar{w}\) in order to force \(t_3\) to be a leaf, and \(x\).}
            		\label{fig:reduceddendro}
            	\end{figure}

Let us come back to the discussion of the general reduction procedure according to steps (1)-(5) described above. The result of this procedure can be interpreted %
as a dendrogram \cite{Murtagh1984} since each intermediate node can at the same time be seen as a grouping of the points below in the simplified tree as well as the common ancestor of these points.

              We denote by $\mathbb{Q}(t_1,\cdots,t_{l})$ the set of dendrograms with terminal nodes $t_1,\cdots,t_{l}$.

              {\color{black}
Given a dendrogram $\q\in\mathbb{Q}(t_1,\cdots,t_{l})$, for any  non-terminal node $v$ of $\q$ we introduce the following notations: 
\begin{itemize}
	\item $\ch_{\q}(v)$ denotes the set of children of $v$ in $\q$,
	\item $\fe_{\q}(v)$ denotes the set of terminal nodes in $\ch_{\q}(v)$.
\end{itemize}

To use the dendrogram representation and perform calculations similar to \eqref{meanmeas2} above, we have to assign %
measure valued weights to the non-terminal nodes and the roots of the dendrograms in $\mathbb{Q}(t_1,\cdots,t_{l})$. 
For each node {\color{black}$v$,} this weight corresponds to the {\color{black}mean measure of the family of children of $v$ in the dendrogram, conditionally on $v. $ Here, the presence of an edge from a parent to its children (where the children are supposed to be different from the parent) in the dendrogram means that the children belong  to independent family lines issued from the parent. Exceptions are possible at the root and at the terminal nodes where equality between the parent and one of the children is possible. } %

{\color{black}This leads to the definition of the following weights.}
\begin{itemize}
	\item The root $x$ {\color{black}(an immigrant)} is weighted by %
 the measure
	\[
	\Lambda[x](\dd u) = \delta_x(\dd u) + \Psi(u - x)\dd u
	\]
 (see \eqref{meanmeas}).
	The Dirac part \(\delta_x\) of this measure corresponds to the case when the immigrant in the primary cluster is the most recent 
	common ancestor of the  \(t_1, \cdots, t_l\) {\color{black}(in this case, we have added an artificial node)}.%
	\item 	A typical internal node $v$ is weighted by the measure
	\begin{multline}\label{eq:internalweight}
		\Lambda[v](\dd u_1^v,\cdots,\dd u_{j_v}^v) = \Psi(u_1^v - v)\cdots\Psi(u_{j_v}^v-v)\dd u_1^v\cdots \dd u_{j_v}^v \\
		+ \mathbbm{1}_{\{\min u_i^v \in\fe_{\q}(v)\}}\delta_v(\dd (\min u_i^v))\prod_{u_l^v\neq \min u_i^v}[\Psi(u_l^v - v) \dd u_l^v]
  	\end{multline}
   where \(\ch_{\q}(v) = \{u_1^v,\cdots,u^v_{j_v}\}\).

 {\color{black}Note that the first term corresponds exactly to the computation made in \eqref{meanmeas2} with $j_v$ independent family lines. The Dirac measure in the second term is there to allow the possibility for the minimal terminal node to be the common ancestor of all other terminal nodes (in this case, $v$ is an artificial node).}
  \item Once we have integrated over all the measures corresponding to the internal nodes of the tree, we need to integrate with respect to $\mu dx$, to integrate over the immigrants.
\end{itemize}

\begin{remark}\label{rem:additional} Let us now explain the role of the artificial nodes in our procedure. %
If no artificial node were ever introduced in the procedure, each edge in the simplified tree would correspond to an independent family line. In this case, once we know the shape of the tree, we know how to condition recursively on the closest common ancestors to compute the corresponding mean measures as e.g. in \eqref{eq:twofamilies} and \eqref{meanmeas2} above.%
             
            The introduction of an artificial node that may happen at step \eqref{itemstep4} is there to have a strict equivalence with dendrograms. The fact that one of the leafs \(t_i\) can be an ancestor of the other ones is taken into account by the addition of the Dirac mass  in \eqref{eq:internalweight}. The last  node  at step \eqref{itemstep5} (which is not there in the classical definition of dendrogram) is there to  address the integration with respect to the immigrant. If this immigrant turns out to be the closest common ancestor, we have an artificial node and  an extra Dirac mass. Except at the root and the \(t_i\), we only envision independent family lines and therefore, there are no further Dirac masses in the center of the dendrogram.
            \end{remark}
            
            }

    {\color{black}Putting these things together allows us to state the following general result which is a formalization of the result by Jovanovi\'c et al. \cite{rotter}.%
    }
     {\color{black}}  	
            	\begin{theorem}
            		The $l$-th order cumulant measure $k_{l}(\dd t_1,\cdots,\dd t_{l})$ of the Hawkes process $N$ can be represented as follows. 
            		\begin{enumerate}
            			\item  
            			\begin{equation}\label{eq:order-1}
            				k_1(\dd t_1) = \frac{\mu}{1-\|h\|_{L^1}}\dd t_1.
            			\end{equation}
            			
            			\item On the diagonals: if $1\leq i,j\leq l$, then
            			\begin{equation}
            				\ind_{t_i=t_j}k_{l}(\dd t_1,\cdots,\dd t_{l}) = \delta_{t_i}\left(\dd t_j\right)k_{l-1}\left(\dd t_1,\cdots,\dd t_{j-1},\dd t_{j+1},\cdots,\dd t_{l}\right).
            			\end{equation}

            			\item Out of the diagonals: %
               setting 
            			\begin{equation}\label{eq:defdeltal}
            				\Delta_{l} = \{(t_1,\cdots,t_{l})\in\R^{l}: t_i\neq t_j, \forall i\neq j\},
            			\end{equation}
            			we have
            			\begin{multline}
            				\ind_{\Delta_{l}}k_{l}\left(\dd t_1,\cdots ,\dd t_{l}\right) 
            				\\
            				=\sum_{\q\in\mathbb{Q}(t_1, \cdots, t_{l})}\int_{\R^{1+p}} \mu \dd x
            				\Lambda[x](\dd v_1)  \Lambda[v_1](\dd u^{v_1}_1, \cdots, \dd u^{v_1}_{j_1})  \cdots  \Lambda[v_p](\dd u^{v_p}_1, \cdots, \dd u^{v_p}_{j_p}),
            			\end{multline}
            			where $p = p(\q)$ is the number of non-terminal nodes of $\q$, denoted $v_1,\cdots,v_p$.
            			In the last integral, the integration is to be understood with respect to \(x\) and \(v_1, \ldots, v_p\).
            			Note that all of the \(u^{v_i}_j\) may coincide with a variable \(v_l \neq v_i\).
            			Moreover, all \(t_i\) are also contained in the set \(\{u^{v_i}_j\}\). 
            		\end{enumerate}
            		
            	\end{theorem}	
            	\begin{proof}
            		The result follows as a straightforward consequence of Theorem~\ref{theo:cluster_cumulant} {\color{black}and successive conditionings on the different internal nodes of the dendrogram and on the immigrants.}
            	\end{proof}	
                           }

            Once the above result stated, it remains to list all dendograms with a given number of terminal nodes. Clearly, the less nodes we have, the less complicated is the enumeration of all possible dendrograms. For instance, Figure~\ref{fig:cumul_order-1} shows the unique dendrogram having two terminal nodes $t_1 $ and $t_2$. 

            \begin{figure}[ht]
	\centering
	\begin{tikzpicture}[level distance=10mm,every node/.style={{fill=blue!25},rectangle,draw},level 1/.style={sibling distance=20mm,nodes={fill=blue!1}},level 2/.style={sibling distance=10mm,nodes=},level 3/.style={sibling distance=5mm,nodes=}]\node {\(x\)}child {node {\(u\)}child {node {\(t_1\)}}child {node {\(t_2\)}}};
	\end{tikzpicture}                 
	\caption{The unique dendrogram for $l=2$.}
	\label{fig:cumul_order-1}
\end{figure}

            In practice,  Murtagh~\cite{Murtagh1984} showed how to compute the number of dendrograms with $l$ terminal nodes, given $l\geq 1$. However listing the dendrograms is tedious when $l$ increases.  For example, according to Murtagh~\cite{Murtagh1984}, there are: one dendrogram for $l=2$, four for $l=3$, twenty six for $l=4$ and even more for $l$ larger.  But the set of dendrograms $\mathbb{Q}(t_1,\cdots,t_{l})$ can be decomposed in equivalence classes whose elements result from permutations of the positions of the terminal nodes $t_1,\cdots,t_{l}$. Notice that the permutations that only swap the positions of terminal nodes with the same parents do not modify the dendrogram. Thus they only count as one.
            We have the following particular cases:
            \begin{itemize}
                \item In the case $l=2$, we have only one dendrogram given in Figure~\ref{fig:cumul_order-1}.
                      We deduce that the second order cumulant measure $k_2(\dd t_1,\dd t_2) = \E[\dd N_{t_1}\dd N_{t_2}]-\E[\dd N_{t_1}]\E[\dd N_{t_2}]$ satisfies
{\color{black}
                \begin{equation}\label{eq:order-2}
                    \ind_{\Delta_2}k_2(\dd t_1,\dd t_2) = \int_{\R^2}\mu \dd x \Lambda[x](\dd u) \Lambda[u](\dd t_1,\dd t_2).
                \end{equation}
                where the integration is with respect to the immigrant and the internal node \(\dd x\) and \(\dd u\).}
                \item In the case $l=3$, we have four dendrograms but only two equivalence classes described in Figure~\ref{fig:cumul_order-3}.

                \begin{figure}[ht]
                    \centering
\begin{subfigure}{0.4\textwidth}
\centering
		\begin{tikzpicture}[level distance=10mm,every node/.style={{fill=blue!25},rectangle,draw},level 1/.style={sibling distance=20mm,nodes={fill=blue!1}},level 2/.style={sibling distance=10mm,nodes=},level 3/.style={sibling distance=5mm,nodes=}]
		\node {\(x\)}
		    child {node {\(u\)}
		        child {node {\(t_1\)}}
		        child {node {\(t_2\)}} 
		        child {node {\(t_3\)}}
		        };
	\end{tikzpicture}               
                        \caption{}
\end{subfigure}
\hfill
\begin{subfigure}{0.4\textwidth}
\centering
         	\begin{tikzpicture}[level distance=10mm,every node/.style={{fill=blue!25},rectangle,draw},level 1/.style={sibling distance=20mm,nodes={fill=blue!1}},level 2/.style={sibling distance=10mm,nodes=},level 3/.style={sibling distance=10mm,nodes=}]
		\node {\(x\)}
		    child {node {\(u\)}
		        child {node {\(v\)}
		            child {node {\(t_1\)}}
		            child {node {\(t_2\)}}}
    		    child {node {\(t_3\)}}
		        };
	\end{tikzpicture}  
	                        \caption{}
	\end{subfigure}
        
\caption{Equivalence classes of dendrograms for $l=3$: \textbf{(a)} has only one permutation, \textbf{(b)} has three permutations..}
\label{fig:cumul_order-3}
\end{figure}

                \noindent It follows that the third order cumulant satisfies
 {\color{black}
                \begin{multline}\label{eq:order-3}
                    \ind_{\Delta_3}k_3(\dd t_1,\dd t_2,\dd t_3) = \bigg\{ \int_{\R^2} \mu \dd x \Lambda[x](\dd u)\Lambda[u](\dd t_1,\dd t_2,\dd t_3)
                    \\
                    + \sum_{\sigma}\int_{\R^3}\mu \dd x \Lambda[x](\dd u)\Lambda[u](\dd v,\dd t_{\sigma(3)}]\Lambda[v](\dd t_{\sigma(1)}, \dd t_{\sigma(2)}) \bigg\}
                \end{multline}}
where the sum is over all the permutations of $\{1,2,3\}$ in the equivalence class that corresponds to the dendrogram \textbf{(b)} in Figure~\ref{fig:cumul_order-3}.
                \item In the case $l=4$, there are $26$ dendrograms, but only five equivalence classes that are described in Figure~\ref{fig:cumul_order-4}.

                 \begin{figure}[ht]
                    \centering
\begin{subfigure}[T]{0.3\textwidth}
\centering
		\begin{tikzpicture}[level distance=10mm,every node/.style={{fill=blue!25},rectangle,draw},level 1/.style={sibling distance=20mm,nodes={fill=blue!1}},level 2/.style={sibling distance=10mm,nodes=},level 3/.style={sibling distance=5mm,nodes=}]
		\node {\(x\)}
		    child {node {\(u\)}
		        child {node {\(t_1\)}}
		        child {node {\(t_2\)}} 
		        child {node {\(t_3\)}}
		        child {node {\(t_4\)}}
		        };
	\end{tikzpicture}               
                        \caption{}
\end{subfigure}
\hfill
\begin{subfigure}[T]{0.3\textwidth}
\centering
         	\begin{tikzpicture}[level distance=10mm,every node/.style={{fill=blue!25},rectangle,draw},level 1/.style={sibling distance=20mm,nodes={fill=blue!1}},level 2/.style={sibling distance=10mm,nodes=},level 3/.style={sibling distance=10mm,nodes=},level 4/.style={sibling distance=10mm,nodes=}]
		\node {\(x\)}
		    child {node {\(u\)}
		        child {node {\(v\)}
		            child {node {\(t_1\)}}
		            child {node {\(t_2\)}}}
    		    child {node {\(t_3\)}}
    		    child {node {\(t_4\)}}
		        };
	\end{tikzpicture}  
	                        \caption{}
	\end{subfigure}
        \hfill
\begin{subfigure}[T]{0.3\textwidth}
\centering
         	\begin{tikzpicture}[level distance=10mm,every node/.style={{fill=blue!25},rectangle,draw},level 1/.style={sibling distance=20mm,nodes={fill=blue!1}},level 2/.style={sibling distance=10mm,nodes=},level 3/.style={sibling distance=10mm,nodes=},level 4/.style={sibling distance=10mm,nodes=}]
		\node {\(x\)}
		    child {node {\(u\)}
		        child {node {\(v\)}
		            child {node {\(t_1\)}}
		            child {node {\(t_2\)}}
		            child {node {\(t_3\)}}}
    		    child {node {\(t_4\)}}
		        };
	\end{tikzpicture}  
	                        \caption{}
	\end{subfigure}
\hfill
\begin{subfigure}[T]{0.3\textwidth}
\centering
         	\begin{tikzpicture}[level distance=10mm,every node/.style={{fill=blue!25},rectangle,draw},level 1/.style={sibling distance=20mm,nodes={fill=blue!1}},level 2/.style={sibling distance=20mm,nodes=},level 3/.style={sibling distance=10mm,nodes=},level 4/.style={sibling distance=10mm,nodes=}]
		\node {\(x\)}
		    child {node {\(u\)}
		        child {node {\(v\)}
		            child {node {\(t_1\)}}
		            child {node {\(t_2\)}}}
		        child {node {\(w\)}
        		    child {node {\(t_3\)}}
        		    child {node {\(t_4\)}}}
        		    };
	\end{tikzpicture}  
	                        \caption{}
	\end{subfigure}
        \hfill
\begin{subfigure}[T]{0.3\textwidth}
\centering
         	\begin{tikzpicture}[level distance=10mm,every node/.style={{fill=blue!25},rectangle,draw},level 1/.style={sibling distance=20mm,nodes={fill=blue!1}},level 2/.style={sibling distance=10mm,nodes=},level 3/.style={sibling distance=10mm,nodes=},level 4/.style={sibling distance=10mm,nodes=}]
		\node {\(x\)}
		    child {node {\(u\)}
		        child {node {\(v\)}
		            child {node {\(w\)}
		                child {node {\(t_1\)}}
		                child {node {\(t_2\)}}}
		            child {node {\(t_3\)}}}
			child {node {\(t_4\)}}
		        };
	\end{tikzpicture}  
	                        \caption{}
	\end{subfigure}

\caption{Equivalence classes of dendrograms for $l=4$: \textbf{(a)} has one permutation, \textbf{(b)} has six permutations, \textbf{(c)} has four permutations, \textbf{(d)} has three permutations, \textbf{(e)} has twelve permutations.}
\label{fig:cumul_order-4}
\end{figure}

                \noindent We deduce that the fourth order cumulant measure satisfies
                {\color{black}
                \begin{multline}\label{eq:order-4}
                    \ind_{\Delta_4}k_4(\dd t_1,\dd t_2,\dd t_3,\dd t_4) = \bigg\{ \int_{\R^2}\mu \dd x \Lambda[x](\dd u) \Lambda[u](\dd t_1,\dd t_2,\dd t_3,\dd t_4)
                    \\
                    + \sum_{\sigma}\int_{\R^3}\mu \dd x \Lambda[x](\dd u)\Lambda[u](\dd v,\dd t_{\sigma(3)},\dd t_{\sigma(4)})
                    \Lambda[v](\dd t_{\sigma(1)},\dd t_{\sigma(2)})
                    \\
                    + \sum_{\sigma}\int_{\R^3}\mu \dd x \Lambda[x](\dd u)\Lambda[u](\dd v,\dd t_{\sigma(4)})\Lambda[v](\dd t_{\sigma(1)}, \dd t_{\sigma(2)}, \dd t_{\sigma(3)})
                    \\
                    + \sum_{\sigma}\int_{\R^4}\mu \dd x \Lambda[x](\dd u)\Lambda[u](\dd v,\dd w)\Lambda[v](\dd t_{\sigma(1)}, \dd t_{\sigma(2)})\Lambda[w](\dd t_{\sigma(3)}, \dd t_{\sigma(4)})
                    \\
                    +\sum_{\sigma}\int_{\R^4}\mu \dd x \Lambda[x](\dd u)\Lambda[u](\dd v,\dd t_{\sigma(4)})\Lambda[v](\dd w,\dd t_{\sigma(3)})\Lambda[w](\dd t_{\sigma(1)}, \dd t_{\sigma(2)}),
                \end{multline}}
                where the sums are taken over the permutations of the terminal nodes within the equivalence classes \textbf{(b)}, \textbf{(c)}, \textbf{(d)}, \textbf{(e)} in Figure~\ref{fig:cumul_order-4} respectively.
            \end{itemize}
            This method can be extended to $l\geq 5$ in order to compute cumulant measures of higher order. In the particular case of our work and as we will see further, we only need bounds on the cumulant measures up to the first four orders. We deduce the following result that will be very useful in the sequel.
            
            \begin{lemma}\label{lem:cumulant}
                Let us consider the univariate Hawkes process $N$ with rate $\lambda(t)$ introduced in \eqref{eq:rate_Hawkes} above. We take $h(s) = ae^{-bs}\ind_{s>0}$ where $0<a<b,$ and write $ \ell = a/b < 1.$ Then the first and second order cumulants of $N$ satisfy 
                \begin{align}
                    k_1(\dd t_1) & = \frac{\mu}{1-\ell }\dd t_1, \label{eq:cumul_order-1}
                    \\
                    \nonumber \ind_{\Delta_2}k_2(\dd t_1,\dd t_2) & = \frac{a\mu}{1 - \ell} \bigg( 1 + \frac{1}{2}\frac{a}{b-a} \bigg)e^{-b (1-\ell )|t_1-t_2|}\dd t_1\dd t_2 \\
                    &=  \frac{a\mu}{1 - \ell} \left( 1 + \frac{1}{2}\frac{\ell }{1- \ell} \right)e^{-b (1-\ell )|t_1-t_2|}\dd t_1\dd t_2 . \label{eq:cumul_order-2}
                \end{align}
                More generally, for any $m=1,2,3,4,$ there exists $C_{m}>0$ such that we have the bound
                \begin{equation}\label{eq:cumul_order-1234}
                    \ind_{\Delta_{m}}k_{m}(\dd t_1,\cdots,\dd t_{m}) \leq C_{m}\frac{a^{m-1}\mu}{(1-\ell )^{m}}e^{-b (1-\ell)\left(\max_it_i - \min_it_i\right)}\dd t_1\cdots \dd t_{m},
                \end{equation}
                with the convention $\Delta_1=\R$.
                            \end{lemma}
            In this result, the inequality $m_1\leq m_2$ between two positive measures holds in the sense that for any non negative test function $\phi$ we have $\int\phi(t)m_1(\dd t) \leq \int\phi(t)m_2(\dd t)$.
            
            \subsubsection*{Main steps of the proof of Theorem~\ref{theo:Hawkes}} Recall that we want to use the criterion given in 
            Lemma~\ref{lem:critgen} in order to control Type II error of the test $\ind_{\bC_n(\X_n)>q_{1-\alpha}(\X_n)}$ by $\beta$. For that purpose, we need to compute \(\Delta_\varphi\) 
            and the variances of $\fp(X_1,X_2)$ in those models. We then split the proof as follows. 
            \begin{itemize}
                \item[\textbf{Step 1:}] We first show that the variances of $\fp(X_1,X_2)$ under each model (the observed and the independent) can be approximated by its variance $\V_0$ under the mean-field limit model. Under this very model, the spike trains are independent Poisson processes with intensity $\frac{\nu}{1-\ell}$ (see Delattre et al. \cite{del}). We prove in the appendix that  
                \begin{equation}\label{eq:var_mf_limit}
                    \V_0 = \frac{4\nu^2\delta T}{(1-\ell )^2} \left\{ 1 + \frac{2\nu\delta}{1-\ell}  - \frac{\delta}{2T}\left( 1 + \frac{10}{3}\frac{\nu\delta}{1-\ell} \right) \right\}.
                \end{equation}
                Using Lemma~\ref{lem:cumulant}, we deduce that there exists a constant $C>0$ independent of the parameters such that
                \begin{equation}\label{eq:var_Hawkes_ind}
                    0\leq \Vz\big[ \fp(X_1,\bX_2) \big] - \V_0 \leq C\frac{aT}{M}\frac{\nu^2\delta^2}{(1-\ell )^3}\left\{ 1 + \frac{\nu+\frac{a}{M}}{b( 1 - \ell)^2} \right\} 
                \end{equation}
                and
                \begin{equation}\label{eq:var_Hawkes}
                    0\leq \Vo\big[ \fp(X_1,\bX_2) \big] - \V_0 \leq C\frac{aT}{M} \frac{\nu\delta}{(1-\ell )^2}\bigg\{ 1 + \frac{(\nu+\frac{a}{M})\delta}{1-\ell }\left[ 1 + \frac{\nu+\frac{a}{M}}{b( 1-\ell)^2} \right] \bigg\}.
                \end{equation}
               As a consequence, the difference between the variances of $\fp(X_1,X_2)$ under the observed and independent models is proportional to $1/M$.
                
                \item[\textbf{Step 2:}] We prove in the Appendix~\ref{appendixC2} that \(\Delta_\varphi\) is equal to
                \begin{equation}\label{eq:mean_cumul_hawkes}
                    \Delta_\varphi = \Eo\big[ \varphi(X^1_1,X^2_1) \big] - \Ez\big[ \varphi(X^1_1,X^2_1) \big] = \frac{1}{M^2}\int_0^T\int_0^T\ind_{|u-v|\leq\delta}\ind_{u\neq v}k_2(\dd u,\dd v)
                \end{equation}
                where $1/M^2$ is the probability that the jump times $u,v$  correspond to the first and the second coordinates of an observation. Using \eqref{eq:cumul_order-2} in Lemma~\ref{lem:cumulant}, we deduce that 
                \begin{equation}
                \label{eq:mean_Hawkes}
                    \Delta_\varphi = 
                    \frac{\nu\delta}{M}\frac{2 - \ell }{(1-\ell )^3}\ell \left\{ (T-\delta)\frac{1 - e^{-b( 1-\ell)\delta}}{\delta} + 1 - \frac{1 - e^{-b(1-\ell)\delta}}{b (1- \ell) \delta} \right\}.
                  \end{equation}
                Supposing that $ b ( 1 - \ell ) T > 4,$ using moreover that $ \delta \le T/2,$ we obtain the lower bound 
                \begin{equation}\label{eq:bound_Hawkes}
                   \Delta_\varphi \geq C \frac{\nu \ell T}{M(1- \ell)^3}(1-e^{-b(1-\ell)\delta}) ,  
                \end{equation}
                for some constant not depending on the model parameters. 
            \end{itemize}
            We conclude by replacing the variances in the criterion \eqref{eq:20230127_1} of Lemma~\ref{lem:critgen} by their upper bounds in the approximations \eqref{eq:var_Hawkes_ind}-\eqref{eq:var_Hawkes}. We solve the resulting inequality with respect to the unknown $n$ and deduce the lower bound \eqref{eq:min_size_Hawkes} on the number of observations that ensures the good control of the {\color{black}error of second kind}.

\bibliographystyle{abbrv} 
\bibliography{biblio.bib}

\appendix
\section{Control of variances}\label{appendixa}
    This section is dedicated to the proof of Lemma~\ref{lem:critgen}. We start with a sketch of the proof and then detail 
    the more technical parts in the next subsections.
                    \subsection{Sketch of the Proof of Lemma~\ref{lem:critgen}}
            We easily show that given $\X_n$, the permuted statistics $\bT_n(\X^{\Pi_n}_n)$ is centered. Indeed, it follows from \eqref{ustat} that
            \[ \E\big[ \bT_n(\X^{\Pi_n}_n)\big|\X_n \big] = \frac{1}{n(n-1)}\sum_{i\neq j}\E\big[ \varphi(X^1_i,X^2_{\Pi_n(i)}) - \varphi(X^1_i,X^2_{\Pi_n(j)}) \big|\X_n\big] = 0 ,\]
            since the laws of $(i,\Pi_n(i))$ and $(i,\Pi_n(j))$ are equal. Given this property, 
            Kim et al. \cite{kim} showed that the type II error of the test $\ind_{\bT_n(\X_n)>q_{1-\alpha}(\X_n)}$ is controlled by $\beta$ if
            \begin{equation}\label{eq:kim}
                \Eo\left[ \bT_n(\X_n) \right] \geq \sqrt{\frac{2\Vo\left[ \bT_n(\X_n) \right]}{\beta}} + \sqrt{\frac{2\Eo\left[\Vo\left[\bT_n(\X^{\Pi_n}_n)\middle|\X_n \right] \right]}{\alpha\beta}} .
            \end{equation}
            Our aim is now to show that if the condition \eqref{eq:20230127_1} is satisfied for a constant $C>0$ and $n\geq 3/\sqrt{\alpha\beta}$, 
            then the criterion 
            \eqref{eq:kim} holds.
            We start with a presentation of the main steps of the proofs.
            
            \noindent \textbf{Step 1:} 
            First we prove 
            that the variance of the test statistics $\bT_n(\X_n)$ under the observed model is bounded by the one of $\fp(X_1,X_2)$ through the inequality
            \begin{equation}\label{bound_var}
                \Vo\big[ \bT_n(\X_n) \big] \leq \frac{4}{n}\Vo\big[ \fp(X_1,X_2) \big].
            \end{equation}

            \noindent \textbf{Step 2:} Using the fact that given $\X_n$ the distribution of $\bT_n(\X_n^{\Pi_n})$ is centered, we show that the expectation of the conditional variance of this permuted test statistics is mainly bounded by the variance of $\fp(X_1,X_2)$ under independence with a rest. This bound is given as follows 
            \begin{equation}\label{bound_pvar}
                \Eo\bigg(\Vo\big[ \bT_n(\X_n^{\Pi_n}) \big|\X_n\big] \bigg) \leq \frac{2}{n} \bigg\{\frac{2}{n}\Eo\big[ \big| \fp(X_1,X_2) \big|^2 \big] + \frac{n-2}{n}\Vz\big[ \fp(X_1,X_2) \big] \bigg\}.
            \end{equation}

            \noindent \textbf{Step 3:} 
            We deduce from the bounds \eqref{bound_var}-\eqref{bound_pvar} 
            that the right-hand term of (\ref{eq:kim}) is upper bounded by 
            \[ 4\sqrt{\frac{2}{\alpha\beta}}\left\{ \sqrt{\frac{\Vo\big[ \fp(X_1,X_2) \big]}{n}} + \sqrt{\frac{\Vz\big[ \fp(X_1,X_2) \big]}{n}} \right\} + \frac{2}{n}\sqrt{\frac{2}{\alpha\beta}}\E\big[ \fp(X_1,X_2) \big]. \]
            Using \eqref{eq:mean_ustat} and $n\geq 3/\sqrt{\alpha\beta}$, we deduce a constant $C>0$ for which the condition \eqref{eq:20230127_1} implies that the criterion \eqref{eq:kim} holds, implying that the {\color{black}error of second kind} is bounded by $\beta$.

    \subsection{Proof of Inequality \texorpdfstring{\eqref{bound_var}}{}}\label{app:var}
        This section is devoted to the proof of the bound
        \[ \Vo\left[ \bT_n(\X_n) \right] \leq \frac{4}{n}\Vo\left[ \fp(X_1,X_2) \right] . \]
        For that purpose, we deduce from \eqref{ustat} that
        \[
            \Vo\left[ \bT_n(\X_n) \right] = \frac{1}{n^2(n-1)^2} \sum_{\substack{i\neq j\\ k\neq l}}\Eo\left[\left\{ \fp(X_i,X_j) - \Eo\left[ \fp(X_1,X_2) \right] \right\}\left\{ \fp(X_k,X_l) - \Eo\left[ \fp(X_1,X_2) \right] \right\}\right]. 
        \]
        Thanks to the independence of the \((X_i)\), the terms in the previous sum with all different \(i, j, k, l\) are null. In addition, using Cauchy-Schwarz inequality, 
        \[
            \Eo\left[\left\{ \fp(X_i,X_j) - \Eo\left[ \fp(X_1,X_2) \right] \right\}\left\{ \fp(X_k,X_l) - \Eo\left[ \fp(X_1,X_2) \right] \right\}\right] \leq \Vo\left[ \fp(X_1,X_2) \right].
        \]
        The number of non zero terms is at most
        \[n^2(n-1)^2 - n(n-1)(n-2)(n-3) = n(n-1)(4n - 6),\]
        that ends the proof.

    \subsection{Proof of Ineq. \texorpdfstring{\eqref{bound_pvar}}{}}\label{app:pvar}
        Since $\bT_n(\X_n^{\Pi_n})$ is centered under the permutation law, its variance under this law is given by
        \begin{equation}
            \Vo\big[ \bT_n(\X^{\Pi_n}) \big| \X_n \big] = \frac{1}{n^2(n-1)^2} \sum_{\substack{i\neq j\\ k\neq l}}\Eo\bigg[\fp(X^{\Pi_n}_i,X^{\Pi_n}_j)\fp(X^{\Pi_n}_k,X^{\Pi_n}_l) \bigg|\X_n\bigg].
        \end{equation}
        Our aim is to compute precise bounds for the conditional expectations in this expression for $i\neq j$ and $k\neq l$. Let us first notice that
        \begin{multline}\label{esp_ijkl}
            \Eo\bigg[ \fp(X^{\Pi_n}_i,X^{\Pi_n}_j)\fp(X^{\Pi_n}_k,X^{\Pi_n}_l) \bigg|\X_n\bigg] 
            \\
            =\, \sum_{\substack{i',j',k',l'\\ i'\neq j'\\ k'\neq l'}}\eP(\Pi_n(i)=i', \Pi_n(j)=j', \Pi_n(k)=k', \Pi_n(l)=l') 
            \\
            \times\, \bigg\{ \varphi(X^1_i,X^2_{i'}) - \varphi(X^1_i,X^2_{j'}) \bigg\}\bigg\{ \varphi(X^1_k,X^2_{k'}) - \varphi(X^1_k,X^2_{l'}) \bigg\},
        \end{multline}
        then it suffices to check the different possible choices of $i,j,k,l$ and bound each of them thanks to the following lemma.
        \begin{lemma}
            For any $i,i',j',k,k',l'\in\{1,\cdots,n\}$ such that $i'\neq j'$ and $k'\neq l'$, we have the bound
            \begin{multline}\label{eq:bound_int}
                \Eo\bigg[ \bigg\{ \varphi(X^1_i,X^2_{i'}) - \varphi(X^1_i,X^2_{j'}) \bigg\}\bigg\{ \varphi(X^1_k,X^2_{k'}) - \varphi(X^1_k,X^2_{l'}) \bigg\} \bigg] 
                \\
                \leq 
                \begin{cases}
                    \Eo\big[ \big|\fp(X_1,X_2)\big|^2 \big] & \textrm{ if } i\in\{i',j'\} \textrm{ and } k\in\{k',l'\}, \vspace{0.20cm}
                    \\
                    \Ez\big[ \big|\fp(X_1,X_2)\big|^2 \big] & \textrm{ if } i\notin\{i',j'\} \textrm{ and } k\notin\{k',l'\}, \vspace{0.20cm}
                    \\
                    \frac{1}{2}\big\{\Eo\big[ \big|\fp(X_1,X_2)\big|^2 \big] + \Ez\big[ \big|\fp(X_1,X_2)\big|^2 \big]\big\} & \textrm{ otherwise.} \vspace{0.20cm}
                \end{cases}
            \end{multline}
        \end{lemma}
        
        This intermediate result directly follows from the inequality
        \begin{multline*}
            \Eo\bigg[ \bigg\{ \varphi(X^1_i,X^2_{i'}) - \varphi(X^1_i,X^2_{j'}) \bigg\}\bigg\{ \varphi(X^1_k,X^2_{k'}) - \varphi(X^1_k,X^2_{l'}) \bigg\} \bigg]
            \\
            \leq \frac{1}{2} \left\{ \Eo\left[ \big| \varphi(X^1_i,X^2_{i'}) - \varphi(X^1_i,X^2_{j'}) \big|^2 \right] + \Eo\left[ \left| \varphi(X^1_k,X^2_{k'}) - \varphi(X^1_k,X^2_{l'}) \right|^2 \right] \right\}
        \end{multline*}
        where we distinguished the cases $i\in\{i',j'\}$, $i\notin\{i',j'\}$, $k\in\{k',l'\}$ and $k\notin\{k',l'\}$. Returning to \eqref{esp_ijkl}, we have seven cases depending on the possible values of the integers $i,j,k,l$:
        \begin{itemize}
            \item[\textbf{Case 1}] We assume that $i=k$ and $j=l$. In this case, the sum in \eqref{esp_ijkl} is such that $i'=k'$ and $j'=l'$. It follows that 
            \begin{align*}
                \Eo\bigg[ \fp(X^{\Pi_n}_i,X^{\Pi_n}_j)\fp(X^{\Pi_n}_k,X^{\Pi_n}_l) \bigg|\X_n\bigg]
                 &= \sum_{i'\neq j'}\eP(\Pi_n(i)=i', \Pi_n(j)=j')\bigg\{ \varphi(X^1_i,X^2_{i'}) - \varphi(X^1_i,X^2_{j'}) \bigg\}^2 
                \\
               & = \frac{1}{n(n-1)} \sum_{i'\neq j'}\bigg\{ \varphi(X^1_i,X^2_{i'}) - \varphi(X^1_i,X^2_{j'}) \bigg\}^2
            \end{align*}
            and then thanks to \eqref{eq:bound_int} with $i=k, i'=k', j'=l'$,
            \begin{multline*}
            \Eo\bigg[ \fp(X^{\Pi_n}_i,X^{\Pi_n}_j)\fp(X^{\Pi_n}_k,X^{\Pi_n}_l) \bigg]
                \\
                 =\, \frac{1}{n(n-1)} \sum_{i'\neq j'}\E\bigg[\bigg\{ \varphi(X^1_i,X^2_{i'}) - \varphi(X^1_i,X^2_{j'}) \bigg\}^2\bigg] \vspace{0.20cm}
                \\
                \leq\, \frac{1}{n(n-1)} \sum_{i'\neq j'}\bigg\{ \Eo\big[ \big|\fp(X_1,X_2)\big|^2 \big]\ind_{i=i'\textrm{ or }i=j'} + \Ez\big[ \big|\fp(X_1,X_2)\big|^2 \big] \ind_{i\neq i'\textrm{ and }i\neq j'\}} \bigg\} \vspace{0.20cm}
                \\
                 =\, \frac{1}{n(n-1)} \bigg\{ 2(n-1)\Eo\big[ \big|\fp(X_1,X_2)\big|^2 \big] + (n-1)(n-2)\Ez\big[ \big|\fp(X_1,X_2)\big|^2 \big] \bigg\},
            \end{multline*}
            which finally gives
            \begin{equation}\label{case1}
                \Eo\bigg[ \fp(X^{\Pi_n}_i,X^{\Pi_n}_j)\fp(X^{\Pi_n}_k,X^{\Pi_n}_l) \bigg] \leq \frac{2}{n}\Eo\big[ \big|\fp(X_1,X_2)\big|^2 \big] + \frac{n-2}{n}\Ez\big[ \big|\fp(X_1,X_2)\big|^2 \big] .
            \end{equation}
            
            \item[\textbf{Case 2}] We now assume that $i=l$ and $j=k$. In this case, the sum in \eqref{esp_ijkl} is such that $i'=l'$ and $j'=k'$. We deduce that
                 \begin{multline*}
                 \Eo\bigg[ \fp(X^{\Pi_n}_i,X^{\Pi_n}_j)\fp(X^{\Pi_n}_k,X^{\Pi_n}_l) \bigg|\X_n\bigg] = 
                \\
                  \sum_{i'\neq j'}\eP(\Pi_n(i)=i', \Pi_n(j)=j')\bigg\{ \varphi(X^1_i,X^2_{i'}) - \varphi(X^1_i,X^2_{j'}) \bigg\}
                 \bigg\{ \varphi(X^1_j,X^2_{j'}) - \varphi(X^1_j,X^2_{i'}) \bigg\} 
                \\
                 = \frac{1}{n(n-1)} \sum_{i'\neq j'}\bigg\{ \varphi(X^1_i,X^2_{i'}) - \varphi(X^1_i,X^2_{j'}) \bigg\}\bigg\{ \varphi(X^1_j,X^2_{j'}) - \varphi(X^1_j,X^2_{i'}) \bigg\} ,
            \end{multline*}
             which implies thanks to \eqref{eq:bound_int} with $k=j, k'=j',l'=i'$ that
               \begin{eqnarray*}
               &&\Eo\bigg[ \fp(X^{\Pi_n}_i,X^{\Pi_n}_j)\fp(X^{\Pi_n}_k,X^{\Pi_n}_l) \bigg]  \leq\, \frac{1}{n(n-1)} \times 
                \\
               && \sum_{i'\neq j'}\bigg\{ \Eo\big[ \big|\fp(X_1,X_2)\big|^2 \big]\ind_{i\in\{i',j'\}}\ind_{j\in\{i',j'\}} + \Ez\big[ \big|\fp(X_1,X_2)\big|^2 \big] \ind_{i\notin\{i',j'\}}\ind_{j\notin\{i',j'\}} \vspace{0.20cm}
                \\
               && +\, \frac{1}{2}\left( \Eo\big[ \big|\fp(X_1,X_2)\big|^2 \big] + \Ez\big[ \big|\fp(X_1,X_2)\big|^2 \big] \right)\times \\
               &&\quad \quad  \quad \quad\quad \quad\quad \quad \quad \quad  \quad \quad \times \left( 1 - \ind_{i\in\{i',j'\}}\ind_{j\in\{i',j'\}} - \ind_{i\notin\{i',j'\}}\ind_{j\notin\{i',j'\}} \right) \bigg\} \vspace{0.20cm}
                \\
               && =\, \frac{1}{n(n-1)} \bigg\{ 2\Eo\big[ \big|\fp(X_1,X_2)\big|^2 \big] + (n-2)(n-3)\Ez\big[ \big|\fp(X_1,X_2)\big|^2 \big] \vspace{0.20cm}
                \\
               && +\, 2(n-2)\left( \Eo\big[ \big|\fp(X_1,X_2)\big|^2 \big] + \Ez\big[ \big|\fp(X_1,X_2)\big|^2 \big] \right) \bigg\}
            \end{eqnarray*}
               and finally gives
            \begin{equation}\label{case2}
                \Eo\bigg[ \fp(X^{\Pi_n}_i,X^{\Pi_n}_j)\fp(X^{\Pi_n}_k,X^{\Pi_n}_l) \bigg] \leq \frac{2}{n}\Eo\big[ \big|\fp(X_1,X_2)\big|^2 \big] + \frac{n-2}{n}\Ez\big[ \big|\fp(X_1,X_2)\big|^2 \big] .
            \end{equation}
            
            \item[\textbf{Case 3}] We assume that $i=k$ and $j\neq l$, then the sum in \eqref{esp_ijkl} is such that $i'=k'$ and $j'\neq l'$. We deduce that
                       \begin{eqnarray*}
               && \Eo\bigg[ \fp(X^{\Pi_n}_i,X^{\Pi_n}_j)\fp(X^{\Pi_n}_k,X^{\Pi_n}_l) \bigg|\X_n\bigg]
                \\
               && =\, \sum_{\substack{i',j',l'\\ \textrm{different}}}\eP(\Pi_n(i)=i', \Pi_n(j)=j',\Pi_n(l)=l') \vspace{0.20cm}
                \\
               && \times\,\bigg\{ \varphi(X^1_i,X^2_{i'}) - \varphi(X^1_i,X^2_{j'}) \bigg\}\bigg\{ \varphi(X^1_i,X^2_{i'}) - \varphi(X^1_i,X^2_{l'}) \bigg\} \vspace{0.20cm}
                \\
               &&=\, \frac{1}{n(n-1)(n-2)} \sum_{\substack{i',j',l'\\\textrm{different}}}\bigg\{ \varphi(X^1_i,X^2_{i'}) - \varphi(X^1_i,X^2_{j'}) \bigg\}\bigg\{ \varphi(X^1_i,X^2_{i'}) - \varphi(X^1_i,X^2_{l'}) \bigg\},
            \end{eqnarray*}
            and thanks to \eqref{eq:bound_int} with $k=i, k'=i'$, we have
                      \begin{eqnarray*}
               &&\Eo\bigg[ \fp(X^{\Pi_n}_i,X^{\Pi_n}_j)\fp(X^{\Pi_n}_k,X^{\Pi_n}_l) \bigg]
                \\
               && \leq\, \frac{1}{n(n-1)(n-2)} \sum_{\substack{i',j',l'\\\textrm{different}}}\bigg\{ \Eo\big[ \big|\fp(X_1,X_2)\big|^2 \big]\ind_{i\in\{i',j'\}}\ind_{i\in\{i',l'\}} \\
               && + \Ez\big[ \big|\fp(X_1,X_2)\big|^2 \big] \ind_{i\notin\{i',j'\}}\ind_{i\notin\{i',l'\}} 
                     \!  + \frac{1}{2}\! \left( \Eo\big[ \big|\fp(X_1,X_2)\big|^2 \big] + \Ez\big[ \big|\fp(X_1,X_2)\big|^2 \big] \right)\\
                     && \quad \quad  \quad  \quad \quad \quad \quad \quad \quad   \quad \quad \quad \quad \quad \times  \left( 1 - \ind_{i\in\{i',j'\}}\ind_{i\in\{i',l'\}} - \ind_{i\notin\{i',j'\}}\ind_{i\notin\{i',l'\}} \right) \bigg\} 
                \\
               && =\, \frac{1}{n(n-1)(n-2)} \bigg\{ (n-1)(n-2)\Eo\big[ \big|\fp(X_1,X_2)\big|^2 \big] \\
               &&  \quad \quad \quad \quad + (n-1)(n-2)(n-3)\Ez\big[ \big|\fp(X_1,X_2)\big|^2 \big] 
                \\
              && \quad \quad \quad \quad\quad \quad \quad \quad  +\, (n-1)(n-2)\left( \Eo\big[ \big|\fp(X_1,X_2)\big|^2 \big] + \Ez\big[ \big|\fp(X_1,X_2)\big|^2 \big] \right) \bigg\},
            \end{eqnarray*}
              which implies 
            \begin{equation}\label{case3}
                \Eo\bigg[ \fp(X^{\Pi_n}_i,X^{\Pi_n}_j)\fp(X^{\Pi_n}_k,X^{\Pi_n}_l) \bigg] \leq \frac{2}{n}\Eo\big[ \big|\fp(X_1,X_2)\big|^2 \big] 
                + \frac{n-2}{n}\Ez\big[ \big|\fp(X_1,X_2)\big|^2 \big] .
            \end{equation}
            
            \item[\textbf{Case 4}] We assume that $i\neq k$ and $j=l$, then the sum in \eqref{esp_ijkl} is such that $i'\neq k'$ and $j'=l'$. It follows that
                       \begin{eqnarray*}
               &&\Eo\bigg[ \fp(X^{\Pi_n}_i,X^{\Pi_n}_j)\fp(X^{\Pi_n}_k,X^{\Pi_n}_l) \bigg|\X_n\bigg]
                \\
                && =\, \sum_{\substack{i',j',k'\\ \textrm{different}}}\eP(\Pi_n(i)=i', \Pi_n(j)=j',\Pi_n(k)=k') \vspace{0.20cm}
                \\
               && \times\,\bigg\{ \varphi(X^1_i,X^2_{i'}) - \varphi(X^1_i,X^2_{j'}) \bigg\}\bigg\{ \varphi(X^1_k,X^2_{k'}) - \varphi(X^1_k,X^2_{j'}) \bigg\} \vspace{0.20cm}
                \\
               && =\, \frac{1}{n(n-1)(n-2)} \sum_{\substack{i',j',k'\\\textrm{different}}}\bigg\{ \varphi(X^1_i,X^2_{i'}) - \varphi(X^1_i,X^2_{j'}) \bigg\}\bigg\{ \varphi(X^1_k,X^2_{k'}) - \varphi(X^1_k,X^2_{j'}) \bigg\} ,
            \end{eqnarray*}
                    and then
            \begin{multline}\label{case4}
                \Eo\bigg[ \fp(X^{\Pi_n}_i,X^{\Pi_n}_j)\fp(X^{\Pi_n}_k,X^{\Pi_n}_l) \bigg]
                \\
                = \frac{1}{n(n-1)(n-2)} \sum_{\substack{i',j',k'\\\textrm{different}}}\Eo\bigg[\bigg\{ \varphi(X^1_i,X^2_{i'}) - \varphi(X^1_i,X^2_{j'}) \bigg\}\bigg\{ \varphi(X^1_k,X^2_{k'}) - \varphi(X^1_k,X^2_{j'}) \bigg\} \bigg].
            \end{multline}
            
            \item[\textbf{Case 5}] We assume here that $i\neq l$ and $j=k$, then the sum in \eqref{esp_ijkl} is such that $i'\neq l'$ and $j'=k'$. Then,
                       \begin{eqnarray*}
              && \Eo\bigg[ \fp(X^{\Pi_n}_i,X^{\Pi_n}_j)\fp(X^{\Pi_n}_k,X^{\Pi_n}_l) \bigg|\X_n\bigg]
                \\
               && =\, \sum_{\substack{i',j',l'\\ \textrm{different}}}\eP(\Pi_n(i)=i', \Pi_n(j)=j',\Pi_n(l)=l') \vspace{0.20cm}
                \\
                && \times\,\bigg\{ \varphi(X^1_i,X^2_{i'}) - \varphi(X^1_i,X^2_{j'}) \bigg\}\bigg\{ \varphi(X^1_j,X^2_{j'}) - \varphi(X^1_j,X^2_{l'}) \bigg\} \vspace{0.20cm}
                \\
                && =\, \frac{1}{n(n-1)(n-2)} \sum_{\substack{i',j',l'\\\textrm{different}}}\bigg\{ \varphi(X^1_i,X^2_{i'}) - \varphi(X^1_i,X^2_{j'}) \bigg\}\bigg\{ \varphi(X^1_j,X^2_{j'}) - \varphi(X^1_j,X^2_{l'}) \bigg\}
            \end{eqnarray*}
           which finally gives
            \begin{multline}\label{case5}
                \Eo\bigg[ \fp(X^{\Pi_n}_i,X^{\Pi_n}_j)\fp(X^{\Pi_n}_k,X^{\Pi_n}_l) \bigg]
                \\
                = \frac{1}{n(n-1)(n-2)} \sum_{\substack{i',j',l'\\\textrm{different}}}\Eo\bigg[\bigg\{ \varphi(X^1_i,X^2_{i'}) - \varphi(X^1_i,X^2_{j'}) \bigg\}\bigg\{ \varphi(X^1_j,X^2_{j'}) - \varphi(X^1_j,X^2_{l'}) \bigg\} \bigg].
            \end{multline}
            \begin{remark}\label{rmk:case45}
                The sum of the expectations under \textbf{Case 4} given by \eqref{case4} and \textbf{Case 5} given by \eqref{case5} is null.
            \end{remark}
            
            \item[\textbf{Case 6}] We now assume that $i=l$ and $j\neq k$, then the sum in \eqref{esp_ijkl} is such that $i'=l'$ and $j'\neq k'$. We deduce that
                       \begin{eqnarray*}
              && \Eo\bigg[ \fp(X^{\Pi_n}_i,X^{\Pi_n}_j)\fp(X^{\Pi_n}_k,X^{\Pi_n}_l) \bigg|\X_n\bigg]
                \\
               && =\, \sum_{\substack{i',j',k'\\ \textrm{different}}}\eP(\Pi_n(i)=i', \Pi_n(j)=j',\Pi_n(k)=k') \vspace{0.20cm}
                \\
              && \times\,\bigg\{ \varphi(X^1_i,X^2_{i'}) - \varphi(X^1_i,X^2_{j'}) \bigg\}\bigg\{ \varphi(X^1_k,X^2_{k'}) - \varphi(X^1_k,X^2_{i'}) \bigg\} \vspace{0.20cm}
                \\
              && =\, \frac{1}{n(n-1)(n-2)} \sum_{\substack{i',j',k'\\\textrm{different}}}\bigg\{ \varphi(X^1_i,X^2_{i'}) - \varphi(X^1_i,X^2_{j'}) \bigg\}\bigg\{ \varphi(X^1_k,X^2_{k'}) - \varphi(X^1_k,X^2_{i'}) \bigg\}
            \end{eqnarray*}
                 and then, thanks to \eqref{eq:bound_int} with $l'=i'$, we have
             \begin{eqnarray*}
             && \Eo\bigg[ \fp(X^{\Pi_n}_i,X^{\Pi_n}_j)\fp(X^{\Pi_n}_k,X^{\Pi_n}_l) \bigg]
                \\
               && \leq\, \frac{1}{n(n-1)(n-2)} \sum_{\substack{i',j',k'\\\textrm{different}}}\bigg\{ \Eo\big[ \big|\fp(X_1,X_2)\big|^2 \big]\ind_{i\in\{i',j'\}}\ind_{k\in\{k',i'\}} \\
               && \quad \quad \quad \quad\quad \quad \quad \quad\quad \quad \quad \quad\quad \quad \quad \quad+ \Ez\big[ \big|\fp(X_1,X_2)\big|^2 \big] \ind_{i\notin\{i',j'\}}\ind_{k\notin\{k',i'\}}                 \\
             && +\, \frac{1}{2}\left( \Eo\big[ \big|\fp(X_1,X_2)\big|^2 \big] + \Ez\big[ \big|\fp(X_1,X_2)\big|^2 \big] \right)\\
             && \quad \quad \quad \quad \quad \quad \quad \quad \quad \quad \quad \quad \left( 1 - \ind_{i\in\{i',j'\}}\ind_{k\in\{k',i'\}} - \ind_{i\notin\{i',j'\}}\ind_{k\notin\{k',i'\}} \right) \bigg\}. 
            \end{eqnarray*}
               In addition, we have
            \[ \sum_{\substack{i',j',k'\\ \textrm{different}}}\ind_{i\in\{i',j'\}}\ind_{k\in\{k',i'\}} = \sum_{\substack{i',j',k'\\ \textrm{different}}}\left( \ind_{i'=i}\ind_{k'=k} + \ind_{j'=i}\ind_{k'=k} + \ind_{j'=i}\ind_{i'=k} \right) = 3(n-2)  \]
            and
             \begin{eqnarray*}
                &&\sum_{\substack{i',j',k'\\ \textrm{different}}} \ind_{i\notin\{i',j'\}}\ind_{k\notin\{k',i'\}} = \sum_{\substack{i',j',k'\\ \textrm{different}}} \ind_{i'\notin\{i,k\}}\ind_{j'\neq i}\ind_{k'\neq k} \vspace{0.20cm}
                \\
               && =\, \sum_{\substack{i',j',k'\\ \textrm{different}}} \ind_{i'\notin\{i,k\}}\ind_{j'\neq i}\ind_{k'\neq k} \left( \ind_{j'=k}\ind_{k'=i} + \ind_{j'=k}\ind_{k'\neq i} + \ind_{j'\neq k}\ind_{k'=i} + \ind_{j'\neq k}\ind_{k'\neq i} \right) \vspace{0.20cm}
                \\
             && =\, \sum_{i'}\ind_{i'\notin\{i,k\}} + \sum_{i'\neq k'}\ind_{i',k'\notin\{i,k\}} + \sum_{i'\neq j'}\ind_{i',j'\notin\{i,k\}} + \sum_{\substack{i',j',k'\\ \textrm{different}}}\ind_{i',j',k'\notin\{i,k\}} \vspace{0.20cm}
                \\
               &&=\, n-2 +2(n-2)(n-3) + (n-2)(n-3)(n-4)
            \end{eqnarray*}
            that gives
            \[ \sum_{\substack{i',j',k'\\ \textrm{different}}} \ind_{i\notin\{i',j'\}}\ind_{k\notin\{k',i'\}} = (n-2)\big( 1 + (n-2)(n-3) \big). \]
            It follows that
                \begin{eqnarray*}
               && \Eo\bigg[ \fp(X^{\Pi_n}_i,X^{\Pi_n}_j)\fp(X^{\Pi_n}_k,X^{\Pi_n}_l) \bigg]
                \\
              && \leq\, \frac{1}{n(n-1)(n-2)} \bigg\{ 3(n-2)\Eo\big[ \big|\fp(X_1,X_2)\big|^2 \big] \\
              &&\quad \quad \quad \quad  + (n-2)\big( 1 + (n-2)(n-3) \big)\Ez\big[ \big|\fp(X_1,X_2)\big|^2 \big] 
                \\
             &&\quad \quad \quad \quad  +\, (n-2)(2n-5)\left( \Eo\big[ \big|\fp(X_1,X_2)\big|^2 \big] + \Ez\big[ \big|\fp(X_1,X_2)\big|^2 \big] \right) \bigg\} ,
            \end{eqnarray*}
               and then
            \begin{equation}\label{case6}
                \Eo\bigg[ \fp(X^{\Pi_n}_i,X^{\Pi_n}_j)\fp(X^{\Pi_n}_k,X^{\Pi_n}_l) \bigg] \leq \frac{2}{n}\Eo\big[ \big|\fp(X_1,X_2)\big|^2 \big] + \frac{n-2}{n}\Ez\big[ \big|\fp(X_1,X_2)\big|^2 \big] .
            \end{equation}
            
            \item[\textbf{Case 7}] We finally assume that $i\notin\{k,l\}$ and $j\notin\{k,l\}$, then the sum in \eqref{esp_ijkl} is such that $i'\notin\{k',l'\}$ and $j'\notin\{k',l'\}$. It follows that
               \begin{eqnarray*}
              &&\Eo\bigg[ \fp(X^{\Pi_n}_i,X^{\Pi_n}_j)\fp(X^{\Pi_n}_k,X^{\Pi_n}_l) \bigg|\X_n\bigg]
                \\
             &&=\, \sum_{\substack{i',j',k',l'\\ \textrm{different}}}\eP(\Pi_n(i)=i', \Pi_n(j)=j',\Pi_n(k)=k',\Pi_n(l)=l') \vspace{0.20cm}
                \\
             && \times\,\bigg\{ \varphi(X^1_i,X^2_{i'}) - \varphi(X^1_i,X^2_{j'}) \bigg\}\bigg\{ \varphi(X^1_k,X^2_{k'}) - \varphi(X^1_k,X^2_{l'}) \bigg\} \vspace{0.20cm}
                \\
          &&=\, \frac{1}{n(n-1)(n-2)(n-3)} \sum_{\substack{i',j',k',l'\\\textrm{different}}}\bigg\{ \varphi(X^1_i,X^2_{i'}) - \varphi(X^1_i,X^2_{j'}) \bigg\}\\
          && \quad \quad \quad \quad\quad \quad \quad \quad\quad \quad \quad \quad\quad \quad \quad \quad\quad \quad \quad \quad\quad  \bigg\{ \varphi(X^1_k,X^2_{k'}) - \varphi(X^1_k,X^2_{l'}) \bigg\} .
            \end{eqnarray*}
            Distinguishing the cases $i'=i, j'=i$ and $i',j'\neq i$, we get
            \begin{eqnarray*}
               && \Eo\bigg[ \fp(X^{\Pi_n}_i,X^{\Pi_n}_j)\fp(X^{\Pi_n}_k,X^{\Pi_n}_l) \bigg|\X_n\bigg]  = \frac{1}{n(n-1)(n-2)(n-3)}  \times 
                \\
               &&\bigg[\sum_{\substack{i',j',k',l'\\\textrm{different}}}\bigg\{ \varphi(X^1_i,X^2_{i'}) - \varphi(X^1_i,X^2_{j'}) \bigg\}\bigg\{ \varphi(X^1_k,X^2_{k'}) - \varphi(X^1_k,X^2_{l'}) \bigg\}\ind_{i'=i}
                \\
                &&+ \sum_{\substack{i',j',k',l'\\\textrm{different}}}\bigg\{ \varphi(X^1_i,X^2_{i'}) - \varphi(X^1_i,X^2_{j'}) \bigg\}\bigg\{ \varphi(X^1_k,X^2_{k'}) - \varphi(X^1_k,X^2_{l'}) \bigg\}\ind_{j'=i}
                \\
                &&+ \sum_{\substack{i',j',k',l'\\\textrm{different}}}\bigg\{ \varphi(X^1_i,X^2_{i'}) - \varphi(X^1_i,X^2_{j'}) \bigg\}\bigg\{ \varphi(X^1_k,X^2_{k'}) - \varphi(X^1_k,X^2_{l'}) \bigg\} \ind_{i',j'\neq i} \bigg]
                \\
                &&= \frac{1}{n(n-1)(n-2)(n-3)} \times  \\
                && \bigg[\sum_{\substack{j',k',l'\\\textrm{different}\\j',k',l'\neq i}}\bigg\{ \varphi(X^1_i,X^2_{i}) - \varphi(X^1_i,X^2_{j'}) \bigg\}\bigg\{ \varphi(X^1_k,X^2_{k'}) - \varphi(X^1_k,X^2_{l'}) \bigg\}
                \\
                && + \sum_{\substack{i',k',l'\\\textrm{different}\\ i',k',l'\neq i}}\bigg\{ \varphi(X^1_i,X^2_{i'}) - \varphi(X^1_i,X^2_{i}) \bigg\}\bigg\{ \varphi(X^1_k,X^2_{k'}) - \varphi(X^1_k,X^2_{l'}) \bigg\}
                \\
                &&+ \sum_{\substack{i',j',k',l'\\\textrm{different}}}\bigg\{ \varphi(X^1_i,X^2_{i'}) - \varphi(X^1_i,X^2_{j'}) \bigg\}\bigg\{ \varphi(X^1_k,X^2_{k'}) - \varphi(X^1_k,X^2_{l'}) \bigg\} \ind_{i',j'\neq i} \bigg].
            \end{eqnarray*}
            The two first sums compensates each other and we get
            \begin{multline*}
                \Eo\bigg[ \fp(X^{\Pi_n}_i,X^{\Pi_n}_j)\fp(X^{\Pi_n}_k,X^{\Pi_n}_l) \bigg|\X_n\bigg]
                \\
                = \sum_{\substack{i',j',k',l'\\\textrm{different}}}\bigg\{ \varphi(X^1_i,X^2_{i'}) - \varphi(X^1_i,X^2_{j'}) \bigg\}\bigg\{ \varphi(X^1_k,X^2_{k'}) - \varphi(X^1_k,X^2_{l'}) \bigg\} \ind_{i',j'\neq i}.
            \end{multline*}
            A similar reduction of this sum holds by distinguishing the cases: $i'=k, j'=k$ and $i',j'\neq k$; $k'=i, l'=i$ and $k',l'\neq i$; $k'=k, l'=k$ and $k',l'\neq k$. It follows that
                     \begin{eqnarray*}
             && \Eo\bigg[ \fp(X^{\Pi_n}_i,X^{\Pi_n}_j)\fp(X^{\Pi_n}_k,X^{\Pi_n}_l) \bigg|\X_n\bigg] =\, \frac{1}{n(n-1)(n-2)(n-3)} \times 
                \\
              &&  \sum_{\substack{i',j',k',l'\\\textrm{different}\\i',j',k',l'\notin\{i,k\}}}\bigg\{ \varphi(X^1_i,X^2_{i'}) - \varphi(X^1_i,X^2_{j'}) \bigg\}\bigg\{ \varphi(X^1_k,X^2_{k'}) - \varphi(X^1_k,X^2_{l'}) \bigg\},
            \end{eqnarray*}
             and we finally get
            \begin{equation}\label{case7}
                \Eo\bigg[ \fp(X^{\Pi_n}_i,X^{\Pi_n}_j)\fp(X^{\Pi_n}_k,X^{\Pi_n}_l) \bigg] = 0.
            \end{equation}
        \end{itemize}
        We deduce from \eqref{case1}---\eqref{case7} and Remark~\ref{rmk:case45} that
                \begin{eqnarray*}
               && \Eo\bigg\{\Vo\big[ T_{n}(\X^{\Pi_n}_n) \big|\,\X_n \big] \bigg\} = \frac{1}{n^2(n-1)^2} \\
               && \quad \sum_{\substack{i\neq j\\ k\neq l}}\bigg\{ \Eo^{\textbf{Case 1}}\bigg[ \fp(X^{\Pi_n}_i,X^{\Pi_n}_j)\fp(X^{\Pi_n}_k,X^{\Pi_n}_l) \bigg]\ind_{i=k,j=l} 
                \\
                && \quad +\, \Eo^{\textbf{Case 2}}\bigg[ \fp(X^{\Pi_n}_i,X^{\Pi_n}_j)\fp(X^{\Pi_n}_k,X^{\Pi_n}_l) \bigg]\ind_{i=l,j=k}  \vspace{0.20cm}
                \\
               &&\quad +\, \Eo^{\textbf{Case 3}}\bigg[ \fp(X^{\Pi_n}_i,X^{\Pi_n}_j)\fp(X^{\Pi_n}_k,X^{\Pi_n}_l) \bigg]\ind_{i=k,j\neq l}  \vspace{0.20cm}
                \\
              &&\quad  +\, \Eo^{\textbf{Case 6}}\bigg[ \fp(X^{\Pi_n}_i,X^{\Pi_n}_j)\fp(X^{\Pi_n}_k,X^{\Pi_n}_l) \bigg]\ind_{i=l,j\neq k} \bigg\}  ,
            \end{eqnarray*}
        and then
            \begin{eqnarray*}
    &&   \Eo\bigg\{\Vo\big[ T_{n}(\X^{\Pi_n}_n) \big|\,\X_n \big] \bigg\} \leq\, \\
    && \frac{2}{n^2(n-1)^2}\bigg\{ n(n-1) \bigg[ \frac{2}{n}\Eo\big[ \big|\fp(X_1,X_2)\big|^2 \big] + \frac{n-2}{n}\Ez\big[ \big|\fp(X_1,X_2)\big|^2 \big] \bigg] \vspace{0.20cm} 
                \\
                && +\, n(n-1)(n-2)\bigg[ \frac{2}{n}\Eo\big[ \big|\fp(X_1,X_2)\big|^2 \big] + \frac{n-2}{n}\Ez\big[ \big| \fp(X_1,X_2) \big|^2 \big] \bigg] \bigg\} \vspace{0.20cm}
                \\
                &&=\, \frac{2}{n}\bigg\{ \frac{2}{n} \Eo\big[ \big|\fp(X_1,X_2)\big|^2 \big] + \frac{n-2}{n}\Ez\big[ \big|\fp(X_1,X_2)\big|^2 \big] \bigg\},
            \end{eqnarray*}
        which ends the proof.
        
        \subsection{Proof of Lemma~\ref{lem:critgen}}\label{app:proof_critgen} 
                This proof consists in showing that the condition \eqref{eq:kim} -- that allows to control the type II error by a given $\beta\in(0,1)$ -- is satisfied under the criterion of Lemma~\ref{lem:critgen}. For this purpose, let us first notice that from \eqref{bound_var} and \eqref{bound_pvar} proved in Appendix~\ref{app:var} and \ref{app:pvar} respectively, we obtain
        \begin{align*}
            & \sqrt{\frac{2\Vo\left[ \bT_n(\X_n) \right]}{\beta}} + \sqrt{\frac{2\Eo\left[\Vo\left[\bT_n(\X^{\Pi_n}_n)\middle|\X_n \right] \right]}{\alpha\beta}}
            \\
            &\quad \leq 2\sqrt{\frac{2}{\beta}}\sqrt{\frac{\Vo\big[ \fp(X_1,X_2) \big]}{n}} + \frac{2}{\sqrt{\alpha\beta}}\sqrt{\frac{2}{n^2}\Eo\big[ |\fp(X_1,X_2)|^2 \big] + \frac{n-2}{n^2}\Vz\big[ \fp(X_1,X_2) \big]}
            \\
            &\quad \leq 2\sqrt{\frac{2}{\alpha\beta}}\left( \sqrt{\alpha} + \frac{1}{\sqrt{n}} \right)\sqrt{\frac{\Vo\big[ \fp(X_1,X_2) \big]}{n}} + \frac{2}{\sqrt{\alpha\beta}}\sqrt{\frac{\Vz\big[ \fp(X_1,X_2) \big]}{n}} 
            \\
            &\qquad + \frac{2}{n}\sqrt{\frac{2}{\alpha\beta}}\E\big[ \fp(X_1,X_2) \big], 
        \end{align*}
        where we used the decomposition 
        $\Eo\big[ |\fp(X_1,X_2)|^2 \big] = \Vo\big[ \fp(X_1,X_2) \big] + (\Eo\big[ \fp(X_1,X_2) \big])^2$ and the inequality $\sqrt{a+b}\leq\sqrt{a} + \sqrt{b}$ for any $a,b\geq 0$. 
        Since $\sqrt{\alpha}+1/\sqrt{n}\leq 2$, it follows that
        \begin{align*}
            & \sqrt{\frac{2\Vo\left[ \bT_n(\X_n) \right]}{\beta}} + \sqrt{\frac{2\Eo\left[\Vo\left[\bT_n(\X^{\Pi_n}_n)\middle|\X_n \right] \right]}{\alpha\beta}}
            \\
            &\quad \leq 4\sqrt{\frac{2}{\alpha\beta}}\left\{ \sqrt{\frac{\Vo\big[ \fp(X_1,X_2) \big]}{n}} + \sqrt{\frac{\Vz\big[ \fp(X_1,X_2) \big]}{n}} \right\} + \frac{2}{n}\sqrt{\frac{2}{\alpha\beta}}\E\big[ \fp(X_1,X_2) \big].
        \end{align*}
        In particular, if $n\geq 3/\sqrt{\alpha\beta}$, then
        \begin{align*}
            & \sqrt{\frac{2\Vo\left[ \bT_n(\X_n) \right]}{\beta}} + \sqrt{\frac{2\Eo\left[\Vo\left[\bT_n(\X^{\Pi_n}_n)\middle|\X_n \right] \right]}{\alpha\beta}}
            \\
            &\quad \leq 4\sqrt{\frac{2}{\alpha\beta}}\left\{ \sqrt{\frac{\Vo\big[ \fp(X_1,X_2) \big]}{n}} + \sqrt{\frac{\Vz\big[ \fp(X_1,X_2) \big]}{n}} \right\} + \frac{2\sqrt{2}}{3}\E\big[ \fp(X_1,X_2) \big].
        \end{align*}
        We deduce that the condition \eqref{eq:kim} is satisfied if this upper bound is lower than $\E\big[ \bT_n(\X_n) \big] = \E\big[ \fp(X_1,X_2) \big]$, i.e. if
        \[ \E\big[ \fp(X_1,X_2) \big] \geq \frac{4}{1-\frac{2\sqrt{2}}{3}}\sqrt{\frac{2}{\alpha\beta}}\left\{ \sqrt{\frac{\Vo\big[ \fp(X_1,X_2) \big]}{n}} + \sqrt{\frac{\Vz\big[ \fp(X_1,X_2) \big]}{n}} \right\}. \]
        This corresponds to the condition \eqref{eq:20230127_1} with $C = \frac{4\sqrt{2}}{1-(2/3)\sqrt{2}}$.

\section{The jittering injection model}\label{app:Poisson}
    For the sake of simplicity, in what follows we use the same notation for a Poisson process and its sequence of jump times.    
    \subsection{Proof of Eq. \texorpdfstring{\eqref{eq:mean_sjit}}{}}\label{subsectionB1}
        We use the notations of Model~\ref{mod:jittering}. Notice that
        \begin{equation}\label{eq:decphi}
            \varphi\left(X^1,X^2\right) = \varphi\left(Y^1, Y^2\right) + \varphi\left(Y^1, Z^2\right) + \varphi\left( Z^1, Y^2\right) 
            + \varphi\left( Z^1, Z^2\right).
        \end{equation}
        Since $Y^1,Y^2,Z^1$ and $(\xi_k)_{k\in\mathbb{Z}}$ are independent, the expectations of $\varphi( Y^1, Y^2)$, $\varphi( Y^1, Z^2)$ and 
        $\varphi( Z^1, Y^2)$ under the observed and the independent models do not change. In addition, we have  
        \begin{align*}
            \varphi\left(Z^1,Z^2\right) & = \int_0^T\ind_{|\xi_{Z^1_u}|\leq\delta}\ind_{u+\xi_{Z^1_{u}}\in[0,T]}\dd Z^1_u 
            \\
            &~ + \int_{\Delta_2}\ind_{|u_1-u_2-\xi_{Z^1_{u_2}}|\leq\delta}
            \ind_{u_1\in[0,T]}
            \ind_{u_2+\xi_{Z^1_{u_2}}\in[0,T]} \dd Z^1_{u_1}\dd Z^1_{u_2},
        \end{align*}
        where we recall that $ \Delta_2$ has been introduced in \eqref{eq:Deltaell} above. 
        Using Lemma~\ref{lem:Poisson}  to calculate the expectation of the second integral, we get  
        \[
        \Eo \left[\int_{\Delta_2}\ind_{|u_1-u_2-\xi_{Z^1_{u_2}}|\leq\delta}
            \ind_{u_1\in[0,T]}
            \ind_{u_2+\xi_{Z^1_{u_2}}\in[0,T]} \dd Z^1_{u_1}\dd Z^1_{u_2} \right] = \Ez \left[ \varphi ( Z^1, Z^2 ) \right]. 
        \]
            Therefore,
        \begin{eqnarray}\label{eq:w}
            \Eo\left[ \varphi(X^1_1,X^2_1) \right] - \Ez\left[ \varphi(X^1_1,X^2_1) \right] &  = \Eo\left[ \int_0^T\ind_{|\xi_{Z^1_{u}}|\leq\delta}\ind_{u+\xi_{Z^1_{u}}\in[0,T]}\dd Z^1_u \right]
            \nonumber \\
            & = \eta^1\Eo\left[ \ind_{|\xi|\leq\delta} \int_0^T\ind_{u+\xi\in[0,T]}\dd u \right]\nonumber 
            \\
            & = \eta^1 \E\big[(T-|\xi|)\ind_{|\xi|\leq\delta} \big]. 
        \end{eqnarray}
        The lower bound \eqref{eq:lower_bound_jit} follows by using the inequality $\delta\leq T/2$ which implies that $|\xi|\ind_{|\xi|\leq\delta}\leq \frac{T}{2}\ind_{|\xi|\leq\delta}$.
    
    \subsection{Proof of Eq. \texorpdfstring{\eqref{eq:var_jit_ind}}{} and \texorpdfstring{\eqref{eq:var_jit}}{}} 
    Consider two counting processes \(N^1\) and \(N^2\) that almost surely do not jump together. Then we have
    \begin{equation}\label{eq:coinc}
        \varphi(N^1,N^2) = \int_0^T\left[ N^2_{(u+\delta)\wedge T} - N^2_{(u-\delta)_+} \right]\dd N^1_u.
    \end{equation}

    \begin{remark}
    Notice that the integrand appearing in \eqref{eq:coinc} is not predictable. Therefore we shall use this representation only in the case when $ N^1 $ and $ N^2$ are independent. 
    \end{remark}

    We deduce the following result
    \begin{lemma}\label{lem:mom2_indep}
        Let $N^1,N^2,N^3$ be three independent Poisson processes on $[0,T]$ with intensities $\mu^1,\mu^2$ and $\mu^3$ respectively. We introduce the quantities 
        \begin{equation}
            I^T_{\delta} = 2\delta T - \delta^2 \textrm{ and }J^T_{\delta} = 4\delta^2T-\frac{10}{3}\delta^3 .
        \end{equation}
        Then we have
        \begin{align}
       \E \big[ \varphi ( N^1, N^2 ) \big] &= \mu^1 \mu^2 I_\delta^T,\\    
            \label{eq:mom2_indep}
          \Vo \big[   \varphi(N^1,N^2)   \big]  &=  \mu^1\mu^2 \left[ (\mu^1+\mu^2)J^T_{\delta} + I^T_{\delta} \right], \\
        \label{eq:cross_indep}
            \E\big[ \varphi(N^1,N^2)\varphi(N^1,N^3) \big] &= \mu^1\mu^2\mu^3J^T_{\delta} + (\mu^1)^2\mu^2\mu^3(I^T_{\delta})^2.
        \end{align}
        As a consequence, if $\mu^3=\mu^2$, then we obtain
        \begin{equation}\label{eq:Vardiff}
            \E\left\{ \left[ \varphi(N^1,N^2) - \varphi(N^1,N^3) \right]^2 \right\} = 2\mu^1\mu^2\left[ I^T_{\delta} + \mu^1J^T_{\delta} \right].
        \end{equation}
    \end{lemma}

    \begin{remark}
        The quantities $I^T_{\delta}$ and $J^T_{\delta}$ correspond to the mean and variance of the conditional expectation $\E[\varphi(N^1,N^2)|N^2]$ in the particular case $\mu^1 = \mu^2 = 1$.
    \end{remark}
    
    \begin{proof}
    \textbf{Step 1.} We start with some preliminary remarks. Consider the measurable functions defined for $(u,v)\in\R^2$ by
    \begin{align}
        w(u,v) & = \ind_{|u-v|\leq\delta}\ind_{u\in[0,T]}\ind_{v\in[0,T]},
        \\
        \tw(u) & := \int_{\R}w(u,v)\dd v = [(u+\delta)\wedge T - (u-\delta)_+]\ind_{u\in[0,T]}.
    \end{align}
    Let  $N$ be a Poisson process with intensity $1.$ Then we have that 
    \[ 
    \int_0^T \E \big[ N_{(u+\delta)\wedge T} - N_{(u-\delta)_+}\big] \dd u  = \int_R \tw (u) \dd u =  I_\delta^T. 
    \]
    In particular, this implies that, since $N^1$ and $N^2$ are independent, 
    \[
    \E \big[ \varphi ( N^1, N^2 ) \big] = \mu^1  \int_0^T \E \big[ N^2_{(u+\delta)\wedge T} - N^2_{(u-\delta)_+} \Big]   \dd u  = \mu^1 \mu^2 \int_0^T \tw(u) \dd u = \mu^1 \mu^2 I_\delta^T.
    \]
    We shall also use throughout the remainder of this proof that 
    \[
    J_\delta^T =  \int_0^T \tw (u)^2 \dd u  ,
    \]
    which can be easily checked.
    
      \textbf{Step 2.} \eqref{eq:coinc} implies 
        \[ \E\big[ \varphi^2(N^1,N^2) \big] = \E\left[ \left( \int_0^T\left[ N^2_{(u+\delta)\wedge T} - N^2_{(u-\delta)_+} \right]\dd N^1_u \right)^2 \right] .\]
        Conditioning on $N^2$, it follows that
        \begin{eqnarray*}
            &&\E\big[ \varphi^2(N^1,N^2) \big] \nonumber \\
            &&= \mu^1\int_0^T\E\left\{\left[ N^2_{(u+\delta)\wedge T} - N^2_{(u-\delta)_+} \right]^2\right\}\dd u + (\mu^1)^2\E\left[ \left( \int_0^T\left[ N^2_{(u+\delta)\wedge T} - N^2_{(u-\delta)_+} \right]\dd u \right)^2 \right] \nonumber 
            \\
             &&= \mu^1\int_0^T\left\{ (\mu^2)^2 \tw^2(u)   + \mu^2 \tw(u)  \right\}\dd u \nonumber \\
             &&  \quad \quad \quad \quad \quad 
             +(\mu^1)^2\int_0^T\int_0^T\E\left\{ \left[ N^2_{(u+\delta)\wedge T} - N^2_{(u-\delta)_+} \right]\left[ N^2_{(v+\delta)\wedge T} - N^2_{(v-\delta)_+} \right] \right\}\dd u\dd v .
            \end{eqnarray*}
         According to Step 0,  the first expression appearing above can be rewritten as
         \[
          \mu^1(\mu^2)^2J^T_{\delta} + \mu^1\mu^2I^T_{\delta}.
         \] 
         Moreover, studying the last expression appearing above and distinguishing the cases when the two intervals $ [ u- \delta, u+ \delta ] , [v - \delta, v+ \delta ] $ (intersected with $ [0, T ] $) do or do not overlap, we get    
            \begin{multline*}
            (\mu^1)^2\int_0^T\int_0^T\E\left\{ \left[ N^2_{(u+\delta)\wedge T} - N^2_{(u-\delta)_+} \right]\left[ N^2_{(v+\delta)\wedge T} - N^2_{(v-\delta)_+} \right] \right\}\dd u\dd v\\
=             (\mu^1)^2\int_0^T\int_0^T\bigg\{ (\mu^2)^2 \tw(u)  \tw(v)  
              + \mu^2\int_0^T\ind_{|v-s|\leq\delta}\ind_{|u-s|\leq\delta}\dd s \bigg\}\dd u\dd v
            \\
             =   (\mu^1\mu^2)^2(I^T_{\delta})^2 + (\mu^1)^2\mu^2J^T_{\delta}.
            \end{multline*}
Putting everything together, this implies that 
\[
\E\big[ \varphi^2(N^1,N^2) \big]         = \mu^1\mu^2\left[ (\mu^1+\mu^2)J^T_{\delta} + I^T_{\delta} + \mu^1\mu^2(I^T_{\delta})^2 \right].
\]

  \textbf{Step 3.} Using \eqref{eq:coinc} again, we have
        \begin{multline*}
        \E\big[ \varphi(N^1,N^2)\varphi(N^1,N^3) \big] \\
        = \E\left[ \left( \int_0^T\left[ N^2_{(u+\delta)\wedge T} - N^2_{(u-\delta)_+} \right]\dd N^1_u \right)\left( \int_0^T\left[ N^3_{(u+\delta)\wedge T} - N^3_{(u-\delta)_+} \right]\dd N^1_u \right)  \right] ,
        \end{multline*} 
        such that, conditioning on $N^2$ and $N^3$,
        \begin{multline*}
            \E\big[ \varphi(N^1,N^2)\varphi(N^1,N^3) \big] \\
            = \mu^1\int_0^T\E\left\{ \left[ N^2_{(u+\delta)\wedge T} - N^2_{(u-\delta)_+} \right]\left[ N^3_{(u+\delta)\wedge T} - N^3_{(u-\delta)_+} \right] \right\}\dd u
            \\
             + (\mu^1)^2\E\left\{ \left( \int_0^T\left[ N^2_{(u+\delta)\wedge T} - N^2_{(u-\delta)_+} \right]\dd u \right)\left( \int_0^T\left[ N^3_{(u+\delta)\wedge T} - N^3_{(u-\delta)_+} \right]\dd u \right) \right\}.
        \end{multline*}
        Since $N_2$ and $N_3$ are independent, we deduce that
        \begin{multline*}
            \E\big[ \varphi(N^1,N^2)\varphi(N^1,N^3) \big] \\
            = \mu^1\int_0^T\left( \E\left[ N^2_{(u+\delta)\wedge T} - N^2_{(u-\delta)_+} \right]\right)\left(\E\left[ N^3_{(u+\delta)\wedge T} - N^3_{(u-\delta)_+} \right] \right)\dd u
            \\
            ~ + (\mu^1)^2\E\left( \int_0^T\left[ N^2_{(u+\delta)\wedge T} - N^2_{(u-\delta)_+} \right]\dd u \right)\E\left( \int_0^T\left[ N^3_{(u+\delta)\wedge T} - N^3_{(u-\delta)_+} \right]\dd u \right) 
            \\
             = \mu^1\mu^2\mu^3J^T_{\delta} + (\mu^1)^2\mu^2\mu^3(I^T_{\delta})^2.
        \end{multline*}
        Taking in particular $\mu^3=\mu^2$, it follows that
        \begin{align*}
            \E\left\{ \left[ \varphi(N^1,N^2) - \varphi(N^1,N^3) \right]^2 \right\} & = 2\left\{ \E\left[ \varphi^2(N^1,N^2) \right] - \E\left[ \varphi(N^1,N^2)\varphi(N^1,N^3) \right] \right\}
            \\
            & = 2\mu^1\mu^2\big[ \mu^1J^T_{\delta} + I^T_{\delta} \big].
        \end{align*}
    \end{proof}
    
    \noindent Let us come back to the jittering injection Poisson model given by Model~\ref{mod:jittering}. Starting with the variance under the independent model, we have 
    \begin{align*}
        \Vz\big[ \fp(X_1,X_2) \big] = \Ez\left[ \left| \varphi(X^1_1,X^2_1) - \varphi(X^1_1,X^2_2) \right|^2 \right]
    \end{align*}
    that implies thanks to \eqref{eq:Vardiff} in Lemma~\ref{lem:mom2_indep} that
    \begin{equation}\label{eq:var_Poisson_indep}
        \Vz\big[ \fp(X_1,X_2) \big] = 2(\eta^1+\lambda^1)(\eta^1+\lambda^2)\left[ I^T_{\delta} + (\eta^1+\lambda^1)J^T_{\delta} \right].
    \end{equation}
    The bound \eqref{eq:var_jit_ind} follows thanks to the inequalities $\eta^1\leq\lamb = \max\{\lambda^1,\lambda^2\}$, $I^T_{\delta}\leq 2\delta T$ and $J^T_{\delta}\leq 4\delta^2T$. Further, the decomposition \eqref{eq:decphi} gives
    \begin{equation}
        \fp(X_1,X_2) = \Du + \Dd + \Dt + \Dq 
    \end{equation}
    where we set
    \begin{align*}
        \Du & = \varphi(Y^1_1,Y^2_1) - \varphi(Y^1_1,Y^2_2),
        \\
        \Dd & = \varphi(Y^1_1,Z^2_1) - \varphi(Y^1_1,Z^2_2),
        \\
        \Dt & = \varphi(Z^1_1,Y^2_1) - \varphi(Z^1_1,Y^2_2),
        \\
        \Dq & = \varphi(Z^1_1,Z^2_1) - \varphi(Z^1_1,Z^2_2).
    \end{align*}
    This implies that
    \begin{equation}\label{eq:dec_mom2}
        \Eo\big[ \left|\fp(X_1,X_2)\right|^2 \big] = \Eo\big[ |\Du|^2\big] + \Eo\big[|\Dd|^2\big] + \Eo\big[|\Dt|^2\big] + \Eo\big[|\Dq|^2\big] + 2\sum_{i<j}\Eo\big[ \Di\Dj \big].
    \end{equation}
    We deduce from \eqref{eq:Vardiff} in Lemma~\ref{lem:mom2_indep} that
    \begin{align*}
        \Eo\big[|\Du|^2\big] & = 2\lambda^1\lambda^2\big[ \lambda^1J^T_{\delta} + I^T_{\delta} \big]
        \\
        \Eo\big[|\Dd|^2\big] & = 2\lambda^1\eta^1\big[ \lambda^1J^T_{\delta} + I^T_{\delta} \big]
        \\
        \Eo\big[|\Dt|^2\big] & = 2\lambda^2\eta^1\big[ \eta^1 J^T_{\delta} + I^T_{\delta} \big].
    \end{align*}
    In addition, conditioning on $Y^1_1$ gives
    \begin{equation*}
        \Eo\big[ \Du\Dd \big] = 0 ,
    \end{equation*}
    and using \eqref{eq:cross_indep} in Lemma~\ref{lem:mom2_indep}, we obtain, due to the independence of $ Y^1 $ and $Y^2 $ and of $Z^1 $ and $ Y^2 , $
    \begin{multline*}
        \Eo\big[ \Du\Dt \big]  = \Eo\left[ \left(\varphi(Y^1_1,Y^2_1) - \varphi(Y^1_1,Y^2_2)\right)\left(\varphi(Z^1_1,Y^2_1) - \varphi(Z^1_1,Y^2_2)\right) \right]
        \\
         = \Eo\left[\varphi(Y^1_1,Y^2_1)\varphi(Z^1_1,Y^2_1)\right] - 2\Eo\left[\varphi(Y^1_1,Y^2_1)\right]\Eo\left[\varphi(Z^1_1,Y^2_2)\right] + \E\left[\varphi(Y^1_1,Y^2_2)\varphi(Z^1_1,Y^2_2)\right]
        \\
         = \eta^1\lambda^1\lambda^2J^T_{\delta} + \eta^1\lambda^1(\lambda^2)^2(I^T_{\delta})^2 - 2\eta^1\lambda^1(\lambda^2)^2(I^T_{\delta})^2 + \eta^1\lambda^1\lambda^2J^T_{\delta} + \eta^1\lambda^1(\lambda^2)^2(I^T_{\delta})^2 ,
    \end{multline*}
    which implies that
    \begin{equation*}
        \Eo\big[ \Du\Dt \big] = 2\eta^1\lambda^1\lambda^2J^T_{\delta}. 
    \end{equation*}
    Further, it directly follows from the independence of $(Y^1_1,Y^2_1,Y^2_2)$ and $(Z^1_1,Z^2_1,Z^2_2)$ that
    \begin{equation*}
        \Eo\big[ \Du\Dq \big] = 0.
    \end{equation*}
    We also have
    \begin{multline*}
            \Eo\left[ \Dd\Dt \right]  = \Eo\left[ \left(\varphi(Y^1_1,Z^2_1) - \varphi(Y^1_1,Z^2_2)\right)\left(\varphi(Z^1_1,Y^2_1) - \varphi(Z^1_1,Y^2_2)\right) \right]
        \\
         = \Eo\left[ \varphi(Y^1_1,Z^2_1)\left(\varphi(Z^1_1,Y^2_1) - \varphi(Z^1_1,Y^2_2)\right) \right] - \Eo\left[ \varphi(Y^1_1,Z^2_2)\left(\varphi(Z^1_1,Y^2_1) - \varphi(Z^1_1,Y^2_2)\right) \right].
    \end{multline*}
    Conditioning on $(Z^1_1,Z^2_1)$ in the first expectation and using the independence of $(Y^1_1,Z^2_2)$ and $(Z^1_1,Y^2_1,Y^2_2)$ to handle the second expression, we get
    \begin{equation*}
        \Eo\left[ \Dd\Dt \right] = 0.
    \end{equation*}
    Similarly, notice that
    \[ \Eo\left[ \Dt\Dq \right] = \Eo\left[ \left( \varphi(Z^1_1,Y^2_1) - \varphi(Z^1_1,Y^2_2) \right)\left( \varphi(Z^1_1,Z^2_1) - \varphi(Z^1_1,Z^2_2) \right) \right]. \]
    By conditioning on $(Z^1_1,Z^2_1, Z^2_2)$, we get
    \begin{equation*}
        \Eo\left[ \Dt\Dq \right] = 0.
    \end{equation*}
    Then Eq.\,\eqref{eq:dec_mom2} becomes
    \begin{multline}\label{eq:dec_mom2_int}
        \Eo\left[ \left|\fp(X_1,X_2)\right|^2 \big] = 2\big[ (\eta^1+\lambda^1)(\eta^1+\lambda^2) - (\eta^1)^2 \right]I^T_{\delta} 
        \\
        + 2\big[ (\eta^1+\lambda^1)^2(\eta^1+\lambda^2) - (\eta^1)^2(\eta^1+2\lambda^1) \big]J^T_{\delta} + \Eo\big[|\Dq|^2\big] + 2\Eo\big[ \Dd\Dq \big].
    \end{multline}
    The computation of the expectations of $\Dd\Dq$ and $|\Dq|^2$ needs another representation of the number of coincidences, different of what we have used in \eqref{eq:coinc}, since now the two processes that are involved are not independent any more. 
    Recalling the definition of $ w $ and of $ \tw $ above, we have: 
    \begin{lemma}
        \begin{multline}\label{eq:newcoinc}
            \varphi(Z^1_1,Z^2_1) = \int_{\R} w(u,u+\xi_{Z^1_1(u-)+1})\dd Z^1_1(u) 
            \\
            + \int_{\Delta_2} w(u_1,u_2+\xi_{Z^1_1(u_2-)+1}) \dd Z^1_1(u_1) \dd Z^1_1(u_2).
        \end{multline}
        In particular  
        \begin{equation}\label{eq:expectationvarphiz}
        \Eo   \left[  \varphi(Z^1_1,Z^2_1) \right] = \eta^1 \Eo \left[ (T- |\xi | ) \ind_{ | \xi | \le \delta }  \right] + ( \eta^1 )^2 I_\delta^T, 
        \end{equation}
        and 
        \begin{multline}\label{eq:condDq}
            \Eo\left[ \Dq \big|Z^1_1,Z^2_1\right] = \int_{\R} \left[ w(u, u+\xi_{Z^1_1(u-)+1}) - \eta^1\tw(u)\right]\dd Z^1_1(u) 
            \\
            + \int_{\Delta_2}w(u_1,u_2+\xi_{Z^1_1(u_2-)+1})\dd Z^1_1(u_1)\dd Z^1_1(u_2).
        \end{multline}
    \end{lemma}
    
    \begin{proof}
        Equation \eqref{eq:newcoinc} 
        follows from the representation 
        \[ \varphi(Z^1_1,Z^2_1) = \int_{\R^2}w(u_1,u_2+\xi_{Z^1_1(u_2-)+1})\dd Z^1_1(u_1)\dd Z^1_1(u_2), \]
        where we distinguished the case for which $u_1=u_2$ from the case $ u_1 \neq u_2$. 
         Using Lemma~\ref{lem:Poisson} and the arguments of section~\ref{subsectionB1}, the equation \eqref{eq:expectationvarphiz} follows immediately.

        Finally, we have
        \begin{align*}
            \Eo\left[\Dq\big|Z^1_1,Z^2_1\right] & = \Eo\left[ \varphi(Z^1_1,Z^2_1) - \varphi(Z^1_1,Z^2_2) \big| Z^1_1,Z^2_1 \right]
            \\
            & = \varphi(Z^1_1,Z^2_1) - \Eo\left[\varphi(Z^1_1,Z^2_2) \big|Z^1_1 \right]
            \\
            & = \varphi(Z^1_1,Z^2_1) - \eta^1\int_{\R}\tw(u)\dd Z^1_1(u),
        \end{align*}
        which implies \eqref{eq:condDq} thanks to \eqref{eq:newcoinc}.
    \end{proof}
    
    \noindent Using \eqref{eq:cross_indep} in Lemma~\ref{lem:mom2_indep} and equation \eqref{eq:expectationvarphiz}, we get
    \begin{align*}
        \Eo\left[ \Dd\Dq \right] & = \Eo\left[ \left( \varphi(Y^1_1,Z^2_1) - \varphi(Y^1_1,Z^2_2) \right)\Dq \right] 
        \\
        & = \Eo\left[ \varphi(Y^1_1,Z^2_1)\Dq \right] - \Eo\left[\varphi(Y^1_1,Z^2_2)\right]\Eo\left[\varphi(Z^1_1,Z^2_1)\right] \\
        & \quad \quad \quad \quad \quad \quad \quad \quad \quad \quad \quad \quad + \Eo\left[\varphi(Y^1_1,Z^2_2)\varphi(Z^1_1,Z^2_2)\right]
        \\
        & = \Eo\left[ \varphi(Y^1_1,Z^2_1)\Dq \right] - \lambda^1\eta^1 I^T_{\delta}\left( (\eta^1)^2I^T_{\delta} + \eta^1\E\big[ (T-|\xi|)\ind_{|\xi|\leq\delta} \big] \right) 
        \\
        &~ + (\eta^1)^2\lambda^1J^T_{\delta} + (\eta^1)^3\lambda^1(I^T_{\delta})^2
    \end{align*}
    and then,
    \begin{equation}\label{eq:D24}
        \Eo\left[ \Dd\Dq \right] = \Eo\left[ \varphi(Y^1_1,Z^2_1)\Dq \right] + (\eta^1)^2\lambda^1J^T_{\delta} - \lambda^1(\eta^1)^2I^T_{\delta}\E\left[ (T-|\xi|)\ind_{|\xi|\leq\delta} \right].
    \end{equation}
    Conditioning on $(Z^1_1,Z^2_1)$ and using \eqref{eq:condDq}, we obtain 
    \begin{align*}
        & \Eo\left[ \varphi(Y^1_1,Z^2_1)\Dq \right] 
        \\
        & = \Eo\left\{ \Eo\left[\varphi(Y^1_1,Z^2_1)\big|Z^1_1 \right]\Eo\left[\Dq\big|Z^1_1,Z^2_1\right] \right\}
        \\
        & = \lambda^1\Eo\bigg\{ \left( \int_{\R}\tw(u+\xi_{Z^1_1(u-)+1})\dd Z^1_1(u) \right)
        \\
        &~ \times \left( \int_{\R} \left[ w(u, u+\xi_{Z^1_1(u-)+1}) - \eta^1\tw(u)\right]\dd Z^1_1(u)\right. \\
        &\left. \quad \quad \quad + \int_{\Delta_2}w(u_1,u_2+\xi_{Z^1_1(u_2-)+1})\dd Z^1_1(u_1)\dd Z^1_1(u_2) \right) \bigg\}
        \\
        & = \lambda^1\E\bigg\{ \int_{\R^2}\left[ w(u_1, u_1+\xi_{Z^1_1(u_1-)}) - \eta^1\tw(u_1)\right]\tw(u_2+\xi_{Z^1_1(u_2-)})\dd Z^1_1(u_1)\dd Z^1_1(u_2)
        \\
        &~ + \int_{\R} \left( \int_{\Delta_2}w(u_1,u_2+\xi_{Z^1_1(u_2-)+1})\tw(u_3+\xi_{Z^1_1(u_3-)+1})\dd Z^1_1(u_1)\dd Z^1_1(u_2) \right) \dd Z^1_1(u_3) \bigg\},
    \end{align*}
    
    which implies, by distinguishing common jump times, 
    \begin{align*}
        & \Eo\left[ \varphi(Y^1_1,Z^2_1)\Dq \right] = \lambda^1\E\bigg\{ \int_{\R}\left[ w(u, u+\xi_{Z^1_1(u-)+1}) - \eta^1\tw(u)\right]\tw(u+\xi_{Z^1_1(u-)+1})\dd Z^1_1(u)
        \\
        &~ + \int_{\Delta_2}\left[ w(u_1, u_1+\xi_{Z^1_1(u_1-)+1}) - \eta^1\tw(u_1)\right]\tw(u_2+\xi_{Z^1_1(u_2-)+1})\dd Z^1_1(u_1)\dd Z^1_1(u_2)
        \\
        &~ + \int_{\Delta_2}w(u_1,u_2+\xi_{Z^1_1(u_2-)+1})\tw(u_1+\xi_{Z^1_1(u_1-)+1})\dd Z^1_1(u_1)\dd Z^1_1(u_2)
        \\
        &~ + \int_{\Delta_2}w(u_1,u_2+\xi_{Z^1_1(u_2-)+1})\tw(u_2+\xi_{Z^1_1(u_2-)+1})\dd Z^1_1(u_1)\dd Z^1_1(u_2)
        \\
        &~ + \int_{\Delta_3}w(u_1,u_2+\xi_{Z^1_1(u_2-)+1})\tw(u_3+\xi_{Z^1_1(u_3-)+1})\dd Z^1_1(u_1)\dd Z^1_1(u_2)\dd Z^1_1(u_3) \bigg\}.
    \end{align*}
    
    Then it follows from Lemma~\ref{lem:Poisson} that
    \begin{align*}
        \Eo\left[ \varphi(Y^1_1,Z^2_1)\Dq \right] & = \lambda^1\E\bigg\{ \eta^1 \int_{\R}\left[ w(u, u+\xi) - \eta^1\tw(u)\right]\tw(u+\xi)\dd u
        \\
        &~ + (\eta^1)^2\int_{\R^2}\left[ w(u_1, u_1+\xi_1) - \eta^1\tw(u_1)\right]\tw(u_2+\xi_{2})\dd u_1\dd u_2
        \\
        &~ + (\eta^1)^2\int_{\R^2}w(u_1,u_2+\xi_1)\tw(u_1+\xi_{2})\dd u_1\dd u_2
        \\
        &~ + (\eta^1)^2\int_{\R^2}w(u_1,u_2+\xi)\tw(u_2+\xi)\dd u_1\dd u_2
        \\
        &~ + (\eta^1)^3\int_{\R^3}w(u_1,u_2+\xi_{2})\tw(u_3+\xi_{3})\dd u_1\dd u_2\dd u_3 \bigg\}
        \\
        & = \lambda^1\E\bigg\{ \eta^1 \int_{\R}\left[ w(u, u+\xi) - \eta^1\tw(u)\right]\tw(u+\xi)\dd u 
        \\
        &~ + (\eta^1)^2I^T_{\delta}\left[ \int_{\R}w(u_1,u_1+\xi_1)\dd u_1 - \eta^1 I^T_{\delta} \right]
        \\
        &~ + (\eta^1)^2\int_{\R}\tw(u_1)\tw(u_1+\xi_2)\dd u_1 + (\eta^1)^2J^T_{\delta} + (\eta^1)^3(I^T_{\delta})^2 \bigg\}
    \end{align*}
    which finally gives
    \begin{align*}
        \Eo\left[ \varphi(Y^1_1,Z^2_1)\Dq \right] & = \lambda^1(\eta^1)^2J^T_{\delta} + \lambda^1(\eta^1)^2I^T_{\delta}\E\big[ (T-|\xi|)\ind_{|\xi|\leq\delta} \big] 
        \\
        &~ + \lambda^1\eta^1\E\bigg\{ \int_{\R}w(u, u+\xi)\tw(u+\xi)\dd u \bigg\}.
    \end{align*}
    Using this equation in \eqref{eq:D24}, we obtain
    \begin{equation}\label{eq:DdDq}
        \Eo\left[ \Dd\Dq \right] = 2\lambda^1(\eta^1)^2J^T_{\delta} + \lambda^1\eta^1\E\bigg\{ \int_{\R}w(u, u-\xi)\tw(u)\dd u \bigg\}.
    \end{equation}
    Let us now focus on the square moment of $\Dq.$ We have the decomposition
    \begin{equation}
        \Eo\left[ |\Dq|^2 \right] = \Eo\left[ \left|\Dq - \Eo\big[\Dq\big|Z^1_1,Z^2_1\big]\right|^2 \right] + \Eo\left[ \left|\E\left(\Dq\big|Z^1_1,Z^2_1\right)\right|^2 \right].
    \end{equation}
    Noticing that
    \[ \Dq - \Eo\big[\Dq\big|Z^1_1,Z^2_1\big] = \eta^1\int_{\R}\tw(u)\dd Z^1_1(u) - \varphi(Z^1_1,Z^2_2), \]
    it follows that
    \begin{align*}
    \Eo\left[ \left|\Dq - \Eo\big[\Dq\big|Z^1_1,Z^2_1\big]\right|^2 \right] & = \Eo\left\{ \left| \int_{\R}\left[ \eta^1\tw(u) - \int_{\R}w(u,v)\dd Z^2_2(v)\right]\dd Z^1_1(u) \right|^2 \right\} 
    \\
    & = \eta^1\Eo\left\{ \int_{\R}\left[ \eta^1\tw(u) - \int_{\R}w(u,v)\dd Z^2_2(v) \right]^2\dd u \right\}
    \\
    &~ + (\eta^1)^2\Eo\left\{ \left| \eta^1 I^T_{\delta} - \int_{\R}\tw(u)\dd Z^2_2(v) \right|^2 \right\} .
    \end{align*}
In the above expression, both squares involve a centered random variable. Developing the integral with respect to $Z^2_2$ in both cases, distinguishing once more common jumps, and finally using that $ w^2 (u,v)  = w (u,v), $     
    we end up with 
    \begin{equation}\label{eq:cond_centered}
        \Eo\left[ \left|\Dq - \Eo\big[\Dq\big|Z^1_1,Z^2_1\big]\right|^2 \right] = (\eta^1)^2I^T_{\delta} + (\eta^1)^3J^T_{\delta}.
    \end{equation}
    Further, we deduce from \eqref{eq:condDq} that 
    \begin{align*}
        & \Eo\left[ \left|\Eo\big[\Dq\big|Z^1_1,Z^2_1\big]\right|^2 \right] 
        \\
        &\quad = \Eo\bigg\{ \bigg| \int_{\R} \left[ w(u, u+\xi_{Z^1_1(u-)+1}) - \eta^1\tw(u)\right]\dd Z^1_1(u)\\
        & \quad \quad \quad \quad \quad \quad \quad \quad \quad \quad \quad \quad \quad + \int_{\Delta_2}w(u_1,u_2+\xi_{Z^1_1(u_2-)+1})\dd Z^1_1(u_1)\dd Z^1_1(u_2) \bigg|^2\bigg\}
        \\
        &\quad = \Eo\bigg\{ \bigg| \int_{\R} \left[ w(u, u+\xi_{Z^1_1(u-)+1}) - \eta^1\tw(u)\right]\dd Z^1_1(u) \bigg|^2 
        \\
        &\quad~ + 2\int_{\R}\Big[ \int_{\Delta_2}w(u_1,u_2+\xi_{Z^1_1(u_2-)+1})\left[ w(u_3, u_3+\xi_{Z^1_1(u_3-)+1}) - \eta^1\tw(u_3)\right]\dd Z^1_1(u_1)\\
        & \quad \quad \quad \quad \quad \quad \quad \quad \quad \quad \quad \quad \quad \quad \quad \quad \quad \quad \quad \quad \quad \quad \quad \quad \quad \quad \dd Z^1_1(u_2)\Big] \dd Z^1_1(u_3)
        \\
        &\quad~ + \bigg| \int_{\Delta_2}w(u_1,u_2+\xi_{Z^1_1(u_2-)+1})\dd Z^1_1(u_1)\dd Z^1_1(u_2) \bigg|^2\bigg\}.
    \end{align*}
    
    By distinguishing the common jumps and using Lemma~\ref{lem:Poisson} and once more that $ w^2 = w$, we obtain
    \begin{align*}
        & \Eo\bigg\{ \bigg| \int_{\R} \left[ w(u, u+\xi_{Z^1_1(u-)+1}) - \eta^1\tw(u)\right]\dd Z^1_1(u) \bigg|^2\bigg\} 
        \\
        & \quad = \E\bigg\{ \int_{\R} \left[ w(u, u+\xi_{Z^1_1(u-)+1}) - \eta^1\tw(u)\right]^2\dd Z^1_1(u)
        \\
        & \quad~ + \int_{\Delta_2} \left[ w(u_1, u_1+\xi_{Z^1_1(u_1-)+1}) - \eta^1\tw(u_1)\right]\left[ w(u_2, u_2+\xi_{Z^1_1(u_2-)+1}) - \eta^1\tw(u_2)\right]\\
        & \quad \quad \quad \quad \quad \quad \quad \quad \quad \quad \quad \quad \quad \quad \quad \quad \quad \quad \quad \quad \quad \quad \quad \quad \quad \quad \quad \dd Z^1_1(u_1)\dd Z^1_1(u_2) \bigg\}
        \\
        & \quad = \E\bigg\{ \eta^1\int_{\R} \left[ w(u, u+\xi) - \eta^1\tw(u)\right]^2\dd u
        \\
        & \quad~ + (\eta^1)^2\int_{\R^2} \left[ w(u_1, u_1+\xi_1) - \eta^1\tw(u_1)\right]\left[ w(u_2, u_2+\xi_2) - \eta^1\tw(u_2)\right]\dd u_1\dd u_2 \bigg\}
        \\
        &\quad = \eta^1\E\bigg\{ \int_{\R}w(u,u+\xi)\dd u - 2\eta^1\int_{\R}w(u,u+\xi)\tw(u)\dd u + (\eta^1)^2J^T_{\delta} \bigg\}
        \\
        &\quad~ + (\eta^1)^2\bigg\{ \left( \E\int_{\R}w(u,u+\xi)\dd u \right)^2 - 2\eta^1 I^T_{\delta}\E\int_{\R}w(u,u+\xi)\dd u + (\eta^1)^2(I^T_{\delta})^2  \bigg\}
    \end{align*}
    and therefore 
    \begin{multline}\label{eq:dec_term1}
        \Eo\bigg\{ \bigg| \int_{\R} \left[ w(u, u+\xi_{Z^1_1(u-)+1}) - \eta^1\tw(u)\right]\dd Z^1_1(u) \bigg|^2\bigg\} \\
        = (\eta^1)^3J^T_{\delta} + (\eta^1)^4(I^T_{\delta})^2 + (\eta^1)^2\left( \E\big[ (T-|\xi|)_{|\xi|\leq\delta} \big] \right)^2
        + \eta^1(1 - 2(\eta^1)^2I^T_{\delta})\E\big[ (T-|\xi|)_{|\xi|\leq\delta} \big] \\
        - 2(\eta^1)^2\E\int_{\R}w(u,u+\xi)\tw(u)\dd u.
    \end{multline}
    Similarly, 
    \begin{align*}
        & \Eo\bigg\{ \int_{\R} \Big[ \int_{\Delta_2}w(u_1,u_2+\xi_{Z^1_1(u_2-)+1})\left[ w(u_3, u_3+\xi_{Z^1_1(u_3-)+1}) - \eta^1\tw(u_3)\right]\dd Z^1_1(u_1)\\
        & \quad \quad \quad \quad \quad \quad \quad \quad \quad \quad\quad \quad\quad \quad\quad \quad\quad \quad \quad\quad \quad \quad \quad \quad \quad \quad \dd Z^1_1(u_2) \Big] \dd Z^1_1(u_3) \bigg\}
        \\
        &~ = \Eo\bigg\{ \int_{\Delta_2}w(u_1,u_2+\xi_{Z^1_1(u_2-)+1})\left[ w(u_1, u_1+\xi_{Z^1_1(u_1-)+1}) - \eta^1\tw(u_1)\right]\dd Z^1_1(u_1)\dd Z^1_1(u_2)
        \\
        &\quad + \int_{\Delta_2}w(u_1,u_2+\xi_{Z^1_1(u_2-)+1})\left[ w(u_2, u_2+\xi_{Z^1_1(u_2-)+1}) - \eta^1\tw(u_2)\right]\dd Z^1_1(u_1)\dd Z^1_1(u_2)
        \\
        &\quad + \int_{\Delta_3}w(u_1,u_2+\xi_{Z^1_1(u_2-)+1})\left[ w(u_3, u_3+\xi_{Z^1_1(u_3-)+1}) - \eta^1\tw(u_3)\right]\dd Z^1_1(u_1) \\
        & \quad \quad \quad \quad \quad \quad \quad \quad\quad \quad\quad \quad\quad \quad\quad \quad\quad \quad \quad\quad \quad\quad \quad \quad \quad \quad \dd Z^1_1(u_2)\dd Z^1_1(u_3) \bigg\}
        \\
        &~ = \Eo\bigg\{ (\eta^1)^2\int_{\R^2}w(u_1,u_2+\xi_2)\left[ w(u_1, u_1+\xi_1) - \eta^1\tw(u_1)\right]\dd u_1\dd u_2
        \\
        &\quad + (\eta^1)^2\int_{\R^2}w(u_1,u_2+\xi)\left[ w(u_2, u_2+\xi) - \eta^1\tw(u_2)\right]\dd u_1\dd u_2
        \\
        &\quad + (\eta^1)^3\int_{\R^3}w(u_1,u_2+\xi_2)\left[ w(u_3, u_3+\xi_3) - \eta^1\tw(u_3)\right]\dd u_1\dd u_2\dd u_3 \bigg\} ,
    \end{align*}
    from which we deduce that
    \begin{align*}
        & \Eo\bigg\{ \int_{\R}\int_{\Delta_2}w(u_1,u_2+\xi_{Z^1_1(u_2-)+1})\left[ w(u_3, u_3+\xi_{Z^1_1(u_3-)+1}) - \eta^1\tw(u_3)\right]\dd Z^1_1(u_1)\\
        & \quad \quad \quad \quad \quad \quad \quad \quad \quad \quad\quad \quad\quad \quad\quad \quad\quad \quad \quad\quad \quad \quad \quad \quad \quad \quad \dd Z^1_1(u_2)\dd Z^1_1(u_3) \bigg\}
        \\
        &~ = (\eta^1)^2\Eo\bigg\{ \int_{\R}\tw(u)w(u,u+\xi)\dd u - \eta^1 J^T_{\delta} + \int_{\R}\tw(u+\xi)\left[ w(u, u+\xi) - \eta^1\tw(u)\right]\dd u
        \\
        &\quad \quad \quad \quad \quad+ \eta^1 I^T_{\delta}\left[ \int_{\R}w(u, u+\xi)\dd u - \eta^1 I^T_{\delta}\right] \bigg\}
    \end{align*}
    and then
    \begin{multline}\label{eq:dec_term2}
        \Eo\bigg\{ \int_{\R}\int_{\Delta_2}w(u_1,u_2+\xi_{Z^1_1(u_2-)+1})\left[ w(u_3, u_3+\xi_{Z^1_1(u_3-)+1}) - \eta^1\tw(u_3)\right]\dd Z^1_1(u_1)\\
         \quad \quad \quad \quad \quad \quad \quad \quad \quad \quad\quad \quad\quad \quad\quad \quad\quad \quad \quad\quad \quad \quad \quad \quad \quad \quad \dd Z^1_1(u_2)\dd Z^1_1(u_3) \bigg\}
        \\
        = -(\eta^1)^3J^T_{\delta} - (\eta^1)^4(I^T_{\delta})^2 + (\eta^1)^3 I^T_{\delta}\E\big[ (T-|\xi|)\ind_{|\xi|\leq\delta} \big] 
        \\
        + (\eta^1)^2\E\int_{\R}\tw(u)\big[ w(u,u+\xi) + w(u,u-\xi) - \eta^1\tw(u+\xi) \big]\dd u.
    \end{multline}
    The last term we need to compute is
    \begin{align*}
        & \Eo\bigg\{\bigg| \int_{\Delta_2}w(u_1,u_2+\xi_{Z^1_1(u_2-)+1})\dd Z^1_1(u_1)\dd Z^1_1(u_2) \bigg|^2\bigg\}
        \\
        &~ = \Eo\bigg\{ \int_{\substack{(u_1,u_2)\in\Delta_2\\(u_3,u_4)\in\Delta_2}}  w(u_1,u_2+\xi_{Z^1_1(u_2-)+1})w(u_3,u_4+\xi_{Z^1_1(u_4-)+1})\dd Z^1_1(u_1)\dd Z^1_1(u_2)\\
        & \quad \quad \quad \quad \quad \quad \quad \quad \quad \quad\quad \quad\quad \quad\quad \quad\quad \quad \quad\quad \quad \quad \quad \quad \quad \quad \dd Z^1_1(u_3)\dd Z^1_1(u_4) \bigg\}.
    \end{align*}
    By distinguishing again the cases $u_i=u_j$ and $u_i\neq u_j$ for $j=1,\cdots,4$, we split this quantity in seven terms that can be explicitly computed. 
    
    The first term consists of taking the stochastic integral on the sub-domain $\{(u_1,u_2,u_3,u_4) : u_1=u_3, u_2=u_4\}$. The corresponding expectation is given by
    \begin{align*}
        K_1 & = \Eo\bigg\{ \int_{\Delta_2}  w^2(u_1,u_2+\xi_{Z^1_1(u_2-)+1})\dd Z^1_1(u_1)\dd Z^1_1(u_2) \bigg\}
        \\
        & = (\eta^1)^2\E\int_{\R^2}w^2 (u_1,u_2+\xi)\dd u_1\dd u_2 = (\eta^1)^2\E\int_{\R^2}w (u_1,u_2+\xi)\dd u_1\dd u_2 
        \\
        & = (\eta^1)^2I^T_{\delta},
    \end{align*}
    since $w^2 = w.$
    
    The second term consists of taking the stochastic integral on the sub-domain $\{(u_1,u_2,u_3,u_4) : u_1=u_4, u_2=u_3\}$. The corresponding expectation is given by
    \begin{align*}
        K_2 & = \Eo\bigg\{ \int_{\Delta_2}  w(u_1,u_2+\xi_{Z^1_1(u_2-)+1})w(u_2,u_1+\xi_{Z^1_1(u_1-)+1})\dd Z^1_1(u_1)\dd Z^1_1(u_2) \bigg\}
        \\
        & = (\eta^1)^2\Eo\bigg\{ \int_{\R^2}  w(u_1,u_2+\xi_2)w(u_2,u_1+\xi_1)\dd u_1\dd u_2  \bigg\}.
    \end{align*}
    
    The third term consists of taking the stochastic integral on the sub-domain $\{(u_1,u_2,u_3,u_4) : u_1=u_3, u_2\neq u_4\}$. The corresponding expectation is given by
    \begin{align*}
        K_3 & = \Eo\bigg\{ \int_{\Delta_3}  w(u_1,u_2+\xi_{Z^1_1(u_2-)+1})w(u_1,u_3+\xi_{Z^1_1(u_3-)+1})\dd Z^1_1(u_1)\dd Z^1_1(u_2)\dd Z^1_1(u_3) \bigg\}
        \\
        & = (\eta^1)^3\Eo\bigg\{ \int_{\R^3}  w(u_1,u_2+\xi_2)w(u_1,u_3+\xi_3)\dd u_1\dd u_2\dd u_3 \bigg\}
        \\
        & = (\eta^1)^3J^T_{\delta}.
    \end{align*}
    
    The fourth term consists of taking the stochastic integral on the sub-domain $\{(u_1,u_2,u_3,u_4) : u_1=u_4, u_2\neq u_3\}$. The corresponding expectation is given by
    \begin{align*}
        K_4 & = \Eo\bigg\{ \int_{\Delta_3}  w(u_1,u_2+\xi_{Z^1_1(u_2-)+1})w(u_3,u_1+\xi_{Z^1_1(u_1-)+1})\dd Z^1_1(u_1)\dd Z^1_1(u_2)\dd Z^1_1(u_3) \bigg\}
        \\
        & = (\eta^1)^3\Eo\bigg\{ \int_{\R^3}  w(u_1,u_2+\xi_2)w(u_3,u_1+\xi_1)\dd u_1\dd u_2\dd u_3 \bigg\}
        \\
        & = (\eta^1)^3\E\int_{\R}\tw(u)\tw(u+\xi)\dd u.
    \end{align*}
    
    The fifth term consists of taking the stochastic integral on the sub-domain $\{(u_1,u_2,u_3,u_4) : u_1\neq u_3, u_2=u_4\}$. The corresponding expectation $K_5$ is calculated in the same way the third term and it equals $ K_5 = K_3 =(\eta^1)^3J^T_{\delta}.  $ 
    
    The sixth term consists of taking the stochastic integral on the sub-domain $\{(u_1,u_2,u_3,u_4) : u_1\neq u_4, u_2=u_3\}$ and is treated in the same way as the fourth term, yielding the same expectation $ K_6 = K_4.$
      
    The seventh term consists of taking the stochastic integral on the sub-domain $\Delta_4$. The corresponding expectation is given by
    \begin{multline*}
        K_7  =\\
         \Eo\bigg\{ \int_{\Delta_4}  w(u_1,u_2+\xi_{Z^1_1(u_2-)+1})w(u_3,u_4+\xi_{Z^1_1(u_4-)+1})\dd Z^1_1(u_1)\dd Z^1_1(u_2)\dd Z^1_1(u_3)\dd Z^1_1(u_4) \bigg\}
        \\
        = (\eta^1)^4\Eo\bigg\{ \int_{\R^4}  w(u_1,u_2+\xi_2)w(u_3,u_4+\xi_4)\dd u_1\dd u_2\dd u_3\dd u_4 \bigg\}
        \\
         = (\eta^1)^4(I^T_{\delta})^4.
    \end{multline*}
    
    By adding $K_1,\cdots,K_7$, we obtain
    \begin{multline}\label{eq:dec_term3}
        \Eo\bigg\{\bigg| \int_{\Delta_2}w(u_1,u_2+\xi_{Z^1_1(u_2-)+1})\dd Z^1_1(u_1)\dd Z^1_1(u_2) \bigg|^2\bigg\} 
        \\
        = (\eta^1)^2I^T_{\delta} + 2(\eta^1)^3J^T_{\delta} + (\eta^1)^4(I^T_{\delta})^2 + 2(\eta^1)^3\E\int_{\R}\tw(u)\tw(u+\xi)\dd u
        \\
        +(\eta^1)^2\E\int_{\R^2}w(u_1+\xi_1,u_2)w(u_1,u_2+\xi_2)\dd u ,
    \end{multline}
    and the sum $\eqref{eq:cond_centered} +  \eqref{eq:dec_term1} + 2\times\eqref{eq:dec_term2} + \eqref{eq:dec_term3}$  gives
   \begin{multline}\label{eq:Dq2}
        \Eo\left[ | \Dq |^2 \right] \\
        = 2(\eta^1)^2I^T_{\delta} + 2(\eta^1)^3J^T_{\delta} + \eta^1\E\big[ (T-|\xi|)\ind_{|\xi|\leq\delta} \big] + (\eta^1)^2\left(\E\big[ (T-|\xi|)\ind_{|\xi|\leq\delta} \big]\right)^2
        \\
        + 2(\eta^1)^2\E\int_{\R}\tw(u)w(u,u-\xi)\dd u + (\eta^1)^2\E\int_{\R^2}w(u_1+\xi_1,u_2)w(u_1,u_2+\xi_2)\dd u_1\dd u_2.
    \end{multline}

    Then using \eqref{eq:DdDq} and \eqref{eq:Dq2}, we deduce from the decomposition \eqref{eq:dec_mom2_int} that
    \begin{align*}
        & \Vo\big[ \fp(X_1,X_2) \big] \\
        & \; = \E\big[ |\fp(X_1,X_2)|^2 \big] - (\eta^1)^2\left( \E\big[(T-|\xi|)\ind_{|\xi|\leq\delta}\big]\right)^2
        \\
        & \; = 2\big[ (\eta^1+\lambda^1)(\eta^1+\lambda^2) - (\eta^1)^2 \big]I^T_{\delta} + 2\big[ (\eta^1+\lambda^1)^2(\eta^1+\lambda^2) - (\eta^1)^2(\eta^1+2\lambda^1) \big]J^T_{\delta}
        \\
        &~\;  + 2(\eta^1)^2I^T_{\delta} + 2(\eta^1)^3J^T_{\delta} + \eta^1\E\big[ (T-|\xi|)\ind_{|\xi|\leq\delta} \big] 
        \\
        &~ \; + 2(\eta^1)^2\E\int_{\R}\tw(u)w(u,u-\xi)\dd u + (\eta^1)^2\E\int_{\R^2}w(u_1+\xi_1,u_2)w(u_1,u_2+\xi_2)\dd u_1\dd u_2
        \\
        &~ \; + 4\lambda^1(\eta^1)^2J^T_{\delta} + 2\lambda^1\eta^1\E\bigg\{ \int_{\R}\tw(u)w(u,u-\xi)\dd u \bigg\}.
    \end{align*}
    It follows that
    \begin{align*}
        \Vo\big[ \fp(X_1,X_2) \big] & = 2(\eta^1+\lambda^1)(\eta^1+\lambda^2)\left[ I^T_{\delta} + (\eta^1+\lambda^1)J^T_{\delta} \right]
        \\
        &~ + \eta^1\E\big[ (T-|\xi|)\ind_{|\xi|\leq\delta} \big] + 2\eta^1(\eta^1+\lambda^1)\E\bigg\{ \int_{\R}\tw(u)w(u,u-\xi)\dd u \bigg\}
        \\
        &~ + (\eta^1)^2\E\bigg\{\int_{\R^2}w(u_1 ,u_2)w(u_1 - \xi_1  ,u_2+\xi_2)\dd u_1\dd u_2\bigg\},
    \end{align*}
    and then using the bounds $\eta^1\leq\lamb=\max\{\lambda^1,\lambda^2\}$, $I^T_{\delta}\leq 2\delta T$, $J^T_{\delta}\leq 4\delta^2T, $ $w(u,u-\xi)\leq\ind_{|\xi|\leq\delta}$ and $w(u_1 - \xi_1  ,u_2+\xi_2) \le 1, $ we obtain, for a constant $C$ not depending on the model parameters, 
    \begin{align*}
        \Vo\big[ \fp(X_1,X_2) \big] & \leq 2(\eta^1+\lambda^1)(\eta^1+\lambda^2)\left[ I^T_{\delta} + (\eta^1+\lambda^1)J^T_{\delta} \right]
        \\
        &~ + \eta^1 T\eP(|\xi|\leq\delta) + 2\eta^1\lamb I^T_{\delta}\eP(|\xi|\leq\delta) + 3 (\eta^1)^2I^T_{\delta}
        \\
        & \leq C (1+\lamb\delta)\big( (\lamb)^2\delta T + \eta^1 T\eP(|\xi|\leq\delta) \big),
    \end{align*}
    which implies \eqref{eq:var_jit}. 
    
    \subsection{Proof of the lower bounds \texorpdfstring{\eqref{eq:crit_jit}}{} and \texorpdfstring{\eqref{eq:min_size_jit}}{}}\label{app:lower_bound_jit}
    At this stage, we have proved, combining all the inequalities on the various variances and Lemma~\ref{lem:critgen}, that there exists a constant $c>0$ such that if 
    \[
\Delta_\varphi \geq \frac{c}{\sqrt{n\alpha\beta}}\sqrt{(1+\lamb\delta)((\lamb)^2\delta T + \eta^1 T\eP(|\xi|\leq\delta))},
    \]
    then the Type II error is controlled.
    But since we can upper bound 
    \[
    \eta^1 T\eP(|\xi|\leq\delta)\leq 2 \Delta_\varphi
    \]
    (recall \eqref{eq:deltaphilb}),
    it is sufficient to have
    \[
    \Delta_\varphi \geq \frac{c}{\sqrt{n\alpha\beta}}\sqrt{(1+\lamb\delta)((\lamb)^2\delta T + 2 \Delta_\varphi)}.
    \]
    It remains to solve this in $\Delta_\varphi$. With an additional upper bound, one can show that there exists another constant $C>0$ such that the Type II error is controlled as soon as
    \[
    \Delta_\varphi \geq C\left[\sqrt{\frac{(1+\lamb\delta)(\lamb)^2\delta T}{n\alpha\beta}} + \frac{1+\lamb\delta}{n\alpha\beta}\right],
    \]
    which concludes the proof of \eqref{eq:crit_jit}.
    
    It remains to solve the second order polynomial in $x=1/\sqrt{n}$ of \eqref{eq:crit_jit} and to lower bound $\Delta_\varphi$ by $\eta^1(T/2) \eP (|\xi|\leq\delta)$,
    to obtain that \eqref{eq:crit_jit} is implied by
    \[
    n\geq C \frac{1+\lamb \delta}{\alpha\beta \eta^1 T \eP(|\xi|\leq\delta)} \left[1+ \frac{(\lamb)^2\delta}{\eta^1\eP(|\xi|\leq\delta)}\right].
    \]
     Since all the reasoning is valid only if $n\geq 3/\sqrt{\alpha\beta}$, we obtain \eqref{eq:min_size_jit}.

\section{The Hawkes model}
    \subsection{Proof of Lemma~\ref{lem:cumulant}}
        The first step of the proof consists of computing the rate $\Psi(t)$ defined by \eqref{eq:Psi_Hawkes}. For the interaction function $h(t) = ae^{-bt}\ind_{t>0}$, clearly for any $m\geq 1,$ 
        \[ h^{\ast m}(t) = \frac{a^m t^{m-1}}{(m-1)!}e^{-b t}\ind_{t>0} \]
        which implies that
        \begin{equation}\label{eq:Psi_exp}
            \Psi(t) = a e^{-(b-a)t}\ind_{t>0}= a e^{-b(1- \ell)t}\ind_{t>0}.
        \end{equation}
        We state the following intermediate result which will be used in our calculations later.
        
        \begin{lemma}\label{lem:bound_integral}
            Let $\alpha_1<\cdots<\alpha_m$ be a sequence of real numbers with $m\geq 2$, then 
            \begin{equation}\label{eq:bound_integral}
                \int_{\R}\prod_{i=1}^m\Psi(\alpha_i-v)\dd v = \frac{a}{m(b-a)}\prod_{i=2}^m\Psi(\alpha_i-\alpha_1).
            \end{equation}
            In addition, for any $\alpha\in\R$,
            \begin{equation}\label{eq:bound_integral_abs}
                \int_{\R}\Psi(|\alpha-v|)\prod_{i=1}^m\Psi(\alpha_i-v)\dd v \leq \frac{a}{(m-1)(b-a)}\Psi(|\alpha-\alpha_1|)\prod_{i=2}^m\Psi(\alpha_i-\alpha_1).
            \end{equation}
        \end{lemma}
        \begin{proof}
            We have
            \begin{align*}
                \int_{\R}\prod_{i=1}^m\Psi(\alpha_i-v)\dd v & = \int_{-\infty}^{\alpha_1}a^{m-1}\Psi\left( \sum_{i=1}^m\alpha_i - mv \right)\dd v
                \\
                & = \frac{a^{m-1}}{m(b-a)}\Psi\left( \sum_{i=2}^m\alpha_i - (m-1)\alpha_1 \right)
                \\
                & = \frac{a^{m-1}}{m(b-a)}\Psi\left(\sum_{i=2}^m(\alpha_i-\alpha_1)\right)
            \end{align*}
            which is \eqref{eq:bound_integral}. Further, putting $ \alpha_0 := \alpha$ to ease the notation, we have
            \begin{multline*}
                \int_{\R}\Psi(|\alpha-v|)\prod_{i=1}^m\Psi(\alpha_i-v)\dd v 
                \\
                 = \int_{-\infty}^{\alpha\wedge\alpha_1} \prod_{i=0}^m\Psi(\alpha_i-v)\dd v + \ind_{\alpha<\alpha_1} \int_{\alpha}^{\alpha_1}\Psi(v-\alpha)\prod_{i=1}^m\Psi(\alpha_i-v)\dd v
                 \end{multline*}
                 which in turn equals 
                 \begin{multline*}
                                  a^m\int_{-\infty}^{\alpha\wedge\alpha_1}\Psi\bigg( \sum_{i=0}^m\alpha_i - (m+1)v \bigg)\dd v + a^m\ind_{\alpha<\alpha_1}\int_{\alpha}^{\alpha_1}\Psi\bigg( \sum_{i=1}^m\alpha_i - \alpha - (m-1)v \bigg)\dd v
                \\
                 = \frac{a^m}{(m+1)(b-a)}\Psi\bigg( \sum_{i=0}^m\alpha_i - (m+1)(\alpha\wedge\alpha_1) \bigg)
                \\
                 + \frac{a^m\ind_{\alpha<\alpha_1}}{(m-1)(b-a)}\left[ \Psi\bigg( \sum_{i=1}^m\alpha_i - \alpha - (m-1)\alpha_1 \bigg) - \Psi\bigg( \sum_{i=1}^m\alpha_i - m\alpha \bigg) \right].
            \end{multline*}
            By distinguishing the cases $\alpha\geq\alpha_1$ and $\alpha<\alpha_1$ in the first term, we obtain
            \begin{align*}
                &\int_{\R}\Psi(|\alpha-v|)\prod_{i=1}^m\Psi(\alpha_i-v)\dd v 
                \\
                & = \frac{a^m\ind_{\alpha\geq \alpha_1}}{(m+1)(b-a)}\Psi\bigg( \sum_{i=0}^m\alpha_i - (m+1)\alpha_1 \bigg)
                \\
                &~ + \frac{a^m\ind_{\alpha<\alpha_1}}{(m-1)(b-a)}\left[ \Psi\bigg( \sum_{i=1}^m\alpha_i - \alpha - (m-1)\alpha_1 \bigg) - \frac{2}{m+1} \Psi\bigg( \sum_{i=1}^m\alpha_i - m\alpha \bigg) \right]
                \\
                & = \frac{a^m\ind_{\alpha\geq \alpha_1}}{(m+1)(b-a)}\Psi\bigg( \alpha-\alpha_1 + \sum_{i=2}^m(\alpha_i - \alpha_1) \bigg)
                \\
                &~ + \frac{a^m\ind_{\alpha<\alpha_1}}{(m-1)(b-a)}\Psi\bigg( \alpha_1-\alpha + \sum_{i=2}^m(\alpha_i - \alpha_1) \bigg)\left[ 1 - \frac{2}{m+1} e^{-(b-a)(m-1)(\alpha_1-\alpha)} \right]
                \\
                & \leq \frac{a^m}{(m-1)(b-a)}\Psi\bigg( |\alpha-\alpha_1| + \sum_{i=2}^m(\alpha_i-\alpha_1) \bigg),
            \end{align*}
            implying \eqref{eq:bound_integral_abs}. 
        \end{proof}
        
        Let us now come back to the study of the first four cumulant measures.

        {\color{black}
        
        According to \eqref{eq:order-2}, the second order cumulant measure satisfies
        \begin{align*}
        \ind_{t_1\neq t_2}k_2(\dd t_1,\dd t_2) & = \int_{\mathbb{R}^2} \mu \dd x  \left( \delta_x(\dd u) + \Psi(u - x)\dd u\right)\\
        & \quad \times\left(\delta_u(\dd (t_1\wedge t_2))\Psi(t_1\vee t_2 - t_1\wedge t_2) + 
        \Psi(t_1 - u)\Psi(t_2 - u)\right) \dd t_1 \, \dd t_2\\       
            & = \frac{\mu}{1-\frac{a}{b}} \bigg( \Psi(|t_1-t_2|) + \int_{\R}\Psi(t_1-u)\Psi(t_2-u)\dd u \bigg)\dd t_1\dd t_2
            \\
            & = \frac{\mu}{1-\frac{a}{b}} \bigg( 1 + \frac{1}{2}\frac{a}{b-a} \bigg)\Psi(|t_1-t_2|)\dd t_1\dd t_2 ,
        \end{align*}}
        where we have used \eqref{eq:bound_integral} of Lemma~\ref{lem:bound_integral}. 
        
        The third and fourth order cumulant measures are more complicated. We need to compute the contribution corresponding to each equivalence class in 
        Figures~\ref{fig:cumul_order-3} and \ref{fig:cumul_order-4} respectively, and to take into account the corresponding permutations. Following \eqref{eq:order-3}, we set
        {\color{black}
        \[ S^{(1)}_3(\dd t_1,\dd t_2,\dd t_3) = \int_{\R^2} \mu \dd x \Lambda[x](\dd u) \Lambda[u](\dd t_1, \dd t_2, \dd t_3) \]
        and 
        \[ S^{(2,\sigma)}_3(\dd t_1,\dd t_2,\dd t_3) = \int_{\R^3}\mu \dd x \Lambda[x](\dd u) \Lambda[u](\dd v, \dd t_{\sigma(3)}] \Lambda[v](\dd t_{\sigma(1)}, \dd t_{\sigma(2)}) \]
    }
        such that
        \begin{equation}\label{eq:dec_order-3}
            \ind_{\Delta_{3}}k_3 = S^{(1)}_3 + \sum_{\sigma}S^{(2,\sigma)}_3.
        \end{equation}
        We have
        \begin{multline*}
             \ind_{t_1<t_2<t_3}S^{(1)}_3(\dd t_1,\dd t_2,\dd t_3)
            \\
             = \frac{\mu}{1-\frac{a}{b}}\ind_{t_1<t_2<t_3}\bigg( \Psi(t_3 - t_1)\Psi(t_2-t_1) + \int_{\R}\Psi(t_1-u)\Psi(t_2-u)\Psi(t_3-u)\dd u\bigg)\dd t_1\dd t_2\dd t_3
            \\
              = \frac{\mu}{1-\frac{a}{b}}\ind_{t_1<t_2<t_3}\bigg( 1 + \frac{1}{3}\frac{a}{b-a} \bigg)\Psi(t_3-t_1)\Psi(t_2-t_1)\dd t_1\dd t_2\dd t_3 ,
        \end{multline*}
        thanks to \eqref{eq:bound_integral} in Lemma~\ref{lem:bound_integral}, and then by symmetry
        \begin{equation}\label{eq:S1_order-3}
            S^{(1)}_3(\dd t_1,\dd t_2,\dd t_3) = a\frac{\mu}{1-\frac{a}{b}}\left( 1 + \frac{1}{3}\frac{a}{b-a} \right)\Psi\bigg(\sum_{i=1}^3t_i-3\min_it_i \bigg)\dd t_1\dd t_2\dd t_3.
        \end{equation}
        Further, the measure associated to the equivalence class \textbf{(b)} in Figure~\ref{fig:cumul_order-3} satisfies for a fixed permutation $\sigma$ of $\{1,2,3\},$
       {\color{black}
        \begin{align*}
            & S^{(2,\sigma)}_3(\dd t_1,\dd t_2,\dd t_3) 
            \\
            &\quad = \int_{\R^3} \mu \dd x \left(\delta_x(\dd u) + \Psi(u-x)\dd u\right)\Lambda[u](\dd v,\dd t_{\sigma(3)})\Lambda[v](\dd t_{\sigma(1)}, \dd t_{\sigma(2)}) 
            \\
            &\quad = \frac{\mu}{1-\frac{a}{b}}\bigg( \int_{\R}\bigg\{ \underbrace{\int_{\R}\Psi(v-u)\Psi(t_{\sigma(3)}-u)\dd u}_{=\,\frac{1}{2}\frac{a}{b-a}\Psi(|t_{\sigma(3)}-v|)} + \Psi(v-t_{\sigma(3)}) \bigg\}
            \Lambda[v](\dd t_{\sigma(1)},\dd t_{\sigma(2)}) \dd v \bigg)\dd t_{\sigma(3)}
            \\
            &\quad \leq \frac{\mu}{1-\frac{a}{b}} \left( 1 + \frac{1}{2}\frac{a}{b-a} \right)\bigg\{\int_{\R}\Psi(|v-t_{\sigma(3)}|)\Lambda[v](\dd t_{\sigma(1)}, \dd t_{\sigma(2)}) \dd v\bigg\}\dd t_{\sigma(3)}
            \\
            & \quad = \frac{\mu}{1-\frac{a}{b}} \left( 1 + \frac{1}{2}\frac{a}{b-a} \right) \bigg\{ \Psi(|t_{\sigma(3)}-t_{\sigma(1)}\wedge t_{\sigma(2)}|)\Psi(|t_{\sigma(1)} - t_{\sigma(2)}|)
            \\
            &\qquad\quad + \int_{\R}\Psi(|t_{\sigma(3)}-v|)\Psi(t_{\sigma(1)}-v)\Psi(t_{\sigma(2)}-v)\dd v \bigg\}
            \dd t_1 \dd t_2 \dd t_3 .
        \end{align*}}
        This implies thanks to \eqref{eq:bound_integral_abs} in Lemma~\ref{lem:bound_integral} that
        \begin{multline*}
             S^{(2,\sigma)}_3(\dd t_1,\dd t_2,\dd t_3) \leq 
            \\
              \frac{\mu}{1-\frac{a}{b}} \left( 1 + \frac{1}{2}\frac{a}{b-a} \right) \left( 1 + \frac{a}{b-a} \right)\Psi(|t_{\sigma(3)}-t_{\sigma(1)}\wedge t_{\sigma(2)}|)\Psi(|t_{\sigma(1)} - t_{\sigma(2)}|)\dd t_1\dd t_2\dd t_3
        \end{multline*}
        and then
        \begin{equation}\label{eq:S2sig_order-3}
            S^{(2,\sigma)}_3(\dd t_1,\dd t_2,\dd t_3) \leq \frac{a\mu}{(1-\frac{a}{b})^3}\Psi(|t_{\sigma(3)}-t_{\sigma(1)}\wedge t_{\sigma(2)}| + |t_{\sigma(1)} - t_{\sigma(2)}|) \dd t_1\dd t_2\dd t_3.
        \end{equation}
        Using \eqref{eq:S1_order-3} and \eqref{eq:S2sig_order-3}, we deduce from \eqref{eq:dec_order-3} that
        \begin{align*}
            & \ind_{\Delta_{\ell}}(t_1,t_2,t_3)k_3(\dd t_1,\dd t_2,\dd t_3) 
            \\
            &\qquad \leq \frac{a\mu}{1-\frac{a}{b}}\bigg\{ 1 + \frac{1}{3}\frac{a}{b-a} + \frac{3}{(1-\frac{a}{b})^2}  \bigg\}\Psi\left(\max_it_i - \min_it_i\right)\dd t_1\dd t_2\dd t_3
            \\
            &\qquad \leq \frac{4a\mu}{(1-\frac{a}{b})^3}\Psi\left(\max_it_i - \min_it_i\right) \dd t_1\dd t_2\dd t_3,
        \end{align*}
        where the factor $3$ comes from the number of permutations associated to the equivalence class \textbf{(b)} in Figure~\ref{fig:cumul_order-3}. This implies \eqref{eq:cumul_order-1234} with $l=3$.
        \medskip
        
        We are now interested in the fourth order cumulant, and we introduce the measures associated to each equivalence class in Figure~\ref{fig:cumul_order-4} that are defined from \eqref{eq:order-4} by
        {\color{black}
        \begin{multline}\label{eq:S_order-4}
            S^{(1)}_4 =  \int_{\R^2} \mu \dd x \Lambda[x](\dd u)\Lambda[u](\dd t_1, \dd t_2,\dd t_3, \dd t_4),
            \\
            S^{(2,\sigma)}_4 =\int_{\R^3} \mu \dd x \Lambda[x](\dd u) \Lambda [u](\dd v,  \dd t_{\sigma(3)},\dd t_{\sigma(4)})\Lambda[v](\dd t_{\sigma(1)},\dd t_{\sigma(2)}),
            \\
            S^{(3,\sigma)}_4 =\int_{\R^3}\mu\dd x \Lambda[x](\dd u)\Lambda[u][\dd v,\dd t_{\sigma(4)})\Lambda[v](\dd t_{\sigma(1)}, \dd t_{\sigma(2)}, \dd t_{\sigma(3)}) ,
            \\
            S^{(4,\sigma)}_4 = \int_{\R^4}\mu \dd x\Lambda[x](\dd u)\Lambda[u](\dd v, \dd w)\Lambda[v](\dd t_{\sigma(1)},\dd t_{\sigma(2)})\Lambda[w](\dd t_{\sigma(3)},\dd t_{\sigma(4)}),
            \\
            S^{(5,\sigma)}_4 = \int_{\R^4}\mu \dd x \Lambda[x](\dd u)\Lambda[u](\dd v,\dd t_{\sigma(4)})\Lambda[v](\dd w, \dd t_{\sigma(3)})\Lambda[w](\dd t_{\sigma(1)},\dd t_{\sigma(2)}),
        \end{multline}
        where to ease notation we wrote $ S^{(k, \sigma)}_4 $ instead of $S^{(k, \sigma)}_4 (\dd t_1, \dd t_2, \dd t_3, \dd t_4), $ for each $ 1 \le k \le 5.$
        
        Then
        \begin{equation}\label{eq:dec_order-4}
            \ind_{\Delta_4}k_4 = S^{(1)}_4 + \sum_{\sigma}S^{(2,\sigma)}_4 + \sum_{\sigma}S^{(3,\sigma)}_4 + \sum_{\sigma}S^{(4,\sigma)}_4 + \sum_{\sigma}S^{(5,\sigma)}_4.
        \end{equation}
        
        \paragraph*{Bound on $S^{(1)}_4$} 
        It follows from the first equation in \eqref{eq:S_order-4} and from Lemma~\ref{lem:bound_integral} that
        \begin{align*}
            & \ind_{t_1<t_2<t_3<t_4}S^{(1)}_4(\dd t_1,\dd t_2,\dd t_3,\dd t_4) 
            \\
            &\qquad = \frac{\mu}{1-\frac{a}{b}}\ind_{t_1<t_2<t_3<t_4}\bigg( \Psi(t_2-t_1)\Psi(t_3-t_1)\Psi(t_4-t_1)
            \\
            &\qquad~ + \int_{\R}\Psi(t_1-u)\Psi(t_2-u)\Psi(t_3-u)\Psi(t_4-u)\dd u\bigg)\dd t_1\dd t_2\dd t_3\dd t_4
            \\
            &\qquad = a^2\frac{\mu}{1-\frac{a}{b}}\ind_{t_1<t_2<t_3<t_4}\bigg( 1 + \frac{1}{4}\frac{a}{b-a} \bigg)\Psi(t_2+t_3+t_4-3t_1)\dd t_1\dd t_2\dd t_3\dd t_4,
        \end{align*}
        which implies by symmetry that
        \begin{equation}\label{eq:S1_order-4}
            S^{(1)}_4(\dd t_1,\dd t_2,\dd t_3,\dd t_4) = \frac{a^2\mu}{1-\frac{a}{b}}\bigg( 1 + \frac{1}{4}\frac{a}{b-a} \bigg)\Psi\bigg(\sum_{i=1}^4t_i-4\min_it_i\bigg)\dd t_1\dd t_2\dd t_3\dd t_4 .
        \end{equation}
        
        \paragraph*{Bound on $S^{(2,\sigma)}_4$} We deduce from the second equation in \eqref{eq:S_order-4} and from Lemma~\ref{lem:bound_integral} that
        \begin{multline*}
             S^{(2,\sigma)}_4(\dd t_1,\dd t_2,\dd t_3,\dd t_4)
             = \frac{\mu}{1-\frac{a}{b}}\\
             \bigg\{ \int_{\R}\bigg[ \underbrace{\int_{\R}\Psi(v-u)\Psi(t_{\sigma(3)}-u)\Psi(t_{\sigma(4)}-u)\dd u}_{=\, \frac{a^2}{3(b-a)}\Psi(v + t_{\sigma(3)} + t_{\sigma(4)} - 3v\wedge t_{\sigma(3)}\wedge t_{\sigma(4)})} + \Psi(v-t_{\sigma(3)}\wedge t_{\sigma(4)})\Psi(|t_{\sigma(3)}-t_{\sigma(4)}|) \bigg]
            \\
            \qquad\qquad\qquad \Lambda[v](\dd t_{\sigma(1)}, \dd t_{\sigma(2)}) \dd v\bigg\}\dd t_{\sigma(3)}\dd t_{\sigma(4)}
            \end{multline*}
            which is in turn upper bounded by
            \[
             \frac{a\mu}{1-\frac{a}{b}}
             \left( 1 + \frac{1}{3}\frac{a}{b-a} \right)
             \bigg\{ \int_{\R}\Psi(v + t_{\sigma(3)} + t_{\sigma(4)} - 3v\wedge t_{\sigma(3)}\wedge t_{\sigma(4)})\Lambda[v](\dd t_{\sigma(1)},\dd t_{\sigma(2)}) \dd v \bigg\}\dd t_{\sigma(3)}\dd t_{\sigma(4)}
            \]
            which equals
            \begin{multline*}
             \frac{a\mu}{1-\frac{a}{b}}\left( 1 + \frac{1}{3}\frac{a}{b-a} \right) 
            \\
            \bigg\{ \Psi(t_{\sigma(1)}\wedge t_{\sigma(2)} + t_{\sigma(3)} + t_{\sigma(4)} - 3(t_{\sigma(1)}\wedge t_{\sigma(2)})\wedge t_{\sigma(3)}\wedge t_{\sigma(4)})\Psi(|t_{\sigma(1)} - t_{\sigma(2)}|)
            \\
             + \int_{\R}\Psi(v + t_{\sigma(3)} + t_{\sigma(4)} - 3v\wedge t_{\sigma(3)}\wedge t_{\sigma(4)})\Psi(t_{\sigma(1)}-v)\Psi(t_{\sigma(2)}-v)\dd v \bigg\}\dd t_1\dd t_2\dd t_3\dd t_4.
        \end{multline*}
        Now,
        \begin{align*}
            & \int_{\R}\Psi(v + t_{\sigma(3)} + t_{\sigma(4)} - 3v\wedge t_{\sigma(3)}\wedge t_{\sigma(4)})\Psi(t_{\sigma(1)}-v)\Psi(t_{\sigma(2)}-v)\dd v
            \\
            &\quad = \ind_{t_{\sigma(3)}\wedge t_{\sigma(4)} > t_{\sigma(1)}\wedge t_{\sigma(2)}} \int_\R  \Psi(t_{\sigma(3)}+  t_{\sigma(4)}  - 2 v ) \Psi ( t_{\sigma(1)} -v) \Psi ( t_{\sigma(2)}- v) \dd v  
            \\
            &\quad~ + \ind_{t_{\sigma(3)}\wedge t_{\sigma(4)} \leq t_{\sigma(1)}\wedge t_{\sigma(2)}} \int_{t_{\sigma(3)}\wedge t_{\sigma(4)}}^{t_{\sigma(1)}\wedge t_{\sigma(2)}} \Psi ( v + t_{\sigma(3)}\vee t_{\sigma(4)} - 2 t_{\sigma(3)}\wedge t_{\sigma(4)}) \\
            & \quad \quad \quad \quad \quad \quad \quad \quad \quad \quad \quad \quad \quad \quad \quad \quad  \Psi (  t_{\sigma(1)}-v) \Psi (   t_{\sigma(2)} - v ) \dd v =: I_1 + I_2.
            \end{align*}
    Using Lemma~\ref{lem:bound_integral}, clearly 
    \begin{multline*} I_1 =\frac{a^2}{4(b-a)}\ind_{t_{\sigma(3)}\wedge t_{\sigma(4)} > t_{\sigma(1)}\wedge t_{\sigma(2)}}\\
    \Psi(|t_{\sigma(1)} - t_{\sigma(2)}| + |t_{\sigma(3)} - t_{\sigma(4)}| + 2(t_{\sigma(3)}\wedge t_{\sigma(4)} - t_{\sigma(1)}\wedge t_{\sigma(2)}))  . 
    \end{multline*}
    Moreover, on $\{t_{\sigma(3)}\wedge t_{\sigma(4)} \leq t_{\sigma(1)}\wedge t_{\sigma(2)}\} ,$
            \begin{align*}
                  I_2 &=\frac{a^2}{b-a} \\
                  &\Big[ \Psi ( t_{\sigma(1)}\wedge t_{\sigma(2)} + |t_{\sigma(3)}- t_{\sigma(4)} | - t_{\sigma(3)}\wedge t_{\sigma(4)} + |t_{\sigma(1)}- t_{\sigma(2)} |)  \\
                  & \quad \quad \quad \quad \quad \quad -
                  \Psi (|t_{\sigma(3)}- t_{\sigma(4)} | + ( t_{\sigma(1)} - t_{\sigma(3)}\wedge t_{\sigma(4)}) + ( t_{\sigma(2)} - t_{\sigma(3)}\wedge t_{\sigma(4)}) ) \Big] \\
                  &= \frac{a^2}{b-a} \Psi ( |t_{\sigma(1)}- t_{\sigma(2)} | + |t_{\sigma(3)}- t_{\sigma(4)} | + t_{\sigma(1)}\wedge t_{\sigma(2)} - t_{\sigma(3)}\wedge t_{\sigma(4)}) \\
                  & \quad \quad \quad \quad \quad \quad \quad \quad \quad \quad  \Big [ 1 -e^{-(b-a)(t_{\sigma(1)}\wedge t_{\sigma(2)} - t_{\sigma(3)}\wedge t_{\sigma(4)})}  \Big],
                      \end{align*}
                      implying that 
        \[ 
        I_2 \le \frac{a^2}{b-a} \Psi ( |t_{\sigma(1)}- t_{\sigma(2)} | + |t_{\sigma(3)}- t_{\sigma(4)} | + t_{\sigma(1)}\wedge t_{\sigma(2)} - t_{\sigma(3)}\wedge t_{\sigma(4)}).
        \]
        Putting things together, we obtain finally that 
        \[ I_1 + I_2 \le 
            \frac{a^2}{b-a}\Psi(|t_{\sigma(3)} - t_{\sigma(4)}| + \beta(t_{\sigma(3)}\wedge t_{\sigma(4)} - t_{\sigma(1)}\wedge t_{\sigma(2)}) + |t_{\sigma(1)} - t_{\sigma(2)}|), 
        \]
        where we set $\beta(x) = |x| + x_+$ for any $x\in\R$. Thus, writing for short $ \dd t_1^4 := \dd t_1 \dd t_2 \dd t_3 \dd t_4,$ 
        \begin{multline*}
             S^{(2,\sigma)}_4(\dd t_1,\dd t_2,\dd t_3,\dd t_4)\leq \frac{\mu}{1-\frac{a}{b}}\bigg(1 + \frac{a}{b-a}\bigg) 
            \\
             \bigg(1 + \frac{a}{3(b-a)}\bigg) \Psi(|t_{\sigma(3)} - t_{\sigma(4)}|)\Psi\circ \beta(t_{\sigma(3)}\wedge t_{\sigma(4)}-t_{\sigma(1)}\wedge t_{\sigma(2)})\Psi(|t_{\sigma(1)} - t_{\sigma(2)}|)\dd t_1^4, 
        \end{multline*}
        and then
        \begin{multline}\label{eq:S2sig_order-4}
            S^{(2,\sigma)}_4(\dd t_1,\dd t_2,\dd t_3,\dd t_4) \leq \frac{a^2\mu}{(1-\frac{a}{b})^3}
            \\
            \Psi(|t_{\sigma(3)} - t_{\sigma(4)}| + \beta(t_{\sigma(3)}\wedge t_{\sigma(4)}-t_{\sigma(1)}\wedge t_{\sigma(2)}) + |t_{\sigma(1)} - t_{\sigma(2)}|)\dd t_1^4.
        \end{multline}

        \paragraph*{Bound on $S^{(3,\sigma)}_4$} We deduce from the third equation in \eqref{eq:S_order-4} that, 
        \begin{align*}
            & S^{(3,\sigma)}_4(\dd t_1,\dd t_2,\dd t_3,\dd t_4)= \frac{\mu}{1-\frac{a}{b}} 
            \\
            & \bigg\{ \int_{\R}\bigg[ \underbrace{\int_{\R}\Psi(v-u)\Psi(t_{\sigma(4)}-u) \dd u}_{=\, \frac{a}{2(b-a)}\Psi(|t_{\sigma(4)} - v|)} + \Psi(v-t_{\sigma(4)}) \bigg]\Lambda[v](\dd t_{\sigma(1)},\dd t_{\sigma(2)},\dd t_{\sigma(3)}) \dd v\bigg\}\dd t_{\sigma(4)}
            \\
            & \leq \frac{\mu}{1-\frac{a}{b}}\left( 1 + \frac{1}{2}\frac{a}{b-a}\right)\bigg\{ \int_{\R}\Psi(|v-t_{\sigma(4)}|)\Lambda[v](\dd t_{\sigma(1)}, \dd t_{\sigma(2)},\dd t_{\sigma(3)})\dd v\bigg\}\dd t_{\sigma(4)}
            \\
            & = \frac{\mu}{1-\frac{a}{b}}\left( 1 + \frac{1}{2}\frac{a}{b-a}\right)\bigg\{ a\Psi(|t_{\sigma(1)}\wedge t_{\sigma(2)}\wedge t_{\sigma(3)}-t_{\sigma(4)}|)\\
            & \qquad\qquad\qquad\qquad\qquad\qquad\qquad\qquad\qquad\Psi\left(\sum_{i=1}^3t_{\sigma(i)} - 3t_{\sigma(1)}\wedge t_{\sigma(2)}\wedge t_{\sigma(3)}\right)
            \\
            &\qquad + \int_{\R}\Psi(|v-t_{\sigma(4)}|)\Psi(t_{\sigma(1)}-v)\Psi(t_{\sigma(2)}-v)\Psi(t_{\sigma(3)}-v)\dd v \bigg\}\dd t_1^4 .
        \end{align*}
        
        Using \eqref{eq:bound_integral_abs} in Lemma~\ref{lem:bound_integral}, it follows that
        \begin{align*}
            S^{(3,\sigma)}_4(\dd t_1,\dd t_2,\dd t_3,\dd t_4) & \leq \frac{a\mu}{1-\frac{a}{b}}\left( 1 + \frac{1}{2}\frac{a}{b-a} \right)^2 \Psi(|t_{\sigma(1)}\wedge t_{\sigma(2)}\wedge t_{\sigma(3)}-t_{\sigma(4)}|)
            \\
            &\qquad\qquad \Psi\left(\sum_{i=1}^3t_{\sigma(i)} - 3t_{\sigma(1)}\wedge t_{\sigma(2)}\wedge t_{\sigma(3)}\right)\dd t_1^4 ,
        \end{align*}
        and then
        \begin{multline}\label{eq:S3sig_order-4}
            S^{(3,\sigma)}_4(\dd t_1,\dd t_2,\dd t_3,\dd t_4) 
            \\
            \leq \frac{a^2\mu}{(1-\frac{a}{b})^3}\Psi\bigg(|t_{\sigma(4)}-t_{\sigma(1)}\wedge t_{\sigma(2)}\wedge t_{\sigma(3)}| + \sum_{i=1}^3t_{\sigma(i)} - 3t_{\sigma(1)}\wedge t_{\sigma(2)}\wedge t_{\sigma(3)}\bigg)\dd t_1^4.
        \end{multline}

        \paragraph*{Bound on $S^{(4,\sigma)}_4$} We deduce from the fourth equation in \eqref{eq:S_order-4} that
        \begin{align*}
            & S^{(4,\sigma)}_4(\dd t_1,\dd t_2,\dd t_3,\dd t_4)
            \\
            & = \frac{\mu}{1-\frac{a}{b}}\bigg\{ \int_{\R^2}\left( \int_{\R}\Psi(v-u)\Psi(w-u)\dd u \right)\Lambda[v](\dd t_{\sigma(1)},\dd t_{\sigma(2)})
            \Lambda[w](\dd t_{\sigma(3)}, \dd t_{\sigma(4)})  \dd v \dd w\bigg\}
            \\
            & = \frac{a}{2b}\frac{\mu}{(1-\frac{a}{b})^2}\bigg\{ \int_{\R^2}\Psi(|v-w|)\Lambda[v](\dd t_{\sigma(1)},\dd t_{\sigma(2)})\Lambda[w](\dd t_{\sigma(3)},\dd t_{\sigma(4)}) \dd v \dd w\bigg\}
            \\
            & = \frac{a}{2b}\frac{\mu}{(1-\frac{a}{b})^2}\bigg\{ \int_{\R} \bigg[ \Psi(|t_{\sigma(1)}\wedge t_{\sigma(2)}-w|)\Psi(|t_{\sigma(1)}-t_{\sigma(2)}|) 
            \\
            &\qquad\qquad\qquad + \int_{\R}\Psi(|v-w|)\Psi(t_{\sigma(1)}-v)\Psi(t_{\sigma(2)}-v)\dd v \bigg]\Lambda[w](\dd t_{\sigma(3)},\dd t_{\sigma(4)}) \dd w\bigg\}\dd t_{\sigma(2)} \dd t_{\sigma(1)} .
        \end{align*}
        Using \eqref{eq:bound_integral_abs} in Lemma~\ref{lem:bound_integral}, it follows that
        \begin{align*}
            & S^{(4,\sigma)}_4(\dd t_1,\dd t_2,\dd t_3,\dd t_4)
            \\
            & \leq \frac{a}{2b}\frac{\mu}{(1-\frac{a}{b})^3}\Psi(|t_{\sigma(1)}-t_{\sigma(2)}|)\bigg\{ \int_{\R} \Psi(|t_{\sigma(1)}\wedge t_{\sigma(2)}-w|)\Lambda[w](\dd t_{\sigma(3)}, \dd t_{\sigma(4)}) \dd w\bigg\}
            \dd t_{\sigma(1)} \dd t_{\sigma(2)}
            \\
            & \leq \frac{a}{2b}\frac{\mu}{(1-\frac{a}{b})^3}\Psi(|t_{\sigma(1)}-t_{\sigma(2)}|)\bigg\{ \Psi(|t_{\sigma(1)}\wedge t_{\sigma(2)}-t_{\sigma(3)}\wedge t_{\sigma(4)}|)\Psi(|t_{\sigma(3)}-t_{\sigma(4)}|)
            \\
            &\quad + \int_{\R} \Psi(|t_{\sigma(1)}\wedge t_{\sigma(2)}-w|)\Psi(t_{\sigma(3)}-w)\Psi(t_{\sigma(4)}-w)\dd w \bigg\}\dd t_1^4
        \end{align*}
        and then
        \begin{multline}\label{eq:S4sig_order-4}
            S^{(4,\sigma)}_4(\dd t_1,\dd t_2,\dd t_3,\dd t_4)
            \\
            \leq \frac{a^3}{2b}\frac{\mu}{(1-\frac{a}{b})^4}\Psi(|t_{\sigma(3)} - t_{\sigma(4)}| + |t_{\sigma(1)}\wedge t_{\sigma(2)}-t_{\sigma(3)}\wedge t_{\sigma(4)}| + |t_{\sigma(1)} - t_{\sigma(2)}|)\dd t_1^4.
        \end{multline}

        \paragraph*{Bound on $S^{(5,\sigma)}_4$} We deduce from the fifth equation in \eqref{eq:S_order-4} that
        \begin{align*}
            & S^{(5,\sigma)}_4(\dd t_1,\dd t_2,\dd t_3,\dd t_4)
            \\
            & = \frac{\mu}{1-\frac{a}{b}}\bigg\{ \int_{\R^2}\bigg( \underbrace{\int_{\R}\Psi(v-u)\Psi(t_{\sigma(4)}-u)\dd u}_{=\,\frac{a}{2(b-a)}\Psi(|v-t_{\sigma(4)}|)} + \Psi(v-t_{\sigma(4)}) \bigg)
            \\
            &\qquad\qquad\qquad \Lambda[v](\dd w,\dd t_{\sigma(3)})\Lambda[w](\dd t_{\sigma(1)},\dd t_{\sigma(2)}) \dd v \bigg\}\dd t_{\sigma(4)}
            \\
            & \leq \frac{\mu}{1-\frac{a}{b}}\bigg( 1 + \frac{1}{2}\frac{a}{b-a} \bigg)\bigg\{ \int_{\R^2}\Psi(|v-t_{\sigma(4)}|)\Lambda[v](\dd w, \dd t_{\sigma(3)})\Lambda[w](\dd t_{\sigma(1)},\dd t_{\sigma(2)}) \dd v \bigg\}\dd t_{\sigma(4)}
            \\
            & = \frac{\mu}{1-\frac{a}{b}}\bigg( 1 + \frac{1}{2}\frac{a}{b-a} \bigg)\bigg\{ \int_{\R}\bigg[ \int_{\R}\Psi(|v-t_{\sigma(4)}|)\Psi(w-v)\Psi(t_{\sigma(3)}-v)\dd v 
            \\
            &\qquad\qquad\qquad\qquad + \Psi(|t_{\sigma(3)}-t_{\sigma(4)}|)\Psi(w-t_{\sigma(3)}) \bigg]\Lambda[w](\dd t_{\sigma(1)}, \dd t_{\sigma(2)}) \dd w\bigg\}\dd t_{\sigma(3)}\dd t_{\sigma(4)}.
        \end{align*}
        Using \eqref{eq:bound_integral_abs} in Lemma~\ref{lem:bound_integral}, we deduce that
        \begin{align*}
            & S^{(5,\sigma)}_4(\dd t_1,\dd t_2,\dd t_3,\dd t_4)
            \\
            & \leq \frac{\mu}{1-\frac{a}{b}}\bigg( 1 + \frac{1}{2}\frac{a}{b-a} \bigg)\bigg\{ \int_{\R}\bigg[ \frac{a}{b-a}\Psi(|t_{\sigma(4)} - w\wedge t_{\sigma(3)}|)\Psi(|w-t_{\sigma(3)}|)
            \\
            &\qquad\qquad\qquad\qquad + \Psi(|t_{\sigma(3)}-t_{\sigma(4)}|)\Psi(w-t_{\sigma(3)}) \bigg]\Lambda[w](\dd t_{\sigma(1)},\dd t_{\sigma(2)}) \dd w \bigg\}\dd t_{\sigma(3)} \dd t_{\sigma(4)}
            \\
            & \leq \frac{\mu}{1-\frac{a}{b}}\bigg( 1 + \frac{1}{2}\frac{a}{b-a} \bigg)\bigg( 1 + \frac{a}{b-a} \bigg)
            \\
            &\qquad \bigg\{ \int_{\R}\Psi(|t_{\sigma(4)} - w\wedge t_{\sigma(3)}|)\Psi(|w-t_{\sigma(3)}|) \Lambda[w](\dd t_{\sigma(1)}, \dd t_{\sigma(2)}) \dd w\bigg\}\dd t_{\sigma(3)} \dd t_{\sigma(4)}
            \\
            & \leq \frac{\mu}{(1-\frac{a}{b})^3}\bigg\{ \Psi(|t_{\sigma(4)} - t_{\sigma(1)}\wedge t_{\sigma(2)}\wedge t_{\sigma(3)}|)\Psi(|t_{\sigma(1)}\wedge t_{\sigma(2)}-t_{\sigma(3)}|) \Psi(|t_{\sigma(1)}-t_{\sigma(2)}|) 
            \\
            &\qquad + \int_{\R}\Psi(|t_{\sigma(4)} - w\wedge t_{\sigma(3)}|)\Psi(|w-t_{\sigma(3)}|) \Psi(t_{\sigma(1)}-w)\Psi(t_{\sigma(2)}-w)\dd w \bigg\}\dd t_1^4.
        \end{align*}
        }
        Moreover, 
        \begin{align*}
            & \int_{\R}\Psi(|t_{\sigma(4)} - w\wedge t_{\sigma(3)}|)\Psi(|w-t_{\sigma(3)}|) \Psi(t_{\sigma(1)}-w)\Psi(t_{\sigma(2)}-w)\dd w
            \\
            &\quad = \int_{\R}\Psi(|t_{\sigma(4)} - w|)\Psi(t_{\sigma(3)}-w) \Psi(t_{\sigma(1)}-w)\Psi(t_{\sigma(2)}-w)\dd w
            \\
            &\quad~ + \ind_{t_{\sigma(3)} < t_{\sigma(1)}\wedge t_{\sigma(2)}}\Psi(|t_{\sigma(3)}-t_{\sigma(4)}|)\! \! \! \!\!\!\! \int\limits_{t_{\sigma(3)}}^{t_{\sigma(1)}\wedge t_{\sigma(2)}}\! \! \! \Psi(w-\! t_{\sigma(3)}) \Psi(t_{\sigma(1)}-\! w)\Psi(t_{\sigma(2)}-\! w)\dd w
            \\
            &\quad \leq \frac{a^2}{2(b-a)}\Psi(|t_{\sigma(4)} - t_{\sigma(1)}\wedge t_{\sigma(2)}\wedge t_{\sigma(3)}|)\Psi\bigg( \sum_{i=1}^3t_{\sigma(i)} - 3t_{\sigma(1)}\wedge t_{\sigma(2)}\wedge t_{\sigma(3)} \bigg)
            \\
            &\quad~ + \frac{a^2}{b-a}\ind_{t_{\sigma(3)} < t_{\sigma(1)}\wedge t_{\sigma(2)}}\Psi(|t_{\sigma(3)}-t_{\sigma(4)}|) \bigg[ \Psi(t_{\sigma(1)} + t_{\sigma(2)} - t_{\sigma(3)} - t_{\sigma(1)}\wedge t_{\sigma(2)}) 
            \\
            &\qquad\qquad\qquad - \Psi(t_{\sigma(1)} + t_{\sigma(2)} - 2t_{\sigma(3)}) \bigg]
        \end{align*}
        thanks to \eqref{eq:bound_integral_abs} in Lemma~\ref{lem:bound_integral}. By distinguishing the cases $t_{\sigma(3)} \geq t_{\sigma(1)}\wedge t_{\sigma(2)}$ and the case $t_{\sigma(3)} < t_{\sigma(1)}\wedge t_{\sigma(2)}$ in the first term, we obtain
        \begin{align*}
            & \int_{\R}\Psi(|t_{\sigma(4)} - w\wedge t_{\sigma(3)}|)\Psi(|w-t_{\sigma(3)}|) \Psi(t_{\sigma(1)}-w)\Psi(t_{\sigma(2)}-w)\dd w
            \\
            &\quad \leq \frac{a^2}{2(b-a)}\ind_{t_{\sigma(3)} \geq t_{\sigma(1)}\wedge t_{\sigma(2)}}\Psi(|t_{\sigma(4)} - t_{\sigma(1)}\wedge t_{\sigma(2)}|)\Psi\bigg(\underbrace{\sum_{i=1}^3t_{\sigma(i)} - 3t_{\sigma(1)}\wedge t_{\sigma(2)}}_{=\, |t_{\sigma(1)}-t_{\sigma(2)}| + t_{\sigma(3)} - t_{\sigma(1)}\wedge t_{\sigma(2)}}\bigg)
            \\
            &\quad~ + \frac{a^2}{b-a}\ind_{t_{\sigma(3)} < t_{\sigma(1)}\wedge t_{\sigma(2)}}\Psi(|t_{\sigma(4)}-t_{\sigma(3)}|) \Psi(\underbrace{t_{\sigma(1)} + t_{\sigma(2)} - t_{\sigma(3)} - t_{\sigma(1)}\wedge t_{\sigma(2)}}_{=\, |t_{\sigma(1)}-t_{\sigma(2)}|+t_{\sigma(1)}\wedge t_{\sigma(2)}-t_{\sigma(3)}})
            \\
            &\quad \leq \frac{a}{b-a}\Psi(|t_{\sigma(4)}-t_{\sigma(1)}\wedge t_{\sigma(2)}\wedge t_{\sigma(3)}|)\Psi(|t_{\sigma(1)}-t_{\sigma(2)}|)\Psi(|t_{\sigma(3)}-t_{\sigma(1)}\wedge t_{\sigma(2)}|).
        \end{align*}
        Then coming back to the bound on $S^{5,\sigma}_4$, we deduce that
        \begin{multline}\label{eq:S5sig_order-4}
            S^{(5,\sigma)}_4(\dd t_1,\dd t_2,\dd t_3,\dd t_4)
            \\
            \leq \frac{a^2\mu}{(1-\frac{a}{b})^4}\Psi(|t_{\sigma(4)}-t_{\sigma(1)}\wedge t_{\sigma(2)}\wedge t_{\sigma(3)}| + |t_{\sigma(1)}\wedge t_{\sigma(2)}-t_{\sigma(3)}| + |t_{\sigma(1)} - t_{\sigma(2)}|)\dd t_1^4.
        \end{multline}

        Thanks to the inequalities \eqref{eq:S1_order-4}, \eqref{eq:S2sig_order-4}, \eqref{eq:S3sig_order-4}, \eqref{eq:S4sig_order-4} and \eqref{eq:S5sig_order-4}, we deduce {\color{black}from} \eqref{eq:dec_order-4} that the fourth order cumulant measure satisfies the bound
        \begin{align*}
            & \ind_{\Delta_4}(t_1,t_2,t_3,t_4)k_4(\dd t_1,\dd t_2, \dd t_3,\dd t_4)
            \\
            & \leq \frac{a^2\mu}{1-\frac{a}{b}} \bigg[ 1 + \frac{1}{4}\frac{a}{b-a} + 6\times \frac{1}{(1 - \frac{a}{b})^2} + 4\times \frac{1}{(1 - \frac{a}{b})^2} + 3\times \frac{a}{2b}\frac{1}{(1 - \frac{a}{b})^3} + 12\times \frac{1}{(1 - \frac{a}{b})^3} \bigg] 
            \\
            &\quad \times\Psi\left(\max_it_i - \min_it_i\right)\dd t_1\dd t_2\dd t_3\dd t_4
            \\
            & \leq \frac{C a^2\mu}{(1-\frac{a}{b})^4}\Psi\left(\max_it_i - \min_it_i\right)\dd t_1\dd t_2\dd t_3\dd t_4,
        \end{align*}
        which implies \eqref{eq:cumul_order-1234} with {\color{black}$m=4$}.

    \subsection{Proof of Eq. \texorpdfstring{\eqref{eq:mean_Haw}}{}}\label{appendixC2}
        Let us recall that the entire network can be constructed using the univariate Hawkes process $N$ with intensity
        \[ \lambda(t) = \nu M + a\int_{(-\infty,t)}e^{-b(t-s)}\dd N_s,  \]
        for all $t.$
        
        We introduce an \textrm{i.i.d} family $\{J_i,i\in\mathbb{Z}\}$ of uniform random variables on $\llbracket 1,M\rrbracket$, independent of anything else. Each jump time $u$ of the Hawkes process $N$ corresponds to a jump time of the neuron with label $J_{N_{u-}+1}$.
       
        It follows that under the observed model, we have
        \begin{align*}
            \varphi(X^1_1,X^2_1) & = \int_0^T\int_0^T\ind_{|u-v|\leq\delta}\dd X^1_1(u)\dd X^2_1(v) 
            \\
            & = \int_0^T\int_0^T\ind_{|u-v|\leq\delta}\ind_{J_{N_{u-}+1}=1}\ind_{J_{N_{v-}+1}=2}\dd N_u\dd N_v ,
        \end{align*}
        and then, conditioning by $N$, we deduce that
        \begin{align*}
            \Eo\big[\varphi(X^1_1,X^2_1)\big] & = \E\left[ \int_0^T\int_0^T\ind_{|u-v|\leq\delta}\eP(J_1=1,J_2=2)\ind_{u\neq v}\dd N_u\dd N_v \right] 
            \\
            & =\frac{1}{M^2} \E\left[\int_0^T\int_0^T\ind_{|u-v|\leq\delta}\ind_{u\neq v}\dd N_u\dd N_v \right]
            \\
            & =\frac{1}{M^2}\int_0^T\int_0^T\ind_{|u-v|\leq\delta}\ind_{u\neq v}\E[\dd N_u\dd N_v].
        \end{align*}
        Using the same idea under the independent model, we have
        \begin{align*}
            \varphi(X^1_1,X^2_1) & = \int_0^T\int_0^T\ind_{|u-v|\leq\delta}\dd X^1_1(u)\dd X^2_1(v) 
            \\
            & = \int_0^T\int_0^T\ind_{|u-v|\leq\delta}\ind_{J_{N_{u-}+1}=1}\ind_{J_{\tilde{N}_{v-}+1}=2}\dd N_u\dd\widetilde{N}_v ,
        \end{align*}
        where $\widetilde{N}$ is an independent copy of $N$, independent of $\{J_i,i\in\mathbb{Z}\}$. Conditioning on $(N,\widetilde{N})$, it follows that
        \begin{align*}
            \Ez\big[\varphi(X^1_1,X^2_1)\big] & = \E\left[ \int_0^T\int_0^T\ind_{|u-v|\leq\delta}\eP(J_1=1)\eP(J_1=2)\dd N_u\dd\widetilde{N}_v \right] 
            \\
            & =\frac{1}{M^2} \E\left[\int_0^T\int_0^T\ind_{|u-v|\leq\delta}\dd N_u\dd\widetilde{N}_v \right]
            \\
            & =\frac{1}{M^2}\int_0^T\int_0^T\ind_{|u-v|\leq\delta}\E[\dd N_u]\E[\dd N_v].
        \end{align*}
        We deduce that
        \begin{align*}
            & \E\big[\varphi(X^1_1,X^2_1)\big] - \Ez\big[\varphi(X^1_1,X^2_1)\big] \\
            & = \frac{1}{M^2}\int_0^T\int_0^T\ind_{|u-v|\leq\delta}\ind_{u\neq v}\big(\E[\dd N_u\dd N_v] - \E[\dd N_u]\E[\dd N_v]\big)
            \\
            & = \frac{1}{M^2}\int_0^T\int_0^T\ind_{|u-v|\leq\delta}\ind_{u\neq v}k_2(\dd u,\dd v),
        \end{align*}
        implying \eqref{eq:mean_cumul_hawkes}. 
        
        Using \eqref{eq:cumul_order-2} in Lemma~\ref{lem:cumulant}, we obtain moreover 
        \begin{align*}
            & \Eo\big[ \varphi(X^1_1,X^2_1) \big] - \Ez\big[ \varphi(X^1_1,X^2_1) \big] 
            \\
            &\qquad = \frac{\nu}{M}\frac{a}{1 - \frac{a}{b}} \bigg( 1 + \frac{1}{2}\frac{a}{b-a} \bigg)\int_0^T\int_0^Te^{-(b-a)|u-v|}\ind_{|u-v|\leq\delta} \dd u \dd v 
            \\
            &\qquad = \frac{\nu}{M}\frac{2a}{1 - \frac{a}{b}} \bigg( 1 + \frac{1}{2}\frac{a}{b-a} \bigg)\int_0^T\int_0^ue^{-(b-a)(u-v)}\ind_{u-v\leq\delta} \dd v \dd u 
            \\
            &\qquad = \frac{2\nu}{M}\frac{a}{1 - \frac{a}{b}} \bigg( 1 + \frac{1}{2}\frac{a}{b-a} \bigg)\bigg\{ \int_0^\delta\int_0^ue^{-(b-a)(u-v)} \dd v \dd u + \int_\delta^T\int_{u-\delta}^ue^{-(b-a)(u-v)} \dd v \dd u \bigg\} 
            \\
            &\qquad = \frac{2\nu}{M}\frac{a}{1 - \frac{a}{b}} \bigg( 1 + \frac{1}{2}\frac{a}{b-a} \bigg)\frac{1}{b-a}\bigg\{ \int_0^\delta\bigg[1 - e^{-(b-a)u}\bigg] \dd u + (T-\delta)\left[ 1 - e^{-(b-a)\delta}\right] \bigg\} 
            \\
            & \qquad = \frac{\nu}{M}\frac{2-\frac{a}{b}}{(1 - \frac{a}{b})^3} \frac{a}{b}\bigg\{ \delta - \frac{1 - e^{-(b-a)\delta}}{b-a} + (T-\delta)\left[ 1 - e^{-(b-a)\delta}\right] \bigg\}.
        \end{align*}
        Factorizing $\delta$ and recalling that $ \ell = a/b, $ we can rewrite the above expression as 
        \begin{align*}\Delta_\varphi &=  \frac{\nu \delta}{M }\frac{(2 - \ell) \ell  }{(1- \ell)^3} \bigg \{  1 -\frac{1 - e^{-(b-a)\delta}}{(b-a) \delta } + ( T - \delta) \frac{1 - e^{-(b-a)\delta} }{ \delta }   \bigg\} \\
        &=\frac{\nu \delta}{M }\frac{(2 - \ell) \ell  }{(1- \ell)^3} \bigg \{ 1+ [ (T- \delta) ( b-a) - 1 ]\frac{1 - e^{-(b-a)\delta}}{(b-a) \delta } \bigg\} \\
        &= \frac{\nu \delta}{M }\frac{(2 - \ell) \ell  }{(1- \ell)^3} \bigg \{ 1+ [ (T- \delta) b( 1-\ell) - 1 ]\frac{1 - e^{-b(1-\ell)\delta}}{b(1-\ell) \delta } \bigg\} ,
        \end{align*} 
        which gives \eqref{eq:mean_Haw}. 
        
        Finally, under the additional condition  $b(1-\ell)T>4,$ and using moreover that   $ \delta \le T/2 ,$ we have $ T- \delta \geq T/2 $ and that $ - 1 >- b(1-\ell)T/4 ,$ such that 
        \[ b(1- \ell ) (T- \delta ) - 1 \geq \frac14 b(1- \ell ) T .\] 
        Since $2 - \ell \geq 1,$ this gives the lower bound  
        \[
        \Delta_\varphi\geq   \frac{\nu\delta\ell}{M(1-\ell)^3} \frac{b(1-\ell)T}{4}  \frac{1-e^{-b(1-\ell)\delta}}{b(1-\ell)\delta}= \frac14  \frac{\nu \ell T}{M(1- \ell)^3}(1-e^{-b(1-\ell)\delta}) ,
        \]
        which is \eqref{eq:bound_Hawkes}.

    \subsection{Proof of Eq. \texorpdfstring{\eqref{eq:var_Hawkes_ind}}{}}
        Let us now focus on the variance of $\fp(X_1,X_2)$ for the model of independent neurons. The proof of the inequality \eqref{eq:var_Hawkes_ind} follows from the following intermediate result.
        \begin{lemma}\label{lem:var_Hawkes_ind}
            Let us consider the positive measures
            \begin{align*} \mathcal{R}_{\indep}(\dd u,\dd v)& = 2k_1(\dd u)k_1(\dd v), \\
     \mathcal{G}_{\indep}(\dd u,\dd v,\dd v') &= 4\ind_{v\neq v'}k_1(\dd u)k_2(\dd v,\dd v') + 2k_1(\dd u)k_1(\dd v)k_1(\dd v'),  
     \end{align*}
    and
                \begin{multline*} \mathcal{Q}_{\indep}(\dd u,\dd u',\dd v,\dd v') = 2\ind_{v\neq v'}k_1(\dd u)k_1(\dd u')k_2(\dd v,\dd v')\\
                + 2\ind_{u\neq u'}\ind_{v\neq v'}k_2(\dd u,\dd u')k_2(\dd v,\dd v').
                \end{multline*}
    
            Then the variance of $\fp(X_1,X_2)$ satisfies the decomposition
            \begin{equation}\label{eq:dec_var_indep}
            \begin{array}{l}
                \displaystyle \Vz\big[ \fp(\bX_1,\bX_2) \big] = \frac{1}{M^2}\int_{[0,T]^2}\ind_{|u-v|\leq\delta}\mathcal{R}_{\indep}(\dd u,\dd v) \vspace{0.20cm}
                \\
                \displaystyle \hspace{3cm} +\, \frac{1}{M^3}\int_{[0,T]^3}\ind_{|u-v|\leq\delta}\ind_{|u-v'|\leq\delta}\mathcal{G}_{\indep}(\dd u,\dd v,\dd v') \vspace{0.20cm} 
                \\
                \displaystyle \hspace{3cm} +\, \frac{1}{M^4}\int_{[0,T]^4}\ind_{|u-v|\leq\delta}\ind_{|u'-v'|\leq\delta}\mathcal{Q}_{\indep}(\dd u,\dd u',\dd v,\dd v').
            \end{array}
            \end{equation}
        \end{lemma}
        \medskip
        
        \begin{proof}
            The neurons are \emph{i.i.d} for this model, and we recall that 
            \[ \fp(X_1,X_2) = \int_{[0,T]^2}\ind_{|u-v|\leq\delta}\dd X^1_1(u)\dd X^2_1(v) - \int_{[0,T]^2}\ind_{|u-v|\leq\delta}\dd X^1_1(u)\dd X^2_2(v) . \]
            Then the random variable $\fp(X_1,X_2)$ is centered and
            \begin{align*}
                \Vz\big[\fp(X_1,X_2)\big] & = \Ez\left[ \big| \fp(X_1,X_2)\big|^2 \right] 
                \\
                & = \Ez\bigg[ \int_{[0,T]^4} \ind_{|u-v|\leq\delta}\ind_{|u'-v'|\leq\delta}\dd X^1_1(u)\dd X^1_1(u')\dd X^2_1(v)\dd X^2_1(v') 
                \\
                &\qquad +  \int_{[0,T]^4}\ind_{|u-v|\leq\delta}\ind_{|u'-v'|\leq\delta}\dd X^1_1(u)\dd X^1_1(u')\dd X^2_2(v)\dd X^2_2(v')
                \\
                &\qquad - 2\int_{[0,T]^4}\ind_{|u-v|\leq\delta}\ind_{|u'-v'|\leq\delta}\dd X^1_1(u)\dd X^1_1(u')\dd X^2_1(v)\dd X^2_2(v') \bigg]
                \\
                & = 2\Ez\bigg[ \int_{[0,T]^4} \ind_{|u-v|\leq\delta}\ind_{|u'-v'|\leq\delta}\dd X^1_1(u)\dd X^1_1(u')\dd X^2_1(v)\dd X^2_1(v')
                \\
                &\qquad - \int_{[0,T]^4}\ind_{|u-v|\leq\delta}\ind_{|u'-v'|\leq\delta}\dd X^1_1(u)\dd X^1_1(u')\dd X^2_1(v)\dd X^2_2(v') \bigg] ,
            \end{align*}
            since 
            the stochastic integral
            $\int\ind_{|u-v|\leq\delta}\ind_{|u'-v'|\leq\delta}\dd X^1_1(u)\dd X^1_1(u')\dd X^2_1(v)\dd X^2_1(v')$ has under the independent model the same distribution as 
            $\int\ind_{|u-v|\leq\delta}\ind_{|u'-v'|\leq\delta}\dd X^1_1(u)\dd X^1_1(u')\dd X^2_2(v)\dd X^2_2(v')$. By distinguishing simultaneous and non simultaneous jump times, it follows that
            \begin{align*}
                    & \Vz\big[ \fp(\bX_1,\bX_2) \big] \\
                    & = 2\Ez\bigg[ \int_{[0,T]^3}\ind_{|u-v|\leq\delta}\ind_{|u-v'|\leq\delta}\dd X^1_1(u)\dd X^2_1(v)\dd X^2_1(v') 
                    \\
                    &\qquad - \int_{[0,T]^3}\ind_{|u-v|\leq\delta}\ind_{|u-v'|\leq\delta}\dd X^1_1(u)\dd X^2_1(v)\dd X^2_2(v') \bigg] 
                    \\
                    &~ + 2\Ez\bigg[ \int_{[0,T]^4}\ind_{|u-v|\leq\delta}\ind_{|u'-v'|\leq\delta}\ind_{u\neq u'}\dd X^1_1(u)\dd X^1_1(u')\dd X^2_1(v)\dd X^2_1(v') 
                    \\
                    &\qquad - \int_{[0,T]^4}\ind_{|u-v|\leq\delta}\ind_{|u'-v'|\leq\delta}\ind_{u\neq u'}dX_1^1(u)\dd X^1_1(u')\dd X^2_1(v)\dd X^2_2(v') \bigg] ,
            \end{align*}
            and then
            \begin{align*}
                    &\Vz\big[ \fp(\bX_1,\bX_2) \big] \\
                    & = 2\Ez\bigg[ \int_{[0,T]^2}\ind_{|u-v|\leq\delta}\dd X^1_1(u)\dd X^2_1(v) 
                    \\
                    &\qquad + \int_{[0,T]^3}\ind_{|u-v|\leq\delta}\ind_{|u-v'|\leq\delta}\ind_{v\neq v'}\dd X^1_1(u)\dd X^2_1(v)\dd X^2_1(v') 
                    \\
                    &\qquad - \int_{[0,T]^3}\ind_{|u-v|\leq\delta}\ind_{|u-v'|\leq\delta}\dd X^1_1(u)\dd X^2_1(v)\dd X^2_2(v') \bigg] 
                    \\
                    &~ + 2\Ez\bigg[ \int_{[0,T]^3}\ind_{|u-v|\leq\delta}\ind_{|u'-v|\leq\delta}\ind_{u\neq u'}\dd X^1_1(u)\dd X^1_1(u')\dd X^2_1(v) 
                    \\
                    &\qquad + \int_{[0,T]^4}\ind_{|u-v|\leq\delta}\ind_{|u'-v'|\leq\delta}\ind_{u\neq u'}\ind_{v\neq v'}\dd X^1_1(u)\dd X^1_1(u')\dd X^2_1(v)\dd X^2_1(v')
                    \\
                    &\qquad - \int_{[0,T]^4}\ind_{|u-v|\leq\delta}\ind_{|u'-v'|\leq\delta}\ind_{u\neq u'}\dd X^1_1(u)\dd X^1_1(u')\dd X^2_1(v)\dd X^2_2(v') \bigg].
            \end{align*}
            We deduce that
            \begin{align*}
                \Vz\big[ \fp(\bX_1,\bX_2) \big] & = 2\bigg\{ \frac{1}{M^2}\int_{[0,T]^2}\ind_{|u-v|\leq\delta}\E[\dd N_u]\E[\dd N_v] 
                \\
                &~ + \frac{1}{M^3}\int_{[0,T]^3}\ind_{|u-v|\leq\delta}\ind_{|u-v'|\leq\delta}\ind_{v\neq v'}\E[\dd N_u]k_2(\dd v,\dd v') 
                \\
                &~ + \frac{1}{M^3}\int_{[0,T]^3}\ind_{|u-v|\leq\delta}\ind_{|u'-v|\leq\delta}\ind_{u\neq u'}\E[\dd N_v]\E[\dd N_u\dd N_{u'}] 
                \\
                &~ + \frac{1}{M^4}\int_{[0,T]^4}\ind_{|u-v|\leq\delta}\ind_{|u'-v'|\leq\delta}\ind_{u\neq u'}\ind_{v\neq v'}\E[\dd N_u\dd N_{u'}]k_2(\dd v,\dd v') \bigg\}.
            \end{align*}
            Using the fact that $\E[\dd N_u\dd N_v] = k_2(\dd u,\dd v) + \E[\dd N_u]\E[\dd N_v]$ and $\E[\dd N_u]=k_1(\dd u)$, it follows that the term in the third integral that is related to the second order cumulant measure is equal to the second integral. We obtain
            \begin{align*}
                \Vz\big[ \fp(\bX_1,\bX_2) \big] & = 2\bigg\{ \frac{1}{M^2}\int_{[0,T]^2}\ind_{|u-v|\leq\delta}k_1(\dd u)k_1(\dd v) 
                \\
                &~ + \frac{2}{M^3}\int_{[0,T]^3}\ind_{|u-v|\leq\delta}\ind_{|u-v'|\leq\delta}\ind_{v\neq v'}k_1(\dd u)k_2(\dd v,\dd v') 
                \\
                &~ + \frac{1}{M^3}\int_{[0,T]^3}\ind_{|u-v|\leq\delta}\ind_{|u'-v|\leq\delta}k_1(\dd u)k_1(\dd v)k_1(\dd u') 
                \\
                &~ + \frac{1}{M^4}\int_{[0,T]^4}\ind_{|u-v|\leq\delta}\ind_{|u'-v'|\leq\delta}\ind_{u\neq u'}\ind_{v\neq v'}k_1(\dd u)k_1(\dd u')k_2(\dd v,\dd v')  
                \\
                &~ + \frac{1}{M^4}\int_{[0,T]^4}\ind_{|u-v|\leq\delta}\ind_{|u'-v'|\leq\delta}\ind_{u\neq u'}\ind_{v\neq v'}k_2(\dd u,\dd u')k_2(\dd v,\dd v') \bigg\},
            \end{align*}
            ending the proof.
        \end{proof}
        
        \begin{proof}[Proof of (\ref{eq:var_mf_limit})]
 Recall that in the mean-field limit, neurons $1$ and $2$ are independent and identically distributed,  and their spike trains are Poisson processes with intensity $\frac{\nu}{1-\ell}$. The variance $\V_0$ of $\fp(X_1,X_2)$ under this limit model can then be deduced  from \eqref{eq:Vardiff}.
        \end{proof}
        
        Coming back to the Hawkes model, we deduce from \eqref{eq:dec_var_indep} in Lemma~\ref{lem:var_Hawkes_ind} and \eqref{eq:cumul_order-1} in Lemma~\ref{lem:cumulant} that
        \begin{align*}
                \Vz\big[ \fp(\bX_1,\bX_2) \big] - \V_0 & = 2\bigg\{ \frac{2}{M^3}\int_{[0,T]^3}\ind_{|u-v|\leq\delta}\ind_{|u-v'|\leq\delta}\ind_{v\neq v'}k_1(\dd u)k_2(\dd v,\dd v') 
                \\
                &~ +\, \frac{1}{M^4}\int_{[0,T]^4}\ind_{|u-v|\leq\delta}\ind_{|u'-v'|\leq\delta}\ind_{v\neq v'}k_1(\dd u)k_1(\dd u')k_2(\dd v,\dd v') 
                \\
                &~ +\, \frac{1}{M^4}\int_{[0,T]^4}\ind_{|u-v|\leq\delta}\ind_{|u'-v'|\leq\delta}\ind_{u\neq u'}\ind_{v\neq v'}k_2(\dd u,\dd u')k_2(\dd v,\dd v') \bigg\},
        \end{align*}
        which is non negative. Using the bound \eqref{eq:cumul_order-1234} in Lemma~\ref{lem:cumulant} for the second order cumulant, we obtain for a certain constant $C>0$ (that can change from one line to another)  
        \begin{align*}
            & \Vz\big[ \fp(\bX_1,\bX_2) \big] - \V_0 
            \\
            & \leq C\bigg\{ \frac{a}{M}\frac{\nu^2}{(1-\frac{a}{b})^3}\int_{[0,T]^3}e^{-(b-a)|v-v'|}\ind_{|u-v|\leq\delta}\ind_{|u-v'|\leq\delta}\dd u\dd v\dd v'
            \\
            &~ + \frac{a}{M}\frac{\nu^3}{(1-\frac{a}{b})^4}\int_{[0,T]^4}e^{-(b-a)|v-v'|}\ind_{|u-v|\leq\delta}\ind_{|u'-v'|\leq\delta}\dd u \dd u' \dd v \dd v'
            \\
            &~ + \frac{1}{M^2}\frac{a^2\nu^2}{(1-\frac{a}{b})^4} \int_{[0,T]^4}e^{-(b-a)|u-u'|}e^{-(b-a)|v-v'|}\ind_{|u-v|\leq\delta}\ind_{|u'-v'|\leq\delta}\dd u\dd u'\dd v\dd v' \bigg\}
            \\
            & \leq C\bigg\{ \frac{a}{M}\frac{\nu^2}{(1-\frac{a}{b})^3}\int_{[0,T]^2}e^{-(b-a)|v-v'|}\underbrace{\left[ (v\wedge v'+\delta)\wedge T - (v\vee v' - \delta)_+ \right]_+}_{\leq (2\delta - |v-v'|)_+}\dd v\dd v'
            \\
            &~ + \frac{a}{M}\frac{\nu^2}{(1-\frac{a}{b})^4}\left(\nu+\frac{a}{M}\right) \times \\
            & \quad   \times \int_{[0,T]^2}e^{-(b-a)|v-v'|}\underbrace{\big[ (v+\delta)\wedge T - (v-\delta)_+ \big]}_{\leq 2\delta}\underbrace{\big[ (v'+\delta)\wedge T - (v'-\delta)_+ \big]}_{\leq 2\delta}\dd v\dd v' \bigg\}.
        \end{align*}
        It follows that
        \begin{align*}
            \Vz\big[ \fp(\bX_1,\bX_2) \big] - \V_0 & \leq C\bigg\{ \frac{a}{M}\frac{\nu^2}{(1-\frac{a}{b})^3}\int_0^T\int_0^Te^{-(b-a)|v-v'|} (2\delta - |v-v'|)_+\dd v\dd v'
            \\
            &~ + \frac{a}{M}\frac{\nu^2}{(1-\frac{a}{b})^4}\left(\nu+\frac{a}{M}\right)\times 4\delta^2\int_0^T\int_0^Te^{-(b-a)|v-v'|}\dd v\dd v' \bigg\}
            \\
            & = C\bigg\{ \frac{a}{M}\frac{\nu^2}{(1-\frac{a}{b})^3}\times 2\int_0^T\int_0^{v\wedge(2\delta)}e^{-(b-a)z} (2\delta - z)dz\dd v
            \\
            &~ + \frac{a}{M}\frac{\nu^2}{(1-\frac{a}{b})^4}\left(\nu+\frac{a}{M}\right)\times 4\delta^2\int_0^T\int_0^Te^{-(b-a)|v-v'|}\dd v\dd v' \bigg\}\\
            & \le C\bigg\{ \frac{a}{M}\frac{\nu^2}{(1-\frac{a}{b})^3}\times 2\int_0^T\int_0^{2\delta}e^{-(b-a)z} (2\delta )dz\dd v
            \\
            &~ + \frac{a}{M}\frac{\nu^2}{(1-\frac{a}{b})^4}\left(\nu+\frac{a}{M}\right)\times 4\delta^2\int_0^T\int_0^Te^{-(b-a)|v-v'|}\dd v\dd v' \bigg\}.
        \end{align*}
        Since 
        \[ (2\delta ) \int_0^T\int_0^{2\delta}e^{-(b-a)z} dz\dd v \le 2 \delta T \frac{1}{b-a} ( 1 - e^{ - (b-a) 2 \delta }) \le 4 \delta^2 T  ,
        \] 
        where we have used that $ 1 - e^{ - x}\le x, $
        this implies  that
        \begin{align*}
            \Vz\big[ \fp(\bX_1,\bX_2) \big] - \V_0 & \leq C\bigg\{ \frac{a}{M}\frac{\nu^2}{(1-\frac{a}{b})^3} 
            \delta^2 T  
            \\
            &~ + \frac{a}{M}\frac{\nu^2}{(1-\frac{a}{b})^4}\left(\nu+\frac{a}{M}\right)\times \frac{\delta^2T}{b-a}\left[ 1 - \frac{1-e^{-(b-a)T}}{(b-a)T} \right] \bigg\}.
        \end{align*}
        Using the inequality $1 - \frac{1-e^{-x}}{x} \leq 1\wedge\frac{x}{2}$ for any $x>0$, we obtain
        \begin{align*}
            & \Vz\big[ \fp(\bX_1,\bX_2) \big] - \V_0 
            \\
            & \leq C\bigg\{ \frac{a}{M}\frac{\nu^2}{(1-\frac{a}{b})^3} \delta^2 T + \frac{a}{M}\frac{\nu^2}{(1-\frac{a}{b})^4}\left(\nu+\frac{a}{M}\right)\times \delta^2T\left( \frac{T}{2}\wedge\frac{1}{b-a} \right) \bigg\}
            \\
            & = C\frac{aT}{M}\frac{\nu^2\delta^2}{(1-\frac{a}{b})^3}\left\{ 1+ \frac{\nu+\frac{a}{M}}{1-\frac{a}{b}}\left( \frac{T}{2}\wedge\frac{1}{b-a} \right) \right\}
            \\
            & \leq C\frac{aT}{M}\frac{\nu^2\delta^2}{(1-\frac{a}{b})^3}\left\{ 1 + \frac{\nu+\frac{a}{M}}{b(1-\frac{a}{b})^2} \right\} = C\frac{aT}{M}\frac{\nu^2\delta^2}{(1-\ell)^3}\left\{ 1 + \frac{\nu+\frac{a}{M}}{b (1-\ell)^2} \right\},
        \end{align*}
        which is \eqref{eq:var_Hawkes_ind} and ends the proof.
    
    \subsection{Proof of Eq. \texorpdfstring{\eqref{eq:var_Hawkes}}{}}
        We are now interested in the variance of $\fp(X_1,X_2)$ for the observed model, that is the one for which the neurons are in the same network. Then, the proof of the bound \eqref{eq:var_Hawkes} follows from the following intermediate result.
        \begin{lemma}
            Let us consider the positive measures defined by
            \begin{align*}
                & \mathcal{R}(\dd u,\dd v) = \ind_{u\neq v}k_2(\dd u,\dd v) + 2k_1(\dd u)k_1(\dd v),
                \\
                & \mathcal{G}(\dd u,\dd v,\dd v') = 2\ind_{\Delta_3}\big\{ k_3(\dd u,\dd v,\dd v') + 2k_1(\dd u)k_2(\dd v,\dd v') 
                \\
                & \qquad + k_1(\dd v)k_2(\dd u,\dd v') + k_1(\dd u)k_1(\dd v)k_1(\dd v') \big\},
                \\
                & \mathcal{Q}(\dd u,\dd u',\dd v,\dd v')  = \ind_{\Delta_4}\big\{ k_4(\dd u,\dd u',\dd v,\dd v') + 2k_1(\dd u)k_3(\dd u',\dd v,\dd v') 
                \\
                & \qquad + 3k_2(\dd u,\dd u')k_2(\dd v,\dd v') + 2k_1(\dd u)k_1(\dd u')k_2(\dd v,\dd v') \big\}.
            \end{align*}
            Then the variance of $\fp(X_1,X_2)$ satisfies the decomposition
            \begin{equation}\label{eq:Var_lemma}
            \begin{array}{l}
                \displaystyle \Vo\big[ \fp(\bX_1,\bX_2) \big] = \frac{1}{M^2}\int_{[0,T]^2}\ind_{|u-v|\leq\delta}\mathcal{R}(\dd u,\dd v) \vspace{0.20cm}
                \\
                \displaystyle \hspace{2.6cm} +\, \frac{1}{M^3}\int_{[0,T]^3}\ind_{|u-v|\leq\delta}\ind_{|u-v'|\leq\delta}\mathcal{G}(\dd u,\dd v,\dd v') \vspace{0.20cm} 
                \\
                \displaystyle \hspace{2.6cm} +\, \frac{1}{M^4}\int_{[0,T]^4}\ind_{|u-v|\leq\delta}\ind_{|u'-v'|\leq\delta}\mathcal{Q}(\dd u,\dd u',\dd v,\dd v').
            \end{array}
            \end{equation}
        \end{lemma}
        
        \begin{proof}
            The proof of this result uses the same ideas as the one of Lemma~\ref{lem:var_Hawkes_ind} taking into account the fact that $X^1_1$ and $X^2_1$ belong to the same network under the observed model. The moment of order 2 of $\fp(X_1,X_2)$ is then given by 
            \begin{align*}
                \Eo\left[ \big|\fp(\bX_1,\bX_2)\big|^2 \right] & = \Eo\bigg[ \int_{[0,T]^4}\ind_{|u-v|\leq\delta}\ind_{|u'-v'|\leq\delta}\dd X^1_1(u)\dd X^1_1(u')\dd X^2_1(v)\dd X^2_1(v') 
                \\
                &\qquad + \int_{[0,T]^4}\ind_{|u-v|\leq\delta}\ind_{|u'-v'|\leq\delta}\dd X^1_1(u)\dd X^1_1(u')\dd X^2_2(v)\dd X^2_2(v') 
                \\
                &\qquad - 2\int_{[0,T]^4}\ind_{|u-v|\leq\delta}\ind_{|u'-v'|\leq\delta}\dd X^1_1(u)\dd X^1_1(u')\dd X^2_1(v)\dd X^2_2(v') \bigg],
            \end{align*}
            and by distinguishing simultaneous jumps, we get
            \begin{align*}
                & \Eo\left[ \big|\fp(\bX_1,\bX_2)\big|^2 \right] = \Eo\bigg[ \int_{[0,T]^2}\ind_{|u-v|\leq\delta}\dd X^1_1(u)\dd X^2_1(v) \\
                & + \int_{[0,T]^2}\ind_{|u-v|\leq\delta}\ind_{|u-v'|\leq\delta}\ind_{v\neq v'}\dd X^1_1(u)\dd X^2_1(v)\dd X^2_1(v') 
                \\
                &+ \int_{[0,T]^2}\ind_{|u-v|\leq\delta}\dd X^1_1(u)\dd X^2_2(v) \\
                &  + \int_{[0,T]^3}\ind_{|u-v|\leq\delta}\ind_{|u-v'|\leq\delta}\ind_{v\neq v'}\dd X^1_1(u)\dd X^2_2(v)\dd X^2_2(v')
                \\ 
                & - 2\int_{[0,T]^3}\ind_{|u-v|\leq\delta}\ind_{|u-v'|\leq\delta}\dd X^1_1(u)\dd X^2_1(v)\dd X^2_2(v') 
                \\
                & + \int_{[0,T]^3}\ind_{|u-v|\leq\delta}\ind_{|u'-v|\leq\delta}\ind_{u\neq u'}\dd X^1_1(u)\dd X^1_1(u')\dd X^2_1(v) 
                \\
                & + \int_{[0,T]^4}\ind_{|u-v|\leq\delta}\ind_{|u'-v'|\leq\delta}\ind_{u\neq u'}\ind_{v\neq v'}\dd X^1_1(u)\dd X^1_1(u')\dd X^2_1(v)\dd X^2_1(v')
                \\
                & + \int_{[0,T]^3}\ind_{|u-v|\leq\delta}\ind_{|u'-v|\leq\delta}\ind_{u\neq u'}\dd X^1_1(u)\dd X^1_1(u')\dd X^2_2(v) 
                \\
                & + \int_{[0,T]^4}\ind_{|u-v|\leq\delta}\ind_{|u'-v'|\leq\delta}\ind_{u\neq u'}\ind_{v\neq v'}\dd X^1_1(u)\dd X^1_1(u')\dd X^2_2(v)\dd X^2_2(v') 
                \\
                & - 2\int_{[0,T]^4}\ind_{|u-v|\leq\delta}\ind_{|u'-v'|\leq\delta}\ind_{u\neq u'}\dd X^1_1(u)\dd X^1_1(u')\dd X^2_1(v)\dd X^2_2(v') \bigg] .
            \end{align*}
            Since ${\mathcal L}(X^1_1,X^2_1) = {\mathcal L}(X^2_1,X^1_1)$ by exchangeability of the neurons, the random variables given by \[\int\ind_{|u-v|\leq\delta}\ind_{|u-v'|\leq\delta}\ind_{v\neq v'}\dd X^1_1(u)\dd X^2_1(v)\dd X^2_1(v')\] and \[\int\ind_{|u-v|\leq\delta}\ind_{|u'-v|\leq\delta}\ind_{u\neq u'}\dd X^1_1(u)\dd X^1_1(u')\dd X^2_1(v)\] 
            have the same distribution. The same reason implies that the random variables 
            \begin{align*}
            \int\ind_{|u-v|\leq\delta}\ind_{|u-v'|\leq\delta}\ind_{v\neq v'}\dd X^1_1(u)\dd X^2_2(v)\dd X^2_2(v')\\
            \text{ and }\quad 
            \int\ind_{|u-v|\leq\delta}\ind_{|u'-v|\leq\delta}\ind_{u\neq u'}\dd X^1_1(u)\dd X^1_1(u')\dd X^2_2(v)
            \end{align*}
            have the same distribution. As a consequence, we obtain
            \begin{align*}
                &\Eo\left[ \big|\fp(\bX_1,\bX_2)\big|^2 \right] \\
                & = \Eo\bigg[ \int\ind_{|u-v|\leq\delta}\dd X^1_1(u)\dd X^2_1(v) + 2\int\ind_{|u-v|\leq\delta}\ind_{|u-v'|\leq\delta}\ind_{v\neq v'}\dd X^1_1(u)\dd X^2_1(v)\dd X^2_1(v') 
                \\
                & + \int\ind_{|u-v|\leq\delta}\dd X^1_1(u)\dd X^2_2(v) + 2\int\ind_{|u-v|\leq\delta}\ind_{|u-v'|\leq\delta}\ind_{v\neq v'}\dd X^1_1(u)\dd X^2_2(v)\dd X^2_2(v')
                \\ 
                & - 2\int\ind_{|u-v|\leq\delta}\ind_{|u-v'|\leq\delta}\dd X^1_1(u)\dd X^2_1(v)\dd X^2_2(v') 
                \\
                & + \int\ind_{|u-v|\leq\delta}\ind_{|u'-v'|\leq\delta}\ind_{u\neq u'}\ind_{v\neq v'}\dd X^1_1(u)\dd X^1_1(u')\dd X^2_1(v)\dd X^2_1(v')
                \\
                & + \int\ind_{|u-v|\leq\delta}\ind_{|u'-v'|\leq\delta}\ind_{u\neq u'}\ind_{v\neq v'}\dd X^1_1(u)\dd X^1_1(u')\dd X^2_2(v)\dd X^2_2(v') 
                \\
                & - 2\int\ind_{|u-v|\leq\delta}\ind_{|u'-v'|\leq\delta}\ind_{u\neq u'}\dd X^1_1(u)\dd X^1_1(u')\dd X^2_1(v)\dd X^2_2(v') \bigg] .
            \end{align*}
            Then using \eqref{eq:mean_cumul_hawkes}, that is, 
            \[ 
            \Eo\big[ \fp(X_1,X_2) \big] = \frac{1}{M^2}\int_{[0,T]^2}\ind_{|u-v|\leq\delta}\ind_{u\neq v}k_2(\dd u,\dd v), 
            \]
            it follows that
            \begin{align*}
                \Vo\big[ \fp(\bX_1,\bX_2) \big] & = \frac{1}{M^2}\int_{[0,T]}\ind_{|u-v|\leq\delta}\mathcal{R}(\dd u,\dd v)
                \\
                &~ + \frac{1}{M^3}\int_{[0,T]^3}\ind_{|u-v|\leq\delta}\ind_{|u-v'|\leq\delta}\mathcal{G}(\dd u,\dd v,\dd v') 
                \\
                &~ + \frac{1}{M^4}\int_{[0,T]^4}\ind_{|u-v|\leq\delta}\ind_{|u'-v'|\leq\delta}\mathcal{Q}(\dd u,\dd u',\dd v,\dd v') ,
            \end{align*}
            where we set 
            \begin{align*}
                \mathcal{R}(\dd u,\dd v) &  = \ind_{u\neq v}\E[\dd N_u\dd N_v] + \E[\dd N_u]\E[\dd N_v] ,
                \\
                \mathcal{G}(\dd u,\dd v,\dd v') & = 2\ind_{\Delta_3}\bigg\{ \E[\dd N_u\dd N_v\dd N_{v'}] + \E[\dd N_u]\E[\dd N_v\dd N_{v'}] \\
                & - \E[\dd N_{v'}]\E[\dd N_u\dd N_v] \bigg\} ,
                \\
                \mathcal{Q}(\dd u,\dd u',\dd v,\dd v') & = \ind_{\Delta_4}\bigg\{\E[\dd N_u\dd N_{u'}\dd N_v\dd N_{v'}] + \E[\dd N_u\dd N_{u'}]\E[\dd N_v\dd N_{v'}]
                \\
                &~ - 2\E[\dd N_{v'}]\E[\dd N_u\dd N_{u'}\dd N_{v}] - k_2(\dd u,\dd v)k_2(\dd u',\dd v') \bigg\}.
            \end{align*}
            
            Since $k_1(\dd u) = \E[\dd N_u]$ and $\E[\dd N_u\dd N_v] = k_2(\dd u,\dd v) + k_1(\dd u)k_1(\dd v)$, we get
            \[  \mathcal{R}(\dd u, \dd v) = \ind_{u\neq v}k_2(\dd u,\dd v) + 2k_1(\dd u)k_1(\dd v). \]
            In addition, using \eqref{eq:dualformula},
            \begin{align*}
                \mathcal{G}(\dd u,\dd v,\dd v') & = 2\ind_{\Delta_3}\bigg\{ k_3(\dd u,\dd v,\dd v') + k_1(\dd u)k_2(\dd v,\dd v') + k_1(\dd v)k_2(\dd u,\dd v') 
                \\
                &\quad + k_1(\dd v')k_2(\dd u,\dd v) + k_1(\dd u)k_1(\dd v)k_1(\dd v') 
                \\
                &\quad + k_1(\dd u)\bigg( k_2(\dd v,\dd v') + k_1(\dd v)k_1(\dd v') \bigg)
                \\
                &\quad - k_1(\dd v')\bigg( k_2(\dd u,\dd v) + k_1(\dd u)k_1(\dd v) \bigg) \bigg\} ,
            \end{align*}
            such that
            \begin{align*}
                \mathcal{G}(\dd u,\dd v,\dd v') & = 2\ind_{\Delta_3}\bigg\{ k_3(\dd u,\dd v,\dd v') + 2k_1(\dd u)k_2(\dd v,\dd v')
                \\
                &~ + k_1(\dd v)k_2(\dd u,\dd v') + k_1(\dd u)k_1(\dd v)k_1(\dd v') \bigg\}  .
               \end{align*}
            The following steps are devoted to obtain a simpler representation of the measure $\mathcal{Q}(\dd u, \dd u, \dd v, \dd v')$. Using \eqref{eq:cumrec},  we get
            \begin{align*}
                \mathcal{Q}(\dd u,\dd u',\dd v,\dd v') & = \ind_{\Delta_4}\bigg\{ \sum_{\substack{\mathbf{p} \textrm{ partition of }\\ \{u,u',v,v'\}}}\bigotimes_{B\in\mathbf{p}}k_{\#B}(\bigotimes_{u\in B}\dd u) 
                \\
                & + \bigg( k_2(\dd u,\dd u') + k_1(\dd u)k_1(\dd u') \bigg)\bigg( k_2(\dd v,\dd v') + k_1(\dd v)k_1(\dd v') \bigg) 
                \\
                & - 2k_1(\dd v')\sum_{\substack{\mathbf{p} \textrm{ partition of }\\ \{u,u',v\}}}\bigotimes_{B\in\mathbf{p}}k_{\#B}(\bigotimes_{u\in B}\dd u) - k_2(\dd u,\dd v)k_2(\dd u',\dd v') \bigg\}.
            \end{align*}
            We extract from the first sum the partitions that contain the singleton $\{v'\}$. The sum over this subset of partitions gives $k_1(\dd v')\sum_{\substack{\mathbf{p} \textrm{ partition of }\\ \{u,u',v\}}}\bigotimes_{B\in\mathbf{p}}k_{\#B}(\bigotimes_{u\in B}\dd u)$, and then
            \begin{align*}
                &\mathcal{Q}(\dd u,\dd u',\dd v,\dd v')  = \ind_{\Delta_4}\bigg\{ \sum_{\substack{\mathbf{p} \textrm{ partition of }\\ \{u,u',v,v'\}\\ \textrm{s.t }\{v'\}\notin\mathbf{p}}}\bigotimes_{B\in\mathbf{p}}k_{\#B}(\bigotimes_{u\in B}\dd u) + k_2(\dd u,\dd u')k_2(\dd v,\dd v') 
                \\
                & + k_2(\dd u,\dd u')k_1(\dd v)k_1(\dd v') + k_1(\dd u)k_1(\dd u')k_2(\dd v,\dd v') + k_1(\dd u)k_1(\dd u')k_1(\dd v)k_1(\dd v')
                \\
                & - k_1(\dd v')\sum_{\substack{\mathbf{p} \textrm{ partition of }\\ \{u,u',v\}}}\bigotimes_{B\in\mathbf{p}}k_{\#B}(\bigotimes_{u\in B}\dd u) - k_2(\dd u,\dd v)k_2(\dd u',\dd v') \bigg\}.
            \end{align*}
            We  extract the partition $\{\{u,v\},\{u',v'\}\}$ from the first sum, and the partitions $\{\{u\},\{u'\},\{v\}\}$ and $\{\{u,u'\},\{v\}\}$ from the second one. Then we obtain
            \begin{multline}\label{eq:dec_q}
                \mathcal{Q}(\dd u,\dd u',\dd v,\dd v') = \ind_{\Delta_4}\bigg\{ \sum_{\mathbf{p}\in\Pi^1(u,u',v,v')}\bigotimes_{B\in\mathbf{p}}k_{\#B}(\bigotimes_{u\in B}\dd u) + k_2(\dd u,\dd u')k_2(\dd v,\dd v') 
                \\
                + k_1(\dd u)k_1(\dd u')k_2(\dd v,\dd v')  - k_1(\dd v')\sum_{\mathbf{p} \in\Pi^2(u,u',v)}\bigotimes_{B\in\mathbf{p}}k_{\#B}(\bigotimes_{u\in B}\dd u)  \bigg\}
            \end{multline}
            where 
            \begin{align*}
                \Pi^1(u,u',v,v') & := \{\mathbf{p} \textrm{ partition of }\{u,u',v,v'\} : \{v'\}\notin\mathbf{p} \textrm{ and }\mathbf{p}\neq\{\{u,v\},\{u',v'\}\} \}
                \\
                \Pi^2(u,u',v) & := \{\mathbf{p} \textrm{ partition of }\{u,u',v\} : |\mathbf{p}|<3 \textrm{ and }\mathbf{p}\neq\{\{u,u'\},\{v\}\} \}
                \\
                &~ = \bigg\{ \big\{ \{u,u',v\} \big\}, \big\{ \{u\}, \{u',v\} \big\}, \big\{ \{u'\}, \{u,v\} \big\} \bigg\}.
            \end{align*}
            We now  verify that the negative measure in the expression of $\mathcal{Q}(\dd u,\dd u',\dd v,\dd v')$ is compensated when we compute the integral $\int_{[0,T]^4}\ind_{|u-v|\leq\delta}\ind_{|u'-v'|\leq\delta}q(\dd u,\dd u',\dd v,\dd v')$. This directly holds by symmetry. Indeed, the integral with the partition $\big\{ \{u,u',v\} \big\}$ in the sum on $\Pi^2(u,u',v)$ in \eqref{eq:dec_q} that is 
            \[ \int_{[0,T]^4}\ind_{|u-v|\leq\delta}\ind_{|u'-v'|\leq\delta}\ind_{\Delta_4}k_1(\dd v')k_3(\dd u,\dd u',\dd v) \] 
            does not depend on the order of the variables in the measure $k_1(\dd v')k_3(\dd u,\dd u',\dd v)$. In particular, the value of the integral remains the same with what we get by using the partitions $\big\{\{u\},\{u',v,v'\}\big\}$, $\big\{\{u'\},\{u,v,v'\}\big\}$ and $\big\{\{v\},\{u,u',v'\}\big\}$ in the sum on $\Pi^1(u,u',v,v')$ in \eqref{eq:dec_q}. The same happens with: 
            \begin{itemize}
                \item the values of the integral with the partition $\big\{ \{u\}, \{u',v\} \big\}$ in the sum on $\Pi^2(u,u',v)$ and the partitions $\big\{ \{u'\}, \{v\}, \{u,v'\} \big\}$, $\big\{ \{u'\}, \{u\}, \{v,v'\} \big\}$ in the sum on $\Pi^1(u,u',v,v')$ in \eqref{eq:dec_q};
                
                \item the values of the integral with the partition $\big\{ \{u'\}, \{u,v\} \big\}$ in the sum on $\Pi^2(u,u',v)$ and the partition $\big\{ \{u\}, \{v\}, \{u',v'\} \big\}$ in the sum on $\Pi^1(u,u',v,v')$ in \eqref{eq:dec_q};
                
                \item the values of the integrals with the partitions $\big\{ \{u,u'\}, \{v,v'\} \big\}$ and $\big\{ \{u,v'\}, \{u',v\} \big\}$ in the sum on $\Pi^1(u,u',v,v')$ in \eqref{eq:dec_q}.
            \end{itemize}
            We then deduce from \eqref{eq:dec_q} that
            \begin{align*}
                & \int_{[0,T]^4}\ind_{|u-v|\leq\delta}\ind_{|u'-v'|\leq\delta}\mathcal{Q}(\dd u,\dd u',\dd v,\dd v') 
                \\
                & = \int_{[0,T]^4}\ind_{|u-v|\leq\delta}\ind_{|u'-v'|\leq\delta}\ind_{\Delta_4}\bigg\{ k_4(\dd u,\dd u',\dd v,\dd v') + 2k_1(\dd u)k_3(\dd u',\dd v,\dd v') 
                \\
                &\qquad\qquad + 3k_2(\dd u,\dd u')k_2(\dd v,\dd v') + 2k_1(\dd u)k_1(\dd u')k_2(\dd v,\dd v') \bigg\},
            \end{align*}
            that ends the proof.
            
        \end{proof}
        From this Lemma and using the value of the variance $\V_0$ of $\fp(X_1,X_2)$ under the mean-field limit given in \eqref{eq:var_mf_limit}, we deduce that
        \begin{align*}
            & \Vo\big[ \fp(\bX_1,\bX_2) \big] - \V_0 
            \\
            &\quad = \frac{1}{M^2}\int_{[0,T]^2}\ind_{|u-v|\leq\delta}\ind_{u\neq v}k_2(\dd u,\dd v) 
            \\
            &\quad~ + \frac{1}{M^3}\int_{[0,T]^3}\ind_{|u-v|\leq\delta}\ind_{|u-v'|\leq\delta}\bigg(\mathcal{G}(\dd u,\dd v,\dd v') - 2k_1(\dd u)k_1(\dd v)k_1(\dd v')\bigg)
            \\
            &\quad~ + \frac{1}{M^4}\int_{[0,T]^4}\ind_{|u-v|\leq\delta}\ind_{|u'-v'|\leq\delta}\mathcal{Q}(\dd u,\dd u',\dd v,\dd v').
        \end{align*}
        This difference is positive. Using \eqref{eq:cumul_order-1234} in Lemma~\ref{lem:cumulant}, it satisfies the following bound for a certain constant $C>0$ (that can change from a line to another)
        \begin{align*}
            &\Vo\big[ \fp(\bX_1,\bX_2) \big] - \V_0 
            \\
            &~ \leq C\bigg\{ \frac{a\nu}{M(1-\frac{a}{b})^2}\int_{[0,T]^2}\ind_{|u-v|\leq\delta}\dd u\dd v
            \\
            &\quad + \frac{a\nu}{M(1-\frac{a}{b})^3}\int_{[0,T]^3}\ind_{|u-v|\leq\delta}\ind_{|u-v'|\leq\delta} 
            \\
            &\qquad\quad \times\bigg[ \frac{a}{M}e^{-(b-a)[\max\{u,v,v'\} - \min\{u,v,v'\}]} + \nu\bigg(e^{-(b-a)|v-v'|} + e^{-(b-a)|u-v'|} \bigg) \bigg]\dd u\dd v\dd v'
            \\
            &\quad + \frac{a\nu}{M(1-\frac{a}{b})^4}\int_{[0,T]^4}\ind_{|u-v|\leq\delta}\ind_{|u'-v'|\leq\delta}\bigg[ \frac{a^2}{M^2}e^{-(b-a)[\max\{u,u',v,v'\}-\min\{u,u',v,v'\}]} 
            \\
            &\qquad\quad + \frac{a\nu}{M}\bigg(e^{-(b-a)[\max\{u',v,v'\}-\min\{u',v,v'\}} + e^{-(b-a)(|u-u'| + |v-v'|)} \bigg) + \nu^2e^{-(b-a)|v-v'|} \bigg]\\
            & \qquad\quad\qquad \qquad\quad\qquad\quad\qquad\quad\qquad\quad\qquad\quad \dd u\dd u'\dd v\dd v' \bigg\},
        \end{align*}
        and then 
        \begin{align*}
            & \Vo\big[ \fp(\bX_1,\bX_2) \big] - \V_0  \leq C\bigg\{ \frac{a\nu\delta T}{M(1-\frac{a}{b})^2} + \frac{a\nu(\nu+\frac{a}{M})}{M(1-\frac{a}{b})^3}\underbrace{\int_{[0,T]^3}\ind_{|u-v|\leq\delta}\ind_{|u-v'|\leq\delta}\dd u\dd v\dd v'}_{\leq 4\delta^2T}
            \\
            &\quad + \frac{a\nu(\nu+\frac{a}{M})^2}{M(1-\frac{a}{b})^4}\int_{[0,T]^4}e^{-(b-a)|v-v'|}\ind_{|u-v|\leq\delta}\ind_{|u'-v'|\leq\delta}\dd u\dd u'\dd v\dd v' \bigg\}
            \\
            &~ \leq C\bigg\{ \frac{a\nu\delta T}{M(1-\frac{a}{b})^2} + \delta^2T\frac{a\nu(\nu+\frac{a}{M})}{M(1-\frac{a}{b})^3} + \delta^2\frac{a\nu(\nu+\frac{a}{M})^2}{M(1-\frac{a}{b})^4}\underbrace{\int_{[0,T]^2}e^{-(b-a)|v-v'|}\dd v\dd v'}_{\leq\, 2T/(b-a)}\bigg\}
            \\
            &~ \leq C\frac{aT}{M} \frac{\nu\delta}{(1-\frac{a}{b})^2}\bigg\{ 1 + \frac{(\nu+\frac{a}{M})\delta}{1-\frac{a}{b}}\left[ 1 + \frac{b(\nu+\frac{a}{M})}{(b-a)^2} \right] \bigg\}.
        \end{align*}
        Replacing $ a/b$ by $ \ell $ this implies  \eqref{eq:var_Hawkes} and ends the proof. 
    
    \subsection{\texorpdfstring{Proof of Theorem~\ref{theo:Hawkes}}{Proof of the lower bound}} 
    
        Let us first notice that the variance $\V_0$ of $\fp(X_1,X_2)$ under the mean-field limit (see \eqref{eq:var_mf_limit}) satisfies the bound
        \[ \V_0 \leq \frac{4\nu^2\delta T}{(1-\ell)^2}\left( 1 + \frac{2\nu\delta}{1-\ell} \right). \]
        We deduce from 
        \eqref{eq:var_Hawkes_ind} and \eqref{eq:var_Hawkes} that there exists a constant $C>0$ such that the criterion \eqref{eq:20230127_1} in Lemma~\ref{lem:critgen} which ensures the 
        control of the type II error of the permutation test $\ind_{\bC_n(\X_n)>q_{1-\alpha}(\X_n)}$ by $\beta$ holds if
        \[
             \Delta_\varphi \geq  
             \frac{C}{\sqrt{n\alpha\beta}}\sqrt{\frac{\nu^2\delta T}{(1-\ell)^2}\left( 1 + \frac{2\nu\delta}{1-\ell }\right) + \frac{aT}{M} \frac{\nu\delta}{(1-\ell)^2}\bigg\{ 1 + \frac{(\nu+\frac{a}{M})\delta}{1-\ell} \left[ 1 + \frac{\nu+\frac{a}{M}}{b (1-\ell)^2} \right] \bigg\}} .  
        \]
        For the sake of simplicity, we denote by $\mathcal{K}$ the term in the square root. Then we obtain
        \begin{align*}
            \mathcal{K} & = \frac{\nu\delta T}{(1-\ell)^2}\bigg\{ \nu\left( 1 + \frac{2\nu\delta}{1-\ell} \right) + \frac{a}{M}\bigg\{ 1 + \frac{(\nu+\frac{a}{M})\delta}{1-\ell} \left[ 1 + \frac{\nu+\frac{a}{M}}{b(1-\ell)^2} \right] \bigg\} \bigg\}
            \\
            & \leq 2\frac{\nu(\nu+\frac{a}{M})\delta T}{(1-\ell)^2} \bigg\{ 1 + \frac{(\nu+\frac{a}{M})\delta}{1-\ell} + \frac{a(\nu+\frac{a}{M})\delta}{b M(1-\ell)^3} \bigg\}
            \\
            & = 2\frac{\nu(\nu+\frac{a}{M})\delta T}{(1-\ell)^2} \bigg\{ 1 + \frac{(\nu+\frac{a}{M})\delta}{1-\ell}\bigg[ 1 + \frac{\ell}{M(1-\ell)^2} \bigg] \bigg\}.
        \end{align*}
        Using that $ \nu \le \nu + a/M,$ this implies the condition \eqref{eq:crit_Hawkes}. 
        
        Finally, to prove \eqref{eq:min_size_Hawkes}, recall that \eqref{eq:bound_Hawkes} gives the lower bound  
        \[
        \Delta_\varphi\geq   \frac{\nu\delta\ell}{M(1-\ell)^3} \frac{b(1-\ell)T}{4}  \frac{1-e^{-b(1-\ell)\delta}}{b(1-\ell)\delta}= \frac14  \frac{\nu \ell T}{M(1- \ell)^3}(1-e^{-b(1-\ell)\delta}) .
        \]
        
        Therefore we only need $n$ large enough such that the right hand side above is larger than the right hand side of \eqref{eq:crit_Hawkes}, implying the bound given in \eqref{eq:min_size_Hawkes}. This concludes the proof.

\subsection{Remaining proofs}\label{app:probabilisticresults}
We now give the proof of the fact that the process $Z^2 $ introduced in Model~\ref{mod:jittering} is still a Poisson process of intensity $ \eta^1.$

\begin{prop}\label{prop:z2estunpoisson}
Let $ Z^1$ be a Poisson process on $\R$ with intensity $ \eta^1$. Denote by $\ldots <  u_{k-1} < u_k < u_{k+1} < \ldots $ its points. Let moreover $ (\xi_i)_{i \in \mathbb{Z}}$ be an \textit{i.i.d.} sequence of real valued random variables, independent of $Z$. Then the point process $Z^2$, having the points $ u_k + \xi_k, k \in \mathbb{Z},$ is still a Poisson process with intensity $\eta^1$. 
\end{prop}

\begin{proof}
By construction, the intensity measure $\hat \mu $ of $ Z^2$ is given by $ \hat \mu (A   ) = \eta^1 \int_{ \R}  \Po ( t + \xi \in A ) \dd t$, for any Borel subset $A$ of $\R$. Using the change of variables $ s = \xi + t, $ we have
\[ \int_{ \R}  \Po ( t + \xi \in A )  \dd t= \E \int_\R \mathbbm{1}_A ( \xi + t ) \dd t = \E \int_\R \mathbbm{1}_A ( s) \dd s = \int_A \dd s , 
\]
which implies that the intensity measure of $ Z^2$ is still given by $ \eta^1 \dd t$. A classical result (see e.g. Theorem II.4.5 of \cite{jacod_limit_2003}) implies then that $ Z^2$ is a Poisson process of intensity $\eta^1$. 
\end{proof}

We close this section with an 

 \begin{proof}[Idea of the proof of Lemma~\ref{lem:Poisson}]
                Let us start by verifying that the stochastic integral in 
                \eqref{eq:stoch_int} makes sense. We construct it by introducing an order 
                between the instants $s_1,\cdots,s_{l}$. Using the partition 
                of $\Delta_{l} = \cup_{\pi\in\mathbb{S}_l} \Delta_{l}^{\pi}$, where for 
                any \(\pi\in\eS_{l}\) we set
                $ \Delta_{l}^{\pi} = \{(s_1,\cdots,s_{l})\in\R^{l}: s_{\pi(i)}<s_{\pi(j)} \textrm{ for any }i<j\} , $ we have
                \begin{align*}
                    & \int_{\Delta_{l}}F\left(s_1,\cdots, s_{l}, \xi_{N_{s_1}},\cdots,\xi_{N_{s_{l}}}\right)\dd N_{s_1}\cdots \dd N_{s_{l}}
                    \\
                    & = \sum_{\pi\in\eS_{l}}\int_{\Delta_{l}^{\pi}}F(s_1,\cdots,s_{l},\xi_{N_{s_1}},\cdots,\xi_{N_{s_{l}}})\dd N_{s_1}\cdots \dd N_{s_{l}}
                    \\
                    & = \sum_{\pi\in\eS_{l}}\int_{\R}\bigg\{\int_{-\infty}^{s_{\pi(l)}-} \cdots \bigg\{ \int_{-\infty}^{s_{\pi(2)}-} F(s_1,...,s_{l},\xi_{N_{s_1}},...,\xi_{N_{s_{l}}}) \dd N_{s_{\pi(1)}}\bigg\}\cdots \dd N_{s_{\pi(l-1)}} \bigg\} \dd N_{s_{\pi(l)}}.
                \end{align*}
                Since the sequence $\{\xi_i,i\in\mathbb{Z}\}$ is independent of $N$, the most internal integral satisfies
                \begin{align*}
                    & \int_{-\infty}^{s_{\pi(2)}-} F(s_1,\cdots,s_{l},\xi_{N_{s_1}},\cdots,\xi_{N_{s_{l}}}) \dd N_{s_{\pi(1)}} 
                    \\
                    &\qquad\qquad = \left( \int_{-\infty}^{s_{\pi(2)}-} F(s_1,...,s_{l},\xi_{N_{s_1}},x_2,...,x_{l}) \dd N_{s_{\pi(1)}} \right)_{\big| x_2 = \xi_{N_{s_2}}, ...,x_{l} = \xi_{N_{s_{l}}}}.
                \end{align*}
                We compute the less internal integrals thanks to the same argument, and the expectation given by $\eqref{eq:stoch_int}$ directly follows by conditionning.
            \end{proof}

\subsection*{Acknowledgments} This work was
 supported by the French government, through CNRS, the UCA$^{Jedi}$ and 3IA C\^ote d'Azur Investissements d'Avenir managed by the National Research Agency (ANR-15
IDEX-01 and ANR-19-P3IA-0002),  directly by the ANR project ChaMaNe (ANR-19-CE40-0024-02) and by the interdisciplinary Institute for Modeling in Neuroscience and Cognition (NeuroMod). %

\end{document}